\documentclass[a4paper,11pt]{article}

\usepackage[utf8]{inputenc}
\usepackage[british]{babel}

\usepackage{xcolor}

\usepackage[mono=false]{libertine}
\usepackage[T1]{fontenc}

\usepackage{amsthm}
\usepackage{amssymb}
\usepackage[cmintegrals,libertine]{newtxmath}
\usepackage[cal=euler, scr=boondoxo]{mathalfa}
\usepackage{booktabs}
\useosf

\usepackage{microtype}
\usepackage[shortlabels]{enumitem}  
\usepackage{numprint}
\usepackage{subcaption}

\usepackage[margin=2.5cm]{geometry}
\usepackage{graphicx}
\graphicspath{{./images/}}

\usepackage[font=small]{caption}
\usepackage{hyperref}
\hypersetup{
	colorlinks=true,
	citecolor=blue!60!black,
	linkcolor=red!60!black,
	urlcolor=green!40!black,
	filecolor=yellow!50!black,
	breaklinks=true,
	pdfpagemode=UseNone,
	bookmarksopen=false,
}

\usepackage{enumitem}

\newtheorem{lemma}{Lemma}
\newtheorem{theorem}{Theorem}

\newtheorem{proposition}{Proposition}
\newtheorem{corollary}{Corollary}
\newtheorem{remark}{Remark}

\newtheorem{definition}{Definition}

\newcommand{\R}{\mathbb{R}}
\newcommand{\Z}{\mathbb{Z}}

\newcommand{\N}{\mathbb{N}}
\newcommand{\expec}{\mathbb{E}}
\newcommand{\prob}{\mathbb{P}}
\newcommand{\probric}{\mathbf{P}}
\newcommand{\expecric}{\mathbf{E}}

\newcommand{\hdownc}{{\mathfrak{h}^\downarrow}}
\newcommand{\hupc}{{\mathfrak{h}^\uparrow}}

\newcommand{\rmd}{\mathrm{d}}

\newcommand{\qseq}{\mathbf{q}}
\newcommand{\hqseq}{{\hat{\mathbf{q}}}}

\newcommand{\map}{\mathfrak{m}}
\newcommand{\emap}{\mathfrak{e}}
\newcommand{\umap}{\mathfrak{u}}

\newcommand{\dmap}{\map^\dagger}

\newcommand{\loopmaps}{\mathcal{LM}}

\newcommand{\gasket}{\mathfrak{g}}
\newcommand{\one}{\mathbf{1}}
\newcommand{\sgn}{\mathrm{sgn}}
\newcommand{\dfpp}{d_{\mathrm{fpp}}}

\newcommand{\pr}{\mathfrak{p}}

\newcommand{\Adomain}{\mathbb{A}}
\newcommand{\Ddomain}{\mathbb{D}}
\newcommand{\maps}{\mathcal{M}}
\newcommand{\edges}{\mathsf{Edges}}
\newcommand{\vertices}{\mathsf{Vertices}}

\newcommand{\loopconf}{\boldsymbol{\ell}}
\newcommand{\aloop}{\ell}

\newcommand{\rt}{\mathrm{r}}

\newcommand{\metr}{\mathbb{X}}

\DeclareMathOperator{\arctanh}{arctanh}
\DeclareMathOperator{\arccot}{arccot}

\newcommand{\compactfrac}[2]{{\textstyle\frac{#1}{#2}}}

\title{The peeling process on random planar maps\\ coupled to an $O(n)$ loop model}
\author{Timothy Budd\thanks{IMAPP, Radboud University Nijmegen, The Netherlands.  \hfill  \href{mailto:t.budd@science.ru.nl}{\texttt{T.Budd@science.ru.nl}}}\\ (with an appendix by Linxiao Chen\thanks{University of Helsinki, Finland. \hfill  \href{mailto:linxiao.chen@helsinki.fi}{\texttt{Linxiao.Chen@helsinki.fi}}}\,\,\,)}

\begin{document}
	
	\vspace{-5mm}
	\maketitle

\vspace{-9mm}
\begin{abstract}
We extend the peeling exploration introduced in \cite{budd_peeling_2015} to the setting of Boltzmann planar maps coupled to a rigid $O(n)$ loop model.
Its law is related to a class of discrete Markov processes obtained by confining random walks to the positive integers with a new type of boundary condition.
As an application we give a rigorous justification of the phase diagram of the model presented in \cite{borot_recursive_2012}.
This entails two results pertaining to the so-called fixed-point equation: the first asserts that any solution determines a well-defined model, while the second result, contributed by Chen in the appendix, establishes   precise existence criteria.

A scaling limit for the exploration process is identified in terms of a new class of positive self-similar Markov processes, going under the name of ricocheted stable processes. 
As an application we study distances on loop-decorated maps arising from a particular first passage percolation process on the maps. In the scaling limit these distances between the boundary and a marked point are related to exponential integrals of certain L\'evy processes.
The distributions of the latter can be identified in a fairly explicit form using machinery of positive self-similar Markov processes.

Finally we observe a relation between the number of loops that surround a marked vertex in a Boltzmann loop-decorated map and the winding angle of a simple random walk on the square lattice. 
As a corollary we give a combinatorial proof of the fact that the total winding angle around the origin of a simple random walk started at $(p,p)$ and killed upon hitting $(0,0)$ normalized by $\log p$ converges in distribution to a Cauchy random variable.
\end{abstract}

\vspace{-1mm}
\begin{figure}[h]
\begin{center}
\includegraphics[height=.37\linewidth]{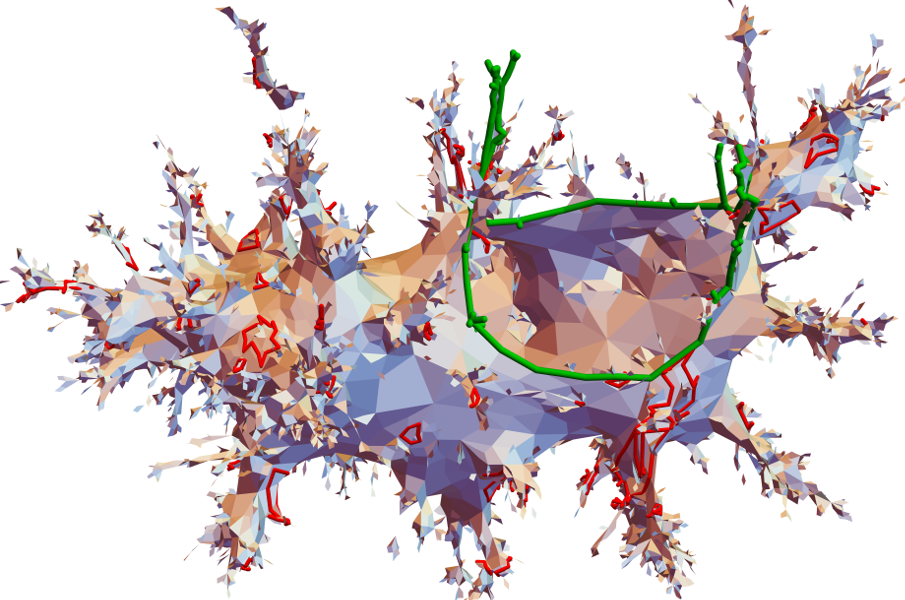}
\includegraphics[height=.37\linewidth]{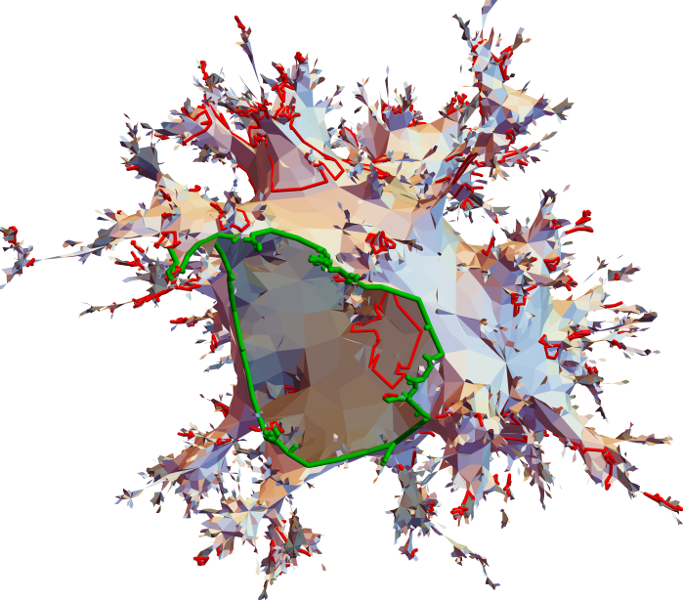}
\end{center}
\caption{Simulations of a $O(n)$ loop model (with loops in red) coupled to a Boltzmann quadrangulation of the disk (with the boundary in green). Both are sampled from the dilute critical phase with $n=0.3$ (left) and $n=0.6$ (right) and have a perimeter of order $200$.
}\label{fig:simulation}
\vspace{-4mm}
\end{figure}
\clearpage

\section{Introduction}

\subsection{Motivation}
In recent years there has been a lot of progress in understanding the geometry of large random maps on surfaces. 
For many classes of random planar maps it is by now known that the large-scale geometry, as defined through a suitable rescaling of the graph distance of the map, is described in the scaling limit by the Brownian metric on the sphere \cite{le_gall_uniqueness_2013,miermont_brownian_2013} with a Hausdorff dimension (almost surely) equal to $4$.
The same random continuous metric has been proven \cite{miller_quantum_2016,miller_axiomatic_2015} to arise from Liouville quantum gravity \cite{polyakov_quantum_1981} in its ``pure gravity'' regime corresponding to parameter $\gamma=\sqrt{8/3}$.
Different random metric spaces corresponding to Liouville quantum gravity with other values of $\gamma\in(0,2)$ are conjectured to appear as scaling limits of random planar maps decorated by critical statistical systems.
However, much less is known about the existence and properties of these metric spaces, which arguably poses one of the main open challenges in the theory of random planar maps and Liouville quantum gravity.
See \cite{gwynne_mating--trees_2017,ding_fractal_2018} for the state of the art in estimates on graph distances in some families of random planar maps arising in the mating of trees approach.

In this work we take a small step towards understanding geometric aspects of large random maps decorated by a particularly convenient statistical system, namely a version of the $O(n)$ loop model.
Random triangulations decorated by such loop models were already studied in the physics literature in the early nineties \cite{kostov_on_1989,kostov_multicritical_1992,eynard_on_1992} leading to the conjecture that they possess scaling limits described by Liouville quantum gravity with $\gamma$ ranging in $[\sqrt{2},2]$ depending on the value of $n\in[0,2]$ and the phase of the model. 
In the dense phase, where loops are believed to touch themselves and each other, the conjectural relation for $n\in[0,2)$ is $\gamma = 2\sqrt{1-b}$ where $b\coloneqq\tfrac{1}{\pi}\arccos\tfrac{n}{2}$.
In the dilute phase the loops are believed to avoid each other and themselves and one has the relation $\gamma = 2/\sqrt{1+b}$.
This phase includes undecorated maps as a special case at $n=0$ and the critical Ising model at $n=1$.
Given the wide range of potential universality classes and the fact that the corresponding enumeration problem can be solved quite explicitly \cite{eynard_exact_1995,eynard_more_1996,borot_recursive_2012,borot_more_2012}, makes the model a good candidate to investigate new geometrical scaling limits.

The precise model we investigate is the \emph{rigid} $O(n)$ loop model on bipartite planar maps that was introduced and extensively studied in \cite{borot_recursive_2012} (see Figure \ref{fig:simulation} for a simulation).
It was realized that the enumeration of the model could be solved in a recursive fashion by dissecting the maps along the loops that decorate them, an operation known as the gasket decomposition.
The remaining components are each distributed as particular random undecorated maps that support faces with a heavy-tailed degree distribution, which are precisely the \emph{non-generic} Boltzmann maps first studied in \cite{le_gall_scaling_2011}.
The universality class of these maps is determined by the exponent $\alpha$ appearing in the tail of the degree distribution, which is related to the parameters of the $O(n)$ model via $\alpha = \tfrac{3}{2}- b$ in the dense phase and $\alpha = \tfrac{3}{2}+b$ in the dilute phase.
In particular, these maps equipped with a metric arising from the graph distance were shown \cite{le_gall_scaling_2011} to possess (subsequential) scaling limits in the Gromov-Hausdorff sense to a family of \emph{stable maps} (or \emph{stable gaskets}) depending on $\alpha$, which are different from the Brownian metric and have Hausdorff dimensions equal to $2\alpha$. 
In a more recent series of papers \cite{budd_geometry_2017,bertoin_martingales_2017,budd_infinite_2017} the same maps were investigated from the point of view of distances on their duals, which support vertices of heavy-tailed degree.
The results of \cite{budd_geometry_2017} suggest that such maps may posses Gromov-Hausdorff scaling limits too (tentatively called the \emph{stable spheres}), but only for $\alpha$ in the range $(\tfrac{3}{2},2]$ corresponding to the dilute phase of the $O(n)$ loop model.
As $\alpha$ approaches $\tfrac{3}{2}$ the (tentative) Hausdorff dimension $\frac{2\alpha}{2\alpha-3}$ blows up, and one enters a regime with exponential \cite{budd_geometry_2017} or quasi-exponential \cite{budd_infinite_2017} volume growth.

The large-scale geometry of random maps decorated by $O(n)$ model loops should correspond to another family of random metric spaces, that again depends on $\alpha$ and, as discussed above, is conjecturally described by Liouville quantum gravity.
Studying general distances in such maps is a hard problem in general, especially at a combinatorial level, since shortest paths in the map may intersect loops in complicated ways. 
However, one can make the problem tractable by adapting the notion of distance used on the map.
One way to do this is to introduce a metric space with ``shortcuts'' along the loops of the map, meaning that in the new metric space the sections of a path that trace a loop on the map do not contribute to its length.
The advantage of this adapted distance is that one may study its statistics by extending the techniques of \cite{budd_geometry_2017} (which in turn generalized those of \cite{curien_scaling_2017}), and this is the goal of the current work.

The main tool used in \cite{budd_geometry_2017} to study the geometry of Boltzmann maps is the \emph{peeling process}.
It was first studied in the physics literature in the case of random triangulations \cite{watabiki_construction_1995}, leading to the first heuristic derivation of the growth dimension being equal to $4$ in the pure gravity regime \cite{ambjorn_scaling_1995}.
Once put on a rigorous footing it was recognized as a versatile tool to study many statistics of random triangulations and quadrangulations (see e.g.  \cite{angel_growth_2002,benjamini_simple_2013,richier_universal_2015,curien_scaling_2017,ambjorn_multi-point_2016,curien_first-passage_2015}).
The particular formulation of the peeling process that will form the basis of the current work was introduced in \cite{budd_peeling_2015} and has the advantage of its economical description even in the presence of faces of large degree (see also \cite{budd_geometry_2017,bertoin_martingales_2017,budd_infinite_2017} and \cite{curien_peeling_nodate} for a nice review).
A significant portion of this paper will be devoted to generalizing this peeling exploration, its probabilistic description and scaling limit, to the setting where loops are present on the map.

\subsection{Main results}

\subsubsection{The $O(n)$ loop model and its phase diagram}\label{sec:Onintro}

The central combinatorial objects under investigation are maps decorated by loop configurations.
The maps $\map$ we consider are bipartite planar maps that are rooted by distinguishing an oriented \emph{root edge}.
The face to the right of the root edge is called the \emph{root face}, while all other face are called \emph{internal} faces.
The \emph{perimeter} of the map is the degree of its root face.
A \emph{loop configuration} on $\map$ is a collection $\loopconf = \{\aloop_1,\ldots_,\aloop_k\}$ of disjoint (unoriented) loops on the dual map that avoid the root face.
We say $\loopconf$ is \emph{rigid} if the loops only visit quadrangles and they enter and exit the quadrangles through opposite sides (see Figure \ref{fig:loopmap}).
A pair $(\map,\loopconf)$ consisting of a map $\map$ and a rigid loop configuration $\loopconf$ on $\map$ is called a \emph{loop-decorated map}.
The set of all such loop-decorated maps with a fixed perimeter $2p$ is denoted by $\loopmaps^{(p)}$.

\begin{figure}[h]
	\centering
	\vspace{-4mm}
	\includegraphics[width=.6\linewidth]{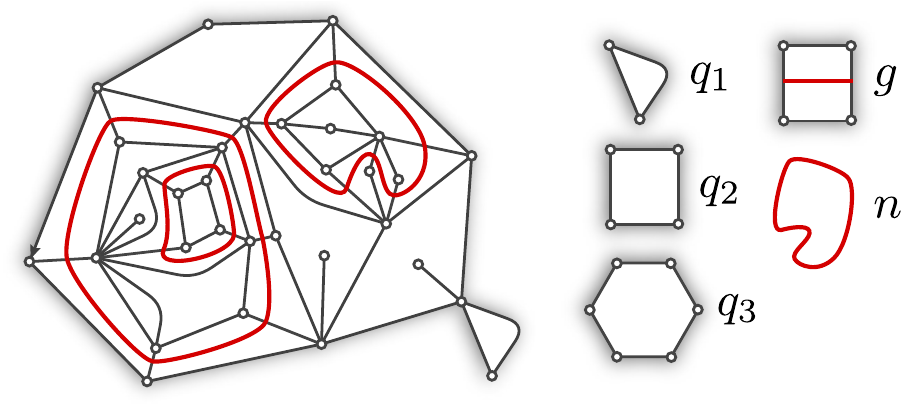}
	
	\vspace{-4mm}
	\caption{A planar map $\map$ of perimeter $10$ with a (rigid) loop configuration $\loopconf = \{\ell_1,\ell_2,\ell_3\}$. 
	The weight of this loop-decorated map is $w_{\qseq,g,n}(\map,\loopconf) = n^3\, g^{16}\, q_1^2\, q_2^7\,q_3^2$.
		\label{fig:loopmap}}
\end{figure}

Given a sequence $\qseq = (q_1,q_2,\ldots)$ of non-negative real numbers and $g,n\geq 0$, we define the \emph{weight} $w_{\qseq,g,n}(\map,\loopconf)$ of a loop-decorated map to be 
\begin{equation}\label{eq:loopweight}
w_{\qseq,g,n}(\map,\loopconf) := \Bigg[ \prod_{\ell\in \loopconf} n\,g^{|\ell|}\Bigg]\Bigg[\prod_{f} q_{\deg(f)/2}\Bigg],
\end{equation}
where $|\ell|$ is the length of the loop $\ell$, the second product is over all internal faces $f$ that are not visited by a loop, and $\deg(f)$ is the degree of face $f$.
The \emph{$O(n)$ model partition function} $F^{(p)}(\qseq,g,n)$ is given by the total weight of all loop-decorated maps of fixed perimeter $2p$, 
\begin{equation*}
F^{(p)}(\qseq,g,n) := \sum_{(\map,\loopconf)\in\loopmaps^{(p)}} w_{\qseq,g,n}(\map,\loopconf).
\end{equation*}
The triple $(\qseq,g,n)$ is said to be \emph{admissible} iff $F^{(p)}(\qseq,g,n)<\infty$ for all $p\geq 1$.
In this case the weights $w_{\qseq,g,n}(\map,\loopconf)$ give rise to a probability measure on $\loopmaps^{(p)}$ when normalized by $1/F^{(p)}(\qseq,g,n)$, the \emph{$(\qseq,g,n)$-Boltzmann loop-decorated map of perimeter $2p$}.
If $g=0$ or $n=0$ then loops are suppressed and the random map is the usual $\qseq$-Boltzmann map. 
In this case we denote the partition function by $W^{(p)}(\qseq) \coloneqq F^{(p)}(\qseq,0,0)$, and say $\qseq$ is admissible iff $W^{(p)}(\qseq) < \infty$ for all $p\geq 1$.

It was realized in \cite{borot_recursive_2012} that Boltzmann loop-decorated maps are conveniently studied via their \emph{gasket}, which is the portion of the loop-decorated map that is exterior to all loops.
To be precise, we define the gasket $\gasket(\map,\loopconf)$ to be the (undecorated) map obtained from $\map$ by removing all edges intersected by a loop in $\loopconf$ and retaining only the connected component containing the root (see Figure \ref{fig:gasket}). 
In \cite{borot_recursive_2012} it was shown that if $(\qseq,g,n)$ is admissible then so is $\hat{\qseq}$ defined by the \emph{fixed-point equation}
\begin{equation}\label{eq:effq}
\hat{q}_k = q_k + n \, g^{2k} F^{(k)}(\qseq,g,n),
\end{equation}
and that the gasket $\gasket(\map,\loopconf)$ of a $(\qseq,g,n)$-Boltzmann loop-decorated map itself is distributed as a $\hat{\qseq}$-Boltzmann map.
The first result we state was left implicit in \cite{borot_recursive_2012} and shows that the converse of the admissibility statement is true as well.

\begin{figure}[h]
	\centering
	\includegraphics[width=.45\linewidth]{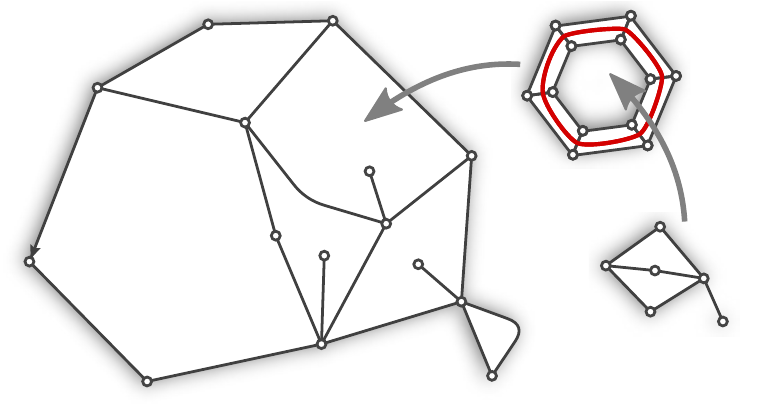}
	\caption{The gasket $\gasket(\map,\loopconf)$ of the loop-decorated map $(\map,\loopconf)$ of Figure \ref{fig:loopmap} is shown on the left. The loop-decorated map can be reconstructed by iteratively inserting a loop surrounding a new undecorated map into a (undecorated) face.
		\label{fig:gasket}}
\end{figure}

\begin{theorem}[Equivalence of admissibility] \label{thm:admissibility}
Suppose $n\in[0,2]$. The triple $(\qseq,g,n)$ is admissible iff there exists an admissible sequence $\hat{\qseq}$ such that 
\begin{equation*}
q_k = \hat{q}_k - n \, g^{2k} W^{(k)}(\hat{\qseq}).
\end{equation*}
In this case $F^{(p)}(\qseq,g,n)=W^{(p)}(\hqseq)$ and the expected number of vertices in a $(\qseq,g,n)$-Boltzmann loop-decorated map is finite.
\end{theorem}  

For any admissible triple $(\qseq,g,n)$ and associated sequence $\hqseq$ the limit
\begin{equation*}
\gamma^2_\hqseq = \lim_{p\to\infty} \frac{W^{(p+1)}(\hqseq)}{W^{(p)}(\qseq)}
\end{equation*}
exists.
We say a triple $(\qseq,g,n)$ is \emph{non-generic critical} if it admissible, $n \neq 0$, and $\gamma^2_\hqseq g = 1$.

To discuss the precise phase diagram, we fix a positive integer $d$ and limit ourselves to triples $(\qseq,g,n)$ in the domain 
\begin{equation*}
\Ddomain \coloneqq [0,\infty)^d \times (0,\infty) \times (0,2),
\end{equation*}
where $\qseq\in [0,\infty)^d$ means that $q_k = 0$ for $k > d$. 
In \cite{borot_recursive_2012} four phases of the $O(n)$ model were identified.
They showed\footnote{Actually they only considered the case of quadrangulations $q_k = q_2 \one_{\{k=2\}}$ but none of their results relies crucially on this choice.} that if $(\qseq,g,n)\in \Ddomain$ is admissible, then 
\begin{equation}\label{eq:Fasymp}
F^{(p)}(\qseq,g,n) \sim C\,\gamma_\hqseq^{2p}\,p^{-(\alpha+1/2)}\qquad\text{as }p\to\infty,
\end{equation}
where $C>0$ depends on $(\qseq,g,n)$.
Moreover, the exponent $\alpha$ can take only the values $1$, $2$, $\frac32+b$ and $\frac32-b$, where $b = \frac1\pi \arccos(\frac n2) \in (0,\frac12)$. The triple $(\qseq,g,n)$ is said to be \emph{subcritical}, \emph{generic critical}, \emph{non-generic critical dense} and \emph{non-generic critical dilute} respectively in the four cases.
Moreover, in the last two cases $(\qseq,g,n)$ is also non-generic critical in the sense above.

However, the techniques in \cite{borot_recursive_2012} are not quite sufficient to rigorously establish the admissibility and the phase of particular values of $(\qseq,g,n)$. 
By combining Theorem \ref{thm:admissibility} with the work in \cite{borot_recursive_2012}, this gap is closed completely by a result of Linxiao Chen included in Appendix \ref{sec:phasediagram}.

\begin{theorem}[Linxiao Chen]\label{thm:phasediagram}
There exist explicit real-valued functions $\mathfrak{f}$, $\mathfrak{g}$, $\mathfrak{h}$ such that a triple $(\qseq,g,n)\in\Ddomain$ is admissible and
\begin{enumerate}
	\item subcritical iff \ $\mathfrak{h}(\qseq,g,n,\gamma)=1$ and \ $\mathfrak{f}(\qseq,g,n,\gamma)>0$ for some $\gamma \in (0,g^{-1/2})$.
	\item generic critical iff \ $\mathfrak{h}(\qseq,g,n,\gamma)=1$ and \ $\mathfrak{f}(\qseq,g,n,\gamma)=0$ for some $\gamma \in (0,g^{-1/2})$.
	\item non-generic critical and dense iff \ $\mathfrak{h}(\qseq,g,n,g^{-1/2})=1$ and \ $\mathfrak{g}(\qseq,g,n,g^{-1/2})>0$.
	\item non-generic critical and dilute iff \ $\mathfrak{h}(\qseq,g,n,g^{-1/2})=1$ and \ $\mathfrak{g}(\qseq,g,n,g^{-1/2})=0$.
\end{enumerate}
Moreover, all four phases are non-empty and form a partition of the set of admissible triples in $\Ddomain$.
\end{theorem}  

\subsubsection{The geometry of large loop-decorated maps}\label{sec:introgeom} 

Let $\loopmaps_\bullet^{(p)}$ be the set of loop-decorated maps with a distinguished vertex, called the \emph{target}.
If $(\qseq,g,n)$ is admissible, then we may introduce the \emph{pointed} $(\qseq,g,n)$-Boltzmann loop-decorated map $(\map_\bullet,\loopconf)$ in $\loopmaps_\bullet^{(p)}$ with probability distribution proportional to $w_{\qseq,g,n}(\map_\bullet,\loopconf)$.
This is well-defined precisely because a $(\qseq,g,n)$-Boltzmann loop-decorated map has a finite expected number of vertices (Theorem \ref{thm:admissibility}).

In order to study geometric properties of such maps, we introduce in Section \ref{sec:looppeeling} a systematic exploration process on $(\map_\bullet,\loopconf)$. 
This so-called (targeted) peeling process iteratively builds a sequence of growing submaps of $(\map_\bullet,\loopconf)$ representing the explored neighbourhoods of the root at each stage of the exploration, stopping once the target is encountered.
Roughly speaking, at each step of the exploration one selects (using any desired \emph{peel algorithm}) an edge in the boundary of the explored region and one discovers what sits on the other side (as determined by the map $(\map_\bullet,\loopconf)$ being explored). 
If a new face or a loop is discovered, then it is added to the explored region together with the contents of any unexplored region that gets detached from the unexplored region containing the target. 
When applied to a pointed Boltzmann loop-decorated map the law of this peeling exploration is conveniently described (see Proposition \ref{thm:perimlaw}) in terms of the \emph{perimeter process} $(P_i)$, which is a Markov process on $\Z$ that keeps track of the half-length of the boundary of the explored region.

For $b\in (0,1/2)$ fixed, we introduce the L\'evy process $(\xi^\downarrow_s)_{s\geq 0}$ started at $\xi^\downarrow_0=0$ with Laplace exponent
\begin{equation*}
\Psi^\downarrow(z) \coloneqq \log \expec\, e^{z\,\xi^\downarrow_1} = \frac{1}{\pi} \Gamma( 1+2b-z)\Gamma(1-b+z)\left[\cos(\pi b) - \cos(\pi z-\pi b)\right].  
\end{equation*}
It has no killing and drifts to $-\infty$ almost surely.
By a classic result of Lamperti \cite{lamperti_semi-stable_1972} a time-changed exponential of such a L\'evy process determines a \emph{positive self-similar Markov process} (pssMp) that dies continuously at zero.
In particular, the pssMp $(X^\downarrow_t)_{t\geq 0}$ associated to the \emph{Lamperti representation} $(\xi^\downarrow_s)_{s\geq 0}$ is determined by
\begin{equation*}
X_t^\downarrow = \exp(\xi^\downarrow_{s(t)}), \quad s(t) \coloneqq \inf\left\{s\geq 0: \int_0^s \exp((1+b) \xi^\downarrow_u)\rmd u \geq t\right\} \in \R \cup \{\infty\}, 
\end{equation*}
with the convention that $\exp \xi^\downarrow_\infty = 0$.
It is one of a class of pssMps that can be constructed from stable processes confined to the  half-line by a new type of boundary condition, that we refer to as \emph{ricocheted stable processes}.

In Section \ref{sec:scalinglimitwalk} we prove that the perimeter process of a non-generic critical Boltzmann loop-decorated map possesses a scaling limit described by such a ricocheted stable process.

\begin{theorem}\label{thm:perimeterscaling}
Let $(\qseq,g,n)\in\Ddomain$ be in the non-generic critical phase (dense or dilute) and let $(P_i)_{i\geq 0}$ be the perimeter process of a pointed $(\qseq,g,n)$-Boltzmann loop-decorated map with perimeter $2p$. Then there exists a $c>0$ such that with growing perimeter $p$ we have the convergence in distribution
\begin{equation*}
\left( \frac{P_{\lfloor c p^{\theta}\, t\rfloor}}{p}\right)_{t\geq 0} \xrightarrow[p\to\infty]{\mathrm{(d)}} (X^\downarrow_t)_{t\geq 0}
\end{equation*}
in the Skorokhod topology, where $(X^\downarrow_t)_{t\geq 0}$ is the pssMp defined above with parameter $b = |\theta - 1|$ and $\theta=\alpha-\tfrac12$ as in \eqref{eq:Fasymp}.
\end{theorem}

In \cite{budd_geometry_2017} in the setting of Boltzmann planar maps (without loops) it was shown that the exploration process, with a suitable choice of peel algorithm, can be used to study metric properties of large Boltzmann planar maps.
We generalize some of these results to the setting of the $O(n)$ loop model.
As an explicit example we consider a type of first passage percolation on the dual $\map^\dagger$ of a loop-decorated map $(\map,\loopconf)$.
To be precise, we assign to each edge $e$ of $\map^\dagger$ a weight $x_e$ that is set to $0$ if $e$ is covered by a loop (the red edges in Figure \ref{fig:fpppath}), while all other weights are taken to be independent exponential random variables of mean $1$.
This way one may assign to any pair of faces $f,f'$ of $\map$ a distance 
\begin{equation}\label{eq:fppd}
d_{\mathrm{fpp}}(f,f') = \inf_{\gamma} \sum_{e\in\gamma} x_e,
\end{equation}
where the infimum is over all paths $\gamma$ in $\map^\dagger$ from $f$ to $f'$.
Notice that, roughly speaking, in the resulting metric space all loops are contracted to points.
In the case of a pointed loop-decorated map $(\map_\bullet,\loopconf)$ we let $d_{\mathrm{fpp}}(\map_\bullet,\loopconf)$ be the minimal fpp distance between the root face and any of the faces incident to the marked vertex (Figure \ref{fig:fpppath}).

\begin{figure}[h]
	\centering
	\includegraphics[width=.85\linewidth]{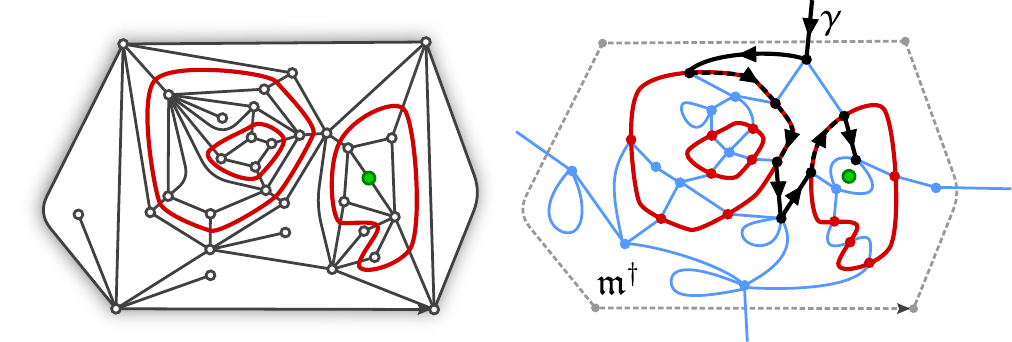}
	\caption{A pointed loop-decorated map (left) together with its dual (right). A possible shortest path $\gamma$ is shown whose fpp-length would determine $d_{\mathrm{fpp}}(\map_\bullet,\loopconf)$. Notice that the dashed edges of the path (the ones traversing the loops) do not contribute to the fpp-length.
		\label{fig:fpppath}}
\end{figure}

\begin{theorem}\label{thm:fpp}
Let $(\qseq,g,n)\in \Ddomain$ be in the non-generic critical dilute phase and let $(\map_\bullet,\loopconf)$ be a pointed $(\qseq,g,n)$-Boltzmann loop-decorated map of perimeter $2p$. Then there exists a $c>0$, such that we have the convergence in distribution
\begin{equation*}
\frac{d_{\mathrm{fpp}}(\map_\bullet,\loopconf)}{c\,p^b} \xrightarrow[p\to\infty]{\mathrm{(d)}} R^\downarrow \coloneqq \int_0^\infty e^{b\xi^\downarrow_s}\rmd s
\end{equation*} 
to an exponential integral of the L\'evy process $\xi^\downarrow$ with parameter $b = \frac{1}{\pi} \arccos(\frac{n}{2})$.
\end{theorem}

It is worth noting that an identical result (differing from Theorem \ref{thm:fpp} only in the value of $c$) can be obtained in the case that the first passage percolation distance is replaced by a (dual) graph distance, in the sense that we assign weight $x_e=1$ instead of a random weight to each dual edge $e$ that is not covered by a loop (but still $x_e=0$ for each covered edge).
Proving Theorem \ref{thm:fpp} for the graph distance requires a number of additional estimates that are not very illuminating and are largely analogous to those of \cite[Section 3.2]{budd_geometry_2017}, so we choose to omit the details.

One may as well interpret Theorem \ref{thm:fpp} (respectively its graph distance analogue) as a non-trivial lower bound on the first passage percolation distance (respectively graph distance) on $\map^\dagger$ without the shortcuts, showing that in the dilute phase the distance to the marked vertex scales at least like $p^b$ as the perimeter $2p$ increases.\footnote{Combined with the fact that the number of vertices in the map is of order $p^2$ in the dilute phase (see Proposition \ref{thm:expecvertices}), this heuristically leads to an upper bound of $2/b$ on the Hausdorff dimension of the tentative Gromov-Hausdorff scaling limit. In terms of the LQG parameter $\gamma$, this comes down to an upper bound of $2\gamma^2/(4-\gamma^2)$ for $\gamma\in[\sqrt{8/3},2)$. Notice that, although sharp at $\gamma=\sqrt{8/3}$, this bound is a lot weaker than the best upper bound known for so-called Mated-CRT-maps \cite[Theorem 1.2 \& 1.6]{ding_fractal_2018}, which is $\sqrt{6}\,\gamma$ in the same range.}

Using machinery on exponential integrals of L\'evy processes (Section \ref{sec:expintgr}) an explicit Mellin transform of $R^\downarrow$ can be determined in terms of special functions (Barnes double gamma).
With the help of these one may deduce the tails of its distribution.

\begin{proposition}\label{thm:rdownasymp}
For $b=\frac{1}{\pi} \arccos(\frac{n}{2})\in(0,1/2)$ the distribution function of $R^\downarrow$ satisfies the asymptotic relations
\begin{align*}
\prob( R^\downarrow < r )&\stackrel{r\to 0}{\sim} b^{1/b} (b+1) \cot\left(\frac{\pi b}{2}\right) r^{1/b}, \\
\prob( R^\downarrow > r )&\stackrel{r\to \infty}{\sim} \frac{\Gamma(-b)^2}{\pi b} \cot\left(\frac{\pi b}{2}\right) \, r^{-2}.
\end{align*}
\end{proposition} 

In the special case\footnote{These values $n=\cos\frac{\pi}{m}$, $m\geq 2$, are special in the context of conformal field theory in the sense that the dilute $O(n)$ loop model is expected in the continuum to be related to the $(m+1,m)$-minimal models in conformal field theory (see e.g. \cite{di_francesco_modular_1987}).} that $n= 2\cos\frac{\pi}{m}$ for $m=3,4,\ldots$ one can perform the inverse Mellin transform explicitly and obtain a relatively simple integral expression for the distribution of $R^\downarrow$.

\begin{proposition}\label{thm:specRdist}
If $b=1/m$ for $m=3,4,\ldots$ then 
\begin{equation*}
\prob( R^\downarrow < r) = \frac{m}{\Gamma(1+\tfrac{1}{m})} \int_{0}^\infty \rmd x\,B_m\!\left(\frac{m}{r x}\right)\,x^m\,e^{-x^m}, \qquad B_m(y)\coloneqq  \frac{1 + y\cot\left(\frac{\pi}{2m}\right)}{\prod_{k=0}^m (1-i\,y\, e^{i\pi k/m})}.
\end{equation*}
\end{proposition}
\begin{figure}[t]
	\centering
	\includegraphics[width=.55\linewidth]{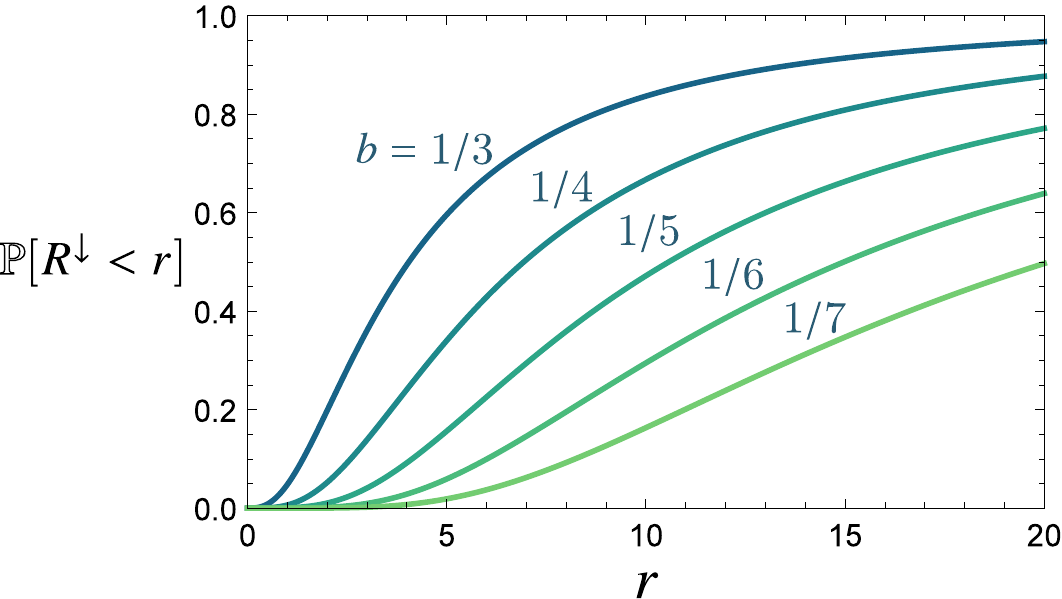}
	\caption{The distribution function of $R^\downarrow$ plotted for various values of $b$.\label{fig:rplot}}
\end{figure}

\subsubsection{Nesting and winding statistics}

The number $N$ of loops that surround the marked vertex in a pointed loop-decorated map is called the \emph{nesting statistic} (for example $N=1$ in the map of Figure \ref{fig:fpppath}).
Its distribution in the case of the $O(n)$ loop model on triangulations of large fixed size was extensively studied in \cite{borot_nesting_2016} (see also \cite{borot_nesting_2016-1} for an extension to arbitrary topology and \cite[Appendix A]{chen_perimeter_2017} for a discussion in terms of the branching tree of nested loops in the case of quadrangulations). 
The large deviations of the nesting were shown to match a similar statistic in a Conformal Loop Ensemble (CLE$_\kappa$) coupled to Liouville Quantum Gravity \cite{borot_nesting_2016}.

The targeted peeling exploration gives a natural tool to determine nesting, since it corresponds to the number of times in the exploration process that one has to cross a loop to discover the marked vertex.
As a byproduct of our analysis of the peeling exploration, we obtain an alternative derivation of a (much less general version of) the nesting distribution \cite[Theorem 2.2]{borot_nesting_2016} valid in the setting of non-generic critical $(\qseq,g,n)$-Boltzmann loop-decorated maps.
To state the results we need to introduce the functions $h_\pr^\downarrow : \Z_{\geq0} \to \R$ for $\pr\in[-1,1]$ by setting $h_\pr^\downarrow(0)=1$ while for $p\geq 1$,
\begin{align}
\begin{split}
h^\downarrow_\pr(p) &= \frac{1}{1+\pr} \frac{\Gamma(b)}{\Gamma(2p+1)\Gamma(b-2p)} \,{_2F_1}(-2p,b-1;b-2p;-1)\quad \text{for }\pr \in (-1,1),\\
h^\downarrow_1(p) &= 1, \qquad\qquad h^\downarrow_{-1}(p) = \frac{4}{\pi^2 p} \sum_{k=1}^p \frac{1}{2k-1},
\end{split}\label{eq:hdowndef}
\end{align} 
where $b \coloneqq \tfrac{1}{\pi}\arccos\pr\in(0,1)$ and ${_2F_1}$ is the hypergeometric function.

\begin{theorem}\label{thm:nesting}
	Let $n\in(0,2]$, $(\qseq,g,n)$ non-generic critical, and $(\map_\bullet,\loopconf)$ a pointed $(\qseq,g,n)$-Boltzmann loop-decorated map of perimeter $2p$. Then the number $N$ of loops surrounding the marked vertex has probability generating function
	\begin{equation*}
	\expec_\bullet^{(p)}[x^N] = \frac{h^\downarrow_{xn/2}(p)}{h^\downarrow_{n/2}(p)}, \qquad x\in[-2/n,2/n].
	\end{equation*}
	If $n=2$ it satisfies the convergence in distribution
	\begin{equation*}
	\frac{N}{\log^2 p} \xrightarrow{\mathrm{(d)}} \frac{1}{\pi^2}\mathcal{L},
	\end{equation*}
	where $\mathcal{L}$ is a L\'evy random variable with density $e^{-1/(2x)}/\sqrt{2\pi x^3}\,\rmd x$.
	If $n<2$ it satisfies as $p\to\infty$ the convergence in probability 
	\begin{equation*}
	\frac{N}{\log p} \xrightarrow{\mathrm{(p)}} \frac{n}{\pi\sqrt{4-n^2}}
	\end{equation*}
	and the large deviation property
	\begin{align*}
	\frac{\log\prob_\bullet^{(p)}\left[ N < \lambda\log p\right]}{\log p}& \xrightarrow{k\to\infty } -\frac{1}{\pi} J_{n/2}(\pi \lambda) \quad\text{for}\quad 0< \lambda < \tfrac{n}{\pi\sqrt{4-n^2}}, \\
	\frac{\log\prob_\bullet^{(p)}\left[ N > \lambda\log p\right]}{\log p}& \xrightarrow{k\to\infty } -\frac{1}{\pi} J_{n/2}(\pi \lambda) \quad\text{for}\quad \lambda >\tfrac{n}{\pi\sqrt{4-n^2}},
	\end{align*} 
	where 
	\begin{equation*}
	J_\pr(x) = x \log \left( \frac{x}{\pr\sqrt{1+x^2}}\right) + \arccot x - \arccos \pr.
	\end{equation*}
\end{theorem}

A surprising byproduct of our analysis is that the functions $h_\pr^\downarrow$ also appears in the probability distribution of the winding of a simple random walk on $\Z^2$.

\begin{figure}[th]
	\centering
	\includegraphics[width=.35\linewidth]{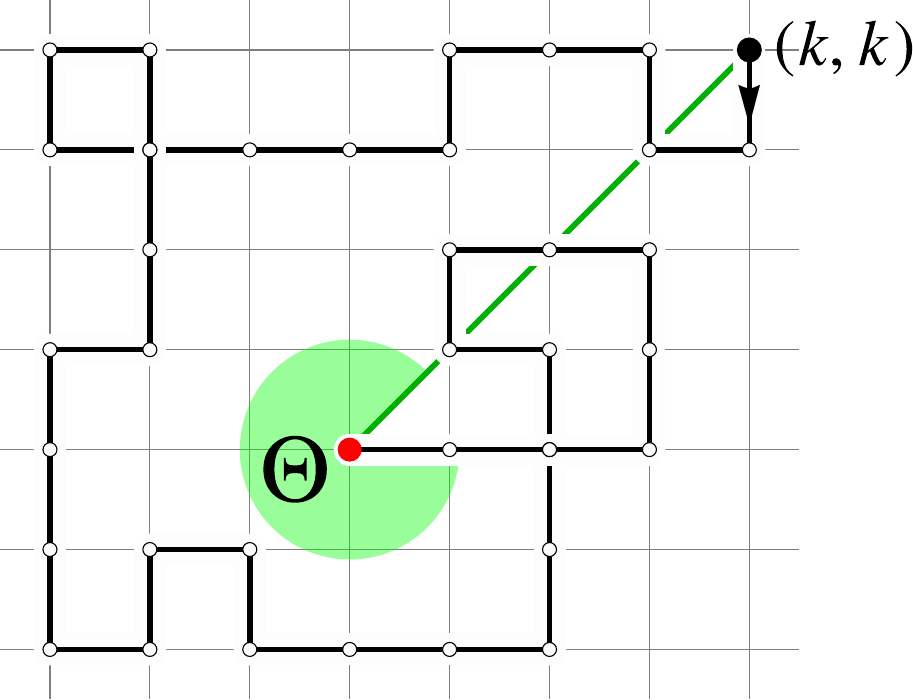}
	\caption{The winding angle $\Theta$ (in this case $7\pi/4$) around the origin of a simple random walk started at $(k,k)$ and killed upon hitting $(0,0)$. \label{fig:windingangle}}
\end{figure}

\begin{theorem}\label{thm:winding}
	Consider a simple random walk on $\Z^2$ started on the diagonal at $(k,k)$ for $k\geq 1$ and killed upon hitting the origin. If $\Theta$ is the total winding angle of the walk around the origin and $[\Theta]\in\pi(\Z+\tfrac{1}{2})$ its value rounded to the nearest half-integer multiple of $\pi$, then its characteristic function is
	\begin{equation*}
	\expec\left[ e^{i \beta\, [\Theta] }\right] = \cos\left(\tfrac{\pi\beta}{2}\right) h^\downarrow_{\cos(\pi \beta)}(k). 
	\end{equation*} 
	As a consequence we have the convergence in distribution
	\begin{equation}\label{eq:windingcauchy}
	\frac{\Theta}{\log k} \xrightarrow[k\to\infty]{(\mathrm{d})} \mathcal{C},
	\end{equation}
	where $\mathcal{C}$ is a Cauchy random variable with density $\frac{1}{\pi}(1+x^2)^{-1}\rmd x$.
\end{theorem} 

One way to understand the similarity between the nesting statistics of pointed Boltzmann maps and the winding statistics of walks on $\Z^2$ is to consider a very particular non-generic critical triple $(\qseq,g,n)$. 
To be precise, let
\begin{equation*}
q_k = (3\pi)^{1-k} \frac{2k}{(k-\tfrac{3}{2})_4}+\one_{\{k=1\}},\quad g = \frac{1}{3\pi}, \quad n = 2 
\end{equation*}
where $(a)_m = \Gamma(a+m)/\Gamma(a)$ is the Pochhammer symbol.
It can be shown (see Proposition \ref{thm:alternation} and Remark \ref{rem:specialparams}) that then the pointed $(\qseq,g,n)$-Boltzmann loop-decorated map $(\map_\bullet,\loopconf)$ of perimeter $2p$ can be coupled with the simple random walk $(X_i)$ on $\Z^2$ started at $(p,p)$ in such a way that the $j$th visit of $(X_i)$ to the diagonal occurs at $\pm(P_j,P_j)$ where $(P_i)_{i\geq0}$ is the perimeter process of $(\map_\bullet,\loopconf)$.
Moreover each loop in $\loopconf$ that surrounds the marked vertex corresponds to a sign change in the visits to the diagonal.
It is not hard to see that the number of such sign changes is closely related to the winding angle $\Theta$, since each sign change contributes $\pm\pi$ to the winding angle.
A bijective explanation of the relation between walks on $\Z^2$ and (loop-decorated) maps will appear in a forthcoming work.

Combinatorial results on winding angles of simple walks on $\Z^2$ similar to those in Theorem \ref{thm:winding} can be found in \cite{budd_winding_2017}. 
In fact, we could have used the machinery of \cite{budd_winding_2017} to derive the explicit expression for $h^\downarrow_\pr$ instead of computing it on the planar map side (Proposition \ref{thm:hdown}), but we have chosen to keep the exposition self-contained.
The asymptotic distribution \eqref{eq:windingcauchy} of the winding angle is easily seen to match that of a 2d Brownian motion killed upon hitting a small disk around the origin.
A strong approximation by Brownian motion could potentially provide an alternative proof of \eqref{eq:windingcauchy}, e.g. along the lines of \cite[Section 5]{shi_windings_1998}.

\subsection{Outline}

This work deals with the peeling exploration of $(\qseq,g,n)$-Boltzmann loop-decorated maps for general admissible triples $(\qseq,g,n)$, but we start in a setting where we do not yet know which triples are admissible.
In fact, we will rely on the peeling exploration to establish the latter as announced in Theorem \ref{thm:admissibility}.
This means some care has to be taken in the ordering of the presented material.

In Section \ref{sec:peeling} we recall the peeling exploration \cite{budd_peeling_2015} of a fixed bipartite planar map and extend the concept to the exploration of fixed loop-decorated maps.
In Section \ref{sec:peelboltz} we determine the law of the peeling exploration applied to the pointed $(\qseq,g,n)$-Boltzmann loop-decorated map under the assumptions that $(\qseq,g,n)$ is admissible (and that the $(\qseq,g,n)$-Boltzmann loop-decorated maps has a finite expected number of vertices).
In particular, we express the transition probabilities of the perimeter (and nesting) process in terms of the gasket weights $\hqseq$ (and $g,n$).
In Section \ref{sec:rw} we mimic these transition probabilities (in the non-generic critical case) by introducing a Markov process on $\Z^2$, the ricocheted random walk, which is valid for any admissible sequence $\hqseq$, and derive some statistics that are necessary in the following.
Along the way (in Section \ref{sec:winding}) we observe that this ricocheted random walk also appears in the winding angle problem of the simple random walk on $\Z^2$.
Section \ref{sec:proofadmiss} then proves Theorem \ref{thm:admissibility} concerning the admissibility of $(\qseq,g,n)$ by showing that one can algorithmically construct a $(\qseq,g,n)$-Boltzmann planar map via a peeling exploration.
To show that the algorithm is well-defined we rely on statistics of the ricocheted random walk determined in Section \ref{sec:rw}.
As a consequence we observe that the ricochet statistics of Section \ref{sec:winding} apply to give the nesting statistics of loops on pointed Boltzmann loop-decorated maps.

Section \ref{sec:ricocheted-stable-processes} introduces the ricocheted stable processes that for particular parameters are proven to appear as the distributional scaling limits of the perimeter process of non-generic Boltzmann planar maps.
Section \ref{sec:the-asymptotics-of-the-fpp-distance} proves a scaling limit for the first-passage percolation distance in terms of the ricocheted stable process, while \ref{sec:expintgr} describes the machinery to explicitly determine this distribution.

Finally the appendix contains the proof of Theorem \ref{thm:phasediagram} contributed by Linxiao Chen.

\subsection*{Acknowledgements} 
We thank Ga\"etan Borot, J\'er\'emie Bouttier, Nicolas Curien, and Bertrand Duplantier for discussions and useful suggestions.
Part of this work was done while the author was at the Niels Bohr Institute, University of Copenhagen in Denmark, and the Institut de Physique Th\'eorique, CEA, Universit\'e Paris-Saclay in France.
This work was supported by a public grant as part of the
Investissement d'avenir project, reference ANR-11-LABX-0056-LMH,
LabEx LMH.
The author of the appendix thanks the hopsitality of the Laboratoire de Math\'ematique d'Orsay of Universit\'e Paris-Saclay in France, as well as the support from the project ANR-14-CE25-0014 (ANR GRAAL) and the ERC Advanced Grant 741487 (QFPROBA).

\section{The peeling process}\label{sec:peeling}

A \emph{(rooted planar) map} is a multigraph, i.e. a graph in which loops and multiple edges are allowed, that is properly embedded in the sphere and for which an oriented edge, the \emph{root edge}, has been distinguished. 
As usual we identify any two maps that are related by an orientation-preserving homeomorphism of the sphere that also preserves the root edge.  
The face to the right of the root edge is called the \emph{root face} of the map and its degree, i.e. the number of edges incident to it, the \emph{perimeter} of the map.
Often we will think of the contour of the root face as the boundary of a tessellation of the $n$-gon corresponding to the map of perimeter $n$, but it is important to keep in mind that the root face of a map need not be \emph{simple}, meaning that different corners of root face may share the same vertex (as for instance is the case in Figure \ref{fig:loopmap}).
For our purposes we restrict to \emph{bipartite} maps, meaning that all faces have even degree.

\subsection{Peeling process on undecorated maps}\label{sec:undecoratedpeeling}

Let us recall the peeling process of a map that was introduced in \cite{budd_peeling_2015} using the formulation appearing in \cite{budd_geometry_2017,curien_peeling_nodate}.
Since we are going to describe an exploration process of a map, we first we need a suitable notion of a submap that summarizes our knowledge of the map $\map$ at an intermediate stage of the exploration.
To this end we introduce a \emph{map with holes} to be a map together with a distinguished set of faces, called \emph{holes}, that are assumed to be simple and pairwise disjoint. 
The latter properties are equivalent to the requirement that each vertex of the map is incident to at most one corner of a hole.
By convention the root face cannot be a hole.
Throughout we will assume that we have fixed some deterministic procedure that takes a map with holes and assigns to each hole a distinguished edge incident to that hole.\footnote{The choice of such a procedure will be irrelevant in the following, but one could for instance choose the first edge incident to a particular hole as encountered in a breadth-first exploration of the map started at the root edge.}

Given a map with holes $\emap$ with a hole $h$ of degree $l$ and a second map (with holes) $\umap$ of perimeter $l$, there is a natural operation of \emph{gluing $\umap$ into the hole $h$} leading to a new map with holes $\emap'$. 
This map $\emap'$ is constructed from $\emap$ by pairwise identification of the edges in the contour of $h$ with those in the contour of the root face of $\umap$, in such a way that the distinguished edge incident to $h$ is identified with the root edge of $\umap$. 
Even though the root face of $\umap$ is not necessarily simple, this procedure is well-defined because the hole $h$ is assumed to be simple. 
We let the \emph{hollow map} of perimeter $2p$ be the map whose only faces are the root face and a single hole both of degree $2p$ (see the top-left of Figure \ref{fig:pointedpeeling}).

We say that a map $\emap$ with holes $h_1,\ldots,h_k$ is a \emph{submap} of another map with holes $\emap'$, denoted $\emap\subset\emap'$, if there exists a sequence of maps (with holes) $\umap_1,\ldots,\umap_k$ such that $\emap'$ is the result of gluing $\umap_i$ into hole $h_i$ for each $i=1,\ldots,k$ (see Figure \ref{fig:map-holes} for an example).
Notice in particular that $\emap\subset\emap$ by taking the $\umap_i$ to be hollow maps.
It is not hard to see that the gluing procedure is \emph{rigid}, in the sense that $\emap$ and $\emap'$ uniquely determine the sequence of maps $\umap_i$.
Moreover, $\subset$ defines a partial order on the set of all maps with holes. 

\begin{figure}[h]
	\centering
	\includegraphics[width=.8\linewidth]{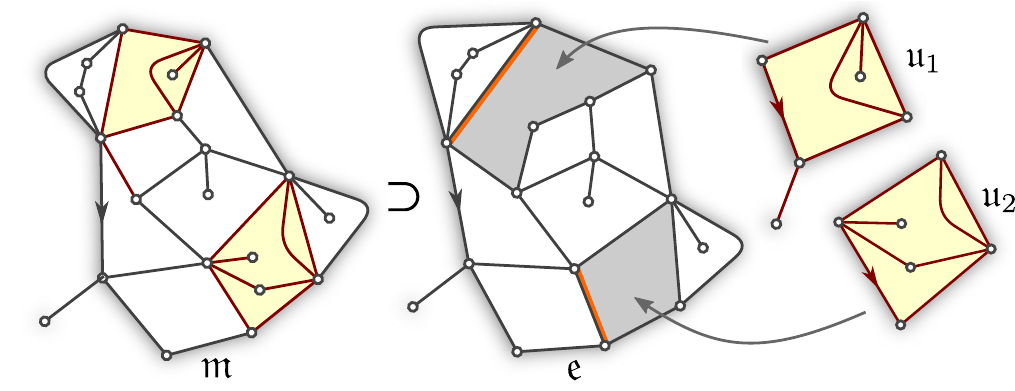}
	\caption{A map $\map$ together with a submap $\emap$ and the maps $\umap_i$ to be glued in the holes of $\emap$ to retrieve $\map$.
		\label{fig:map-holes}}
\end{figure}

For a map with holes $\emap$ we consider the partition $\mathsf{Edges}(\emap)=\mathsf{Active}(\emap)\cup\mathsf{Inactive}(\emap)$ of its edges into \emph{active} and \emph{inactive} edges depending on whether they are incident to a hole or not.
Notice that every inactive edge of a submap $\emap\subset \map$ corresponds to a unique edge in $\map$, while a pair of active edges incident to a hole $h_i$ can be identified upon gluing $\umap_i$ into $h_i$ if $\umap_i$ has an edge that is incident to the root face on both sides (as is the case $\umap_1$ in Figure \ref{fig:map-holes}).

Let $\emap\subset\map$ be a submap of a map $\map$ without holes and $e\in\mathsf{Active}(\emap)$ an active edge.
Then there exists a unique smallest (in the partial order $\subset$) submap $\emap'$ such that $\emap\subset\emap'\subset\map$ and in which $e$ has become inactive, $e\in\mathsf{Inactive}(\emap')$.
We denote this map by $\mathsf{Peel}(\emap,e,\map)$ and call it the result of \emph{peeling the edge }$e$.
By iterating this procedure we can define a growing sequence of submaps of $\map$. 
To define such a sequence uniquely, we assume a \emph{peeling algorithm} $\mathcal{A}$ is given that associates to any map with holes $\emap$ an active edge $\mathcal{A}(\emap)\in\mathsf{Active}(\emap)$.
Then we define the \emph{peeling exploration} of a map $\map$ with algorithm $\mathcal{A}$ to be the sequence
\begin{equation}\label{eq:peelingexploration}
\emap_0\subset \emap_1 \subset \cdots \subset \emap_n = \map,\quad \text{where}\quad \emap_{i+1} = \mathsf{Peel}(\emap_i,\mathcal{A}(\emap_i),\map)\quad \text{for}\quad i=0,\ldots,n-1,
\end{equation}
and $\emap_0$ is the hollow map of the same perimeter as $\map$.

Let us examine the peeling operations that can occur when peeling an edge $e \in \mathsf{Active}(\emap)$ incident to the hole $h$.
We distinguish two types depending on whether or not there is another active edge $e'$ incident to $h$ that gets identified with $e$ in the full map $\map$. 
If so, $\mathsf{Peel}(\emap,e,\map)$ is the map with holes obtained from $\emap$ by gluing $e$ to $e'$.
If $2k_1$ respectively $2k_2$ are the (necessarily even) number of active edges incident to $h$ in between $e$ and $e'$ to the left respectively right of $e$ (as seen from outside the hole), then we denote this event by $\mathsf{G}_{k_1,k_2}$ (see the bottom-left of Figure \ref{fig:peeling} for an example).
The number of holes may decrease by one (if $k_1=k_2=0$), remain unchanged (if either $k_1=0$ or $k_2=0$), or increase by one (if $k_1,k_2>0$). 
In the other situation, $e$ corresponds to an edge in $\map$ that is incident to a face of degree, say, $2k$ that does not occur in $\emap$. 
Then $\mathsf{Peel}(\emap,e,\map)$ is the map with holes obtained from $\emap$ by gluing a new $2k$-gon to the edge $e$ inside the hole $h$.
This event is denoted $\mathsf{C}_{k}$ (bottom-center of Figure \ref{fig:peeling}). 
In this case the number of holes remains unchanged.
Notice that in both cases the number of inactive edges increases exactly by one. 
Since $\emap_0$ has no inactive edges while all edges of $\emap_n = \map$ are inactive, the number of steps in the peeling exploration \eqref{eq:peelingexploration} is exactly the number of edges in the map $\map$.
For this reason this definition of the peeling exploration is sometimes called \emph{edge peeling}.

\subsection{Targeted peeling}\label{sec:targetedpeeling}

We will often consider \emph{pointed maps} which are maps with a marked vertex. 
When exploring a pointed map it is convenient to consider a version of the peeling exploration in which at each stage there is exactly one hole, which represents the unexplored region in the map containing the marked vertex. 

Let $\map_\bullet$ be a pointed map and $\emap\subset\map_\bullet$ a submap with holes $h_1,\ldots,h_k$. 
Recall that there exists a unique sequence of maps $\umap_1,\ldots,\umap_k$ to be glued in the holes of $\emap$ to obtain $\map_\bullet$.
We defined the \emph{filled-in} submap $\mathsf{Fill}(\emap,\map_\bullet)\subset\map_\bullet$ to be the map obtained from $\emap$ by gluing $u_i$ into $h_i$ if $u_i$ does not contain the marked vertex for each $i=1,\ldots,k$ (Figure \ref{fig:fillin}).
Since the holes of $\emap$ are disjoint, the filled-in submap will have at most one hole.

\begin{figure}[b]
	\centering
	\includegraphics[width=.8\linewidth]{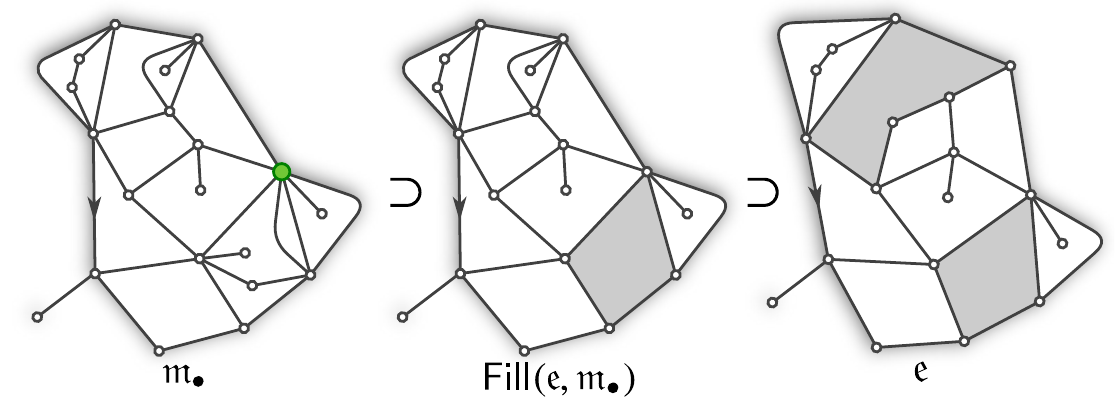}
	\caption{Illustration of the filled-in submap $\mathsf{Fill}(\emap,\map_\bullet)\subset\map_\bullet$.
		\label{fig:fillin}}
\end{figure}

The \emph{targeted peeling exploration} of $\map_\bullet$ can then be defined by iteratively peeling an edge followed by filling in the unpointed holes, i.e.
\begin{equation}\label{eq:targetedpeelingexploration}
\emap_0\subset \emap_1 \subset \cdots \subset \emap_n = \map_\bullet,\quad \text{where}\quad \emap_{i+1} = \mathsf{Fill}(\mathsf{Peel}(\emap_i,\mathcal{A}(\emap_i),\map_\bullet),\map_\bullet)\quad \text{for}\quad i=0,\ldots,n-1.
\end{equation}
Since for each $i=0,\ldots,n-1$ the submap $\emap_i$ has exactly one hole, we may introduce the \emph{perimeter process} $(P_i)_{i=0}^n$ by setting $P_i$ equal to half the degree of the hole of $\emap_i$ and $P_n=0$ by convention. 
Notice, that $P_0$ is half the perimeter of $\map_\bullet$ and that the perimeter process depends only on $\map_\bullet$ and the chosen peeling algorithm $\mathcal{A}$.

Recall from the last subsection that two types of events can occur upon peeling $\mathcal{A}(\emap_i)$ in $\emap_i$: in the event $\mathsf{G}_{k_1,k_2}$ two active edges are glued, while in the event $\mathsf{C}_k$ a $2k$-gon is inserted into the hole.
Only in the first case filling in of one of the holes by some map $\umap$ of perimeter $2k$ is necessary.
Depending on whether the left hole ($k=k_1$) or right hole ($k=k_2$) is filled in, we denote this event by $\mathsf{G}_{k,\ast}$ or $\mathsf{G}_{\ast,k}$.

\subsection{Loop-decorated maps}\label{sec:looppeeling}

We wish to extend our peeling exploration to maps carrying a loop configuration.
Recall that a loop-decorated map $(\map,\loopconf)$ is a (bipartite) map $\map$ together with a loop configuration $\loopconf = \{\ell_1,\ldots,\ell_k\}$ of disjoint unoriented simple closed paths on the dual map $\map^\dagger$.
The loops $\ell_i$ are required to avoid the root face of $\map$ and to be rigid, in the sense that they only visit quadrangles and they enter and exit the quadrangles through opposite sides.
We define a \emph{loop-decorated map with holes} $(\emap,\loopconf)$ to be a map $\emap$ with holes together with a loop configuration $\loopconf$, such that the loops avoid the holes (see the top-right of Figure \ref{fig:peeling}).
As before we put a partial order $\subset$ on the set of all loop-decorated map with holes, by letting $(\emap,\loopconf)\subset(\emap',\loopconf')$ iff $\emap\subset\emap'$ and $\loopconf$ contains all the loops of $\loopconf'$ that intersect a face of $\emap'$ that also appears in $\emap$. 
This partial order allows us to introduce the \emph{peeling exploration of a loop-decorated map} $(\map,\loopconf)$ completely analogously to \eqref{eq:peelingexploration},
\begin{equation}\label{eq:looppeelingexploration}
(\emap_0,\loopconf_0)\subset (\emap_1,\loopconf_1) \subset \cdots \subset (\emap_n,\loopconf_n) = (\map,\loopconf),
\end{equation} 
by setting $(\emap_{i+1},\loopconf_{i+1}) = \mathsf{Peel}((\emap_i,\loopconf_i),\mathcal{A}(\emap_i),(\map,\loopconf))$ to be the minimal larger submap containing the peel edge $\mathcal{A}(\emap_i)$ as an inactive edge.

\begin{figure}[h]
	\centering
	\includegraphics[width=.9\linewidth]{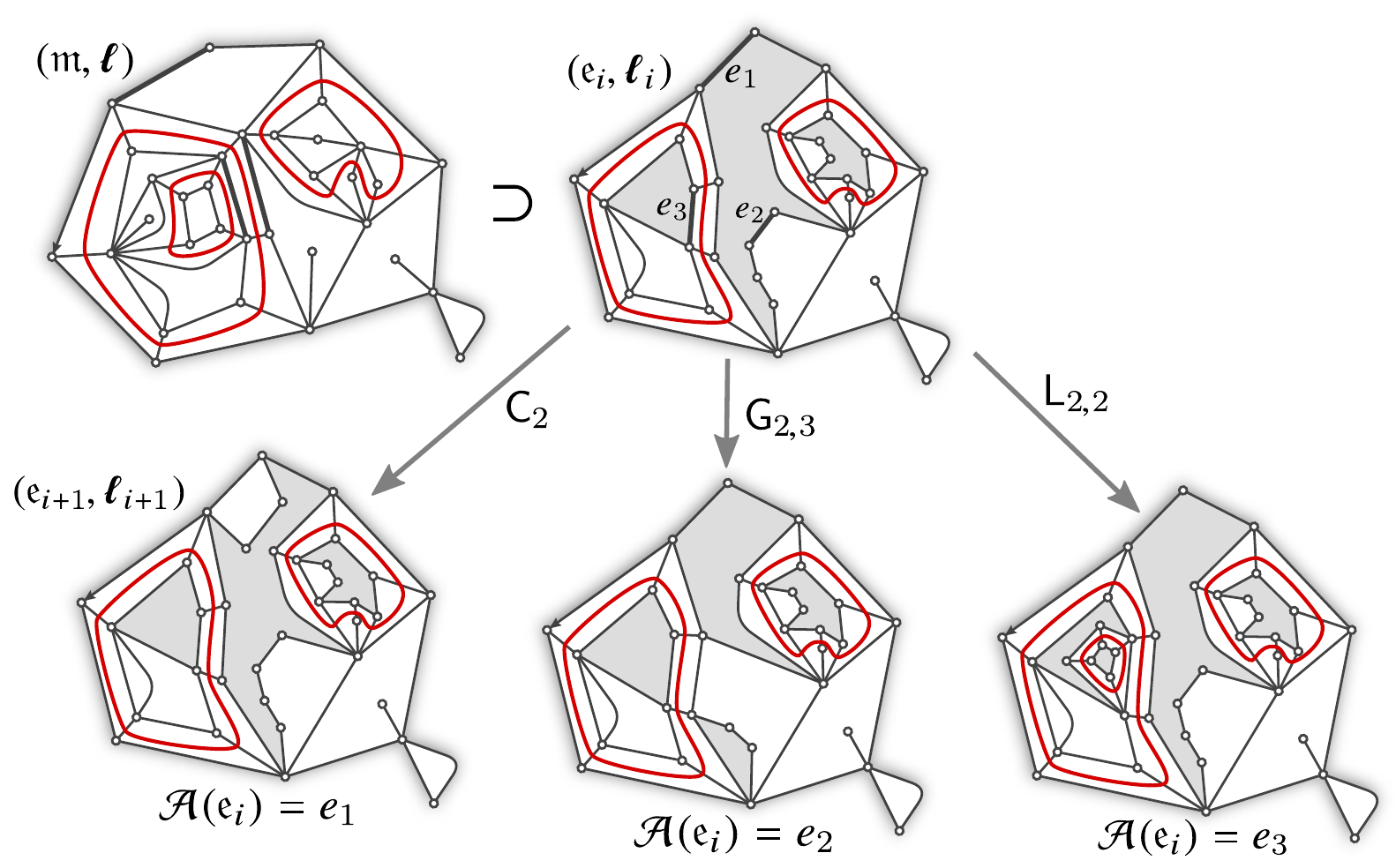}
	\caption{ A possible state $(\emap_i,\loopconf_i)$ of the peeling exploration of the loop-decorated map $(\map,\loopconf)$. 
		Three possible choices $e_1,e_2,e_3$ of the peel edge $\mathcal{A}(\emap_i)$ are indicated and the corresponding edges in $\map$ are highlighted.
		Depending on $\mathcal{A}(\emap_i)$ three different types of events can occur. \label{fig:peeling}}
\end{figure}

As before we may analyze the peeling operations that can occur when peeling an edge $e\in\mathsf{Active}(\emap_i)$ in a loop-decorated map with holes $(\emap_i,\loopconf)$. 
The same events $\mathsf{G}_{k_1,k_2}$ and $\mathsf{C}_k$ can occur, if $e$ is incident in $\map$ to respectively a face already in $\emap_i$ or to a new face of degree $2k$ that is not intersected by a loop of $\loopconf$.
We have to take into account a third type of event, denoted $\mathsf{L}_{k,k}$, in which the $e$ is incident to a new quadrangle intersected by a loop $\ell\in\loopconf$ of length $2k$ (see the bottom-right of Figure \ref{fig:peeling}).
In that case we are not allowed to just glue the quadrangle inside the hole, since the resulting map would not be a submap of $(\map,\loopconf)$.
Any submap containing $e$ as an inactive edge should contain the full loop $\ell$.
Hence, the minimal such submap $(\emap_{i+1},\loopconf_{i+1})$ is obtained from $(\emap_i,\loopconf_i)$ by gluing a ring of length $2k$ inside the hole to the edge $e$.
Here a \emph{ring of length} $2k$ is a loop-decorated map of perimeter $2k$ with a single hole of degree $2k$ constructed by gluing a strip of $2k$ quadrangles into a ring and covering it by a single loop of length $2k$ (Figure \ref{fig:ring}).
This increases the number of holes by one.

\begin{figure}[h]
	\centering
	\includegraphics[width=.17\linewidth]{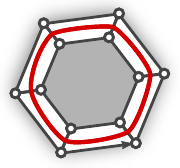}
	\caption{ A ring of length 6. \label{fig:ring}}
\end{figure}

When the peel algorithm $\mathcal{A}$ and the perimeter of $\map$ are fixed, the loop-decorated map $(\map,\loopconf)$ is completely determined by the finite sequence of events, e.g. $(\mathsf{C}_2,\mathsf{G}_{3,2},\mathsf{C}_1,\mathsf{L}_{2,2},\ldots)$.
We will rely on this fact in the proof of Theorem \ref{thm:admissibility} in Section \ref{sec:proofadmiss}.

If the loop-decorated map $(\map_\bullet,\loopconf)$ is equipped with a marked vertex, we may also consider the targeted peeling exploration of $(\map_\bullet,\loopconf)$ as in Section \ref{sec:targetedpeeling},
\begin{equation}\label{eq:targetlooppeeling}
(\emap_0,\loopconf_0)\subset (\emap_1,\loopconf_1) \subset \cdots \subset (\emap_n,\loopconf_n) = (\map_\bullet,\loopconf).
\end{equation} 
As in the undecorated case we may have the events $\mathsf{C}_k$, $\mathsf{G}_{k,\ast}$, and $\mathsf{G}_{\ast,k}$, but also $\mathsf{L}_{k,\ast}$ and $\mathsf{L}_{\ast,k}$ corresponding to the gluing of a ring (event $\mathsf{L}_{k,k}$) followed by filling in the hole inside respectively outside the ring (see Figure \ref{fig:pointedpeeling}).

\begin{figure}[h]
	\centering
	\includegraphics[width=.98\linewidth]{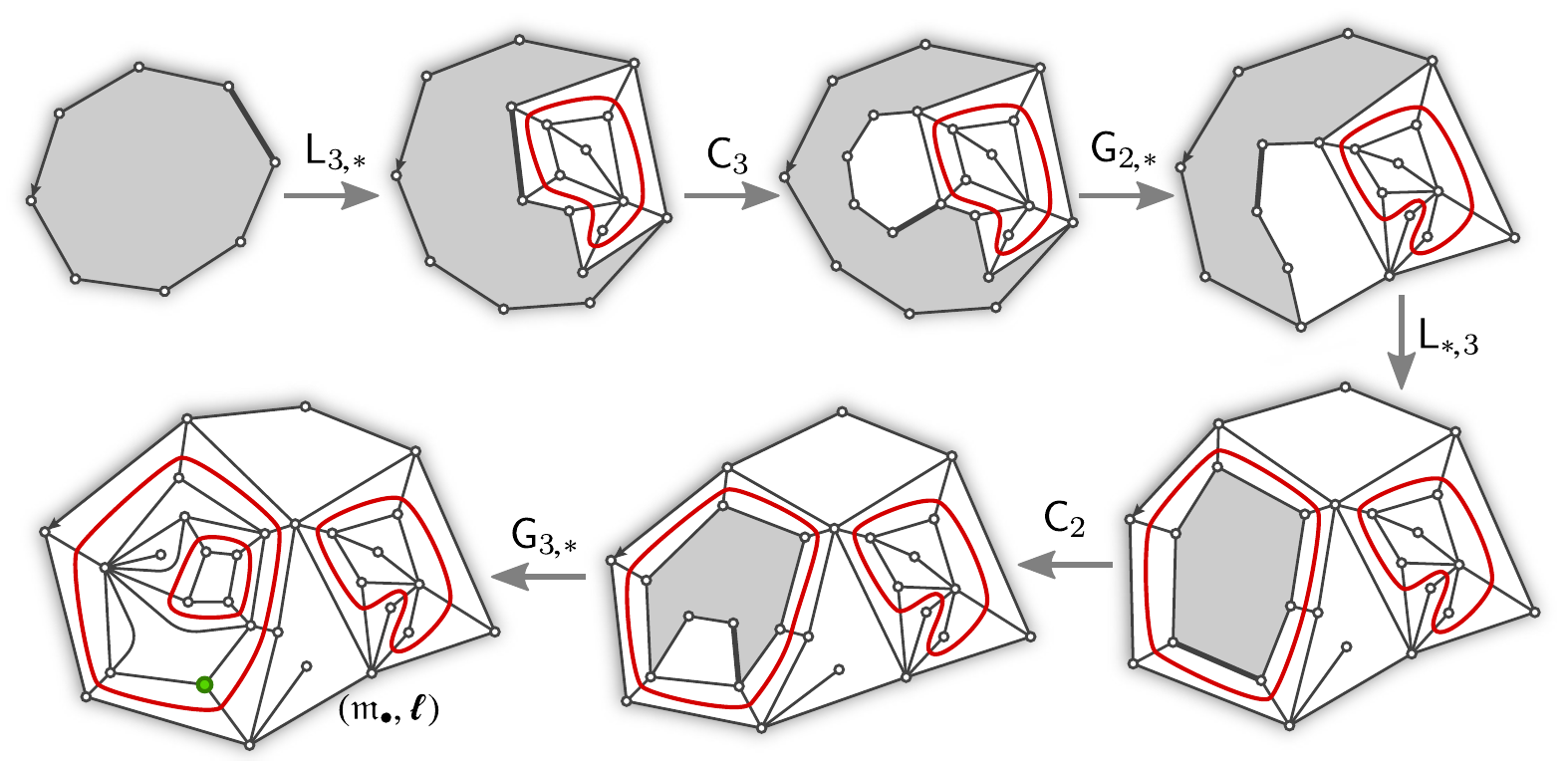}
	\caption{ The targeted peeling process of the pointed loop-decorated map $(\map_\bullet,\loopconf)$. The peel edges are indicated by thick lines. The corresponding perimeter process is $(P_i)_{i=0}^6 = (4,6,8,5,3,4,0)$, whereas the nesting process is $(N_i)_{i=0}^6 = (0,0,0,0,1,1,1)$. \label{fig:pointedpeeling}}
\end{figure}

In addition to the perimeter process, it is natural to keep track of the \emph{nesting process} $(N_i)_{i}$ such that $N_0=0$ and $N_i$ increases, $N_{i+1}=N_i+1$, only in the event $\mathsf{L}_{\ast,k}$.
Then $N_n$ is precisely the \emph{nesting statistic} of $(\map_\bullet,\loopconf)$, i.e. the number of loops in $\loopconf$ that wind around the marked vertex as seen from the root face.

\subsection{First passage percolation}\label{sec:first-passage-percolation}

As advertised in the introduction we will apply the peeling exploration to study a type of first passage percolation on a loop-decorated map $(\map,\loopconf)$.
Recall from Section \ref{sec:introgeom} that we do this by assigning weights $x_e$ to the dual edges $e\in\edges(\dmap)$ of the map $\map$.
This gives rise to the metric \eqref{eq:fppd} on the set of faces of $\map$.
In order to relate balls of growing radius in this metric to a peeling exploration of $(\map,\loopconf)$ we need to ensure that balls cannot contain partial loops.
The only way to enforce this is to assign weight $x_e=0$ to each dual edge that is covered by a loop, such that all faces intersected by a particular loop are at identical distance to the any other face of $\map$.
The dual edges $e\in\edges(\dmap)$ that are not covered by a loop are assigned i.i.d. exponential random weights with mean one.

It is convenient to associate to the map $(\map,\loopconf)$ with weights $x_e$ a continuous length metric space $(\metr,d_{\mathrm{fpp}})$ by considering the graph underlying $\dmap$ as a topological space and viewing each edge $e$ as a real interval of length $x_e$ with appropriate identifications at the endpoints (see also \cite[Section 2.4]{budd_geometry_2017}).
Notice that in this construction all points on the loop edges of a single loop are identified in $\metr$ (the red curves in Figure \ref{fig:fppball}).
We say a dual edge $e$ is \emph{within distance $t$ from the root} if this is true for all its points in $\metr$, i.e. they have distance less than or equal to $t$ from the point $O\in\metr$ associated to root face of $\map$.
Following \cite[Section 1.4]{budd_geometry_2017} we then define the submap $\mathrm{Ball}_t(\map,\loopconf)$ of $\map$ to be obtained from $\map$ by keeping all faces at distance within $t$ from the root face, while ``ungluing'' those faces along dual edges that are not within distance $t$ from the root in the above sense (see Figure \ref{fig:fppball}).
Since any loop in $\loopconf$ is either completely inside or outside $\mathrm{Ball}_t(\map,\loopconf)$, we can comfortably view $\mathrm{Ball}_t(\map,\loopconf)$ as a submap of $(\map,\loopconf)$ containing only the loops of $\loopconf$ that are inside.
If the map $\map_\bullet$ is equipped with a target, i.e. a marked vertex, then we can consider the \emph{hull} $\overline{\mathrm{Ball}}_t(\map_\bullet,\loopconf) \coloneqq \mathsf{Fill}( \mathrm{Ball}_t(\map_\bullet,\loopconf),(\map_\bullet,\loopconf))$ by filling in all holes not containing the target.

\begin{figure}[h]
	\centering
	\includegraphics[width=.85\linewidth]{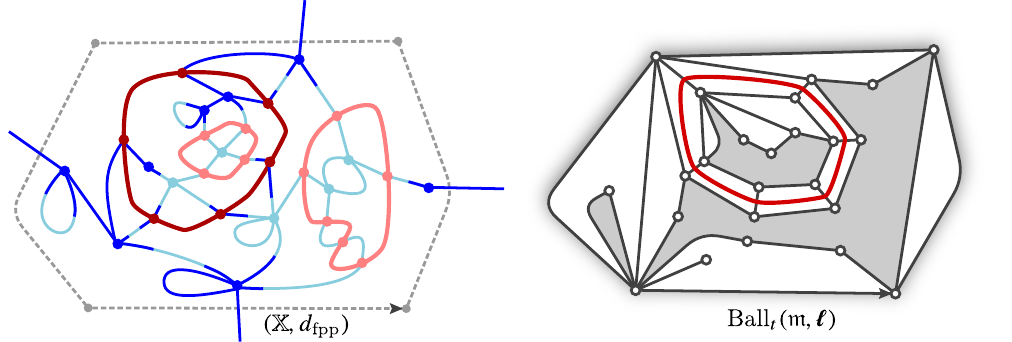}
	\caption{Illustration of the length metric space $(\metr,\dfpp)$ of the loop-decorated map  $(\map,\loopconf)$ of Figure \ref{fig:fpppath} with the darker shading corresponding to points within a distance $t$ from the root face (whose node in $\map^\dagger$ is not drawn here). 
	The corresponding submap $\mathrm{Ball}_t(\map,\loopconf)\subset(\map,\loopconf)$ is shown on the right. \label{fig:fppball}}
\end{figure}

Let $(t_i)_{i= 0}^n$ be the increasing sequence of times $t$ at which $\overline{\mathrm{Ball}}_t(\map_\bullet,\loopconf)$ changes such that $t_0=0$ and $\overline{\mathrm{Ball}}_{t_n}(\map_\bullet,\loopconf)=(\map_\bullet,\loopconf)$. 
We also introduce the \emph{uniform peeling process} on $(\map_\bullet,\loopconf)$ to be the targeted peeling exploration of \eqref{eq:targetlooppeeling} with peel algorithm $\mathcal{A}$ given by a uniform sampling (independent of everything else) of an edge $\mathcal{A}(\emap)$ among the active edges $\mathsf{Active}(\emap)$ of a loop-decorated map $\emap$ with holes. 
Then we have the following direct analogue of \cite[Proposition 1]{budd_geometry_2017}.

\begin{proposition}\label{thm:fpppeeling}
If $(\map_\bullet,\loopconf)\in\loopmaps_\bullet^{(p)}$ is a fixed loop-decorated map with a marked vertex, then the law of $(\overline{\mathrm{Ball}}_{t_i}(\map_\bullet,\loopconf))_{i=0}^n$ is equal to that of the explored maps $(\emap_i)_{i=0}^n$ in the uniform peeling process on $(\map_\bullet,\loopconf)$.
Conditionally on $(\overline{\mathrm{Ball}}_{t_i}(\map_\bullet,\loopconf))_{i\geq0}$ the time differences $\Delta t_i\coloneqq t_{i+1}-t_i$ for $i\geq 0$ are independent and $\Delta t_i$ is distributed as an exponential variable with mean $1/(2P_i)$.
\end{proposition}

Let $\mathcal{C}_\bullet\subset \metr$ be the closed path consisting of all points belonging to dual edges adjacent to the marked vertex.
We denote by $\hat{d}_{\mathrm{fpp}}(\map_\bullet,\loopconf) = \sup_{x\in\mathcal{C}_{\bullet}} d_{\mathrm{fpp}}(O,x)$ the maximal distance of the points on this cycle to the root $O\in\metr$. 
It is not hard to see that $\hat{d}_{\mathrm{fpp}}(\map_\bullet,\loopconf) = \inf\{ t: \overline{\mathrm{Ball}}_{t}(\map_\bullet,\loopconf) = (\map_\bullet,\loopconf)\} = t_n$.
Proposition \ref{thm:fpppeeling} therefore implies the identity in law
\begin{equation}\label{eq:fppperim}
\hat{d}_{\mathrm{fpp}}(\map_\bullet,\loopconf) \overset{(d)}{=} \sum_{i=0}^\infty \frac{\mathbf{e}_i}{2P_i} \one_{\{P_i > 0\}},
\end{equation}
where $(P_i)_{i\geq 0}$ is the perimeter process of the uniform peeling on $(\map_\bullet,\loopconf)$ and $(\mathbf{e}_i)_{i\geq 0}$ are independent exponential random variables with mean one.

Notice that the definition of $\hat{d}_{\mathrm{fpp}}(\map_\bullet,\loopconf)$ differs from that of $d_{\mathrm{fpp}}(\map_\bullet,\loopconf)$ given in Section \ref{sec:introgeom} and Theorem \ref{thm:fpp}.
The latter is given by the minimal distance $d_{\mathrm{fpp}}(\map_\bullet,\loopconf) = \inf_{x\in\mathcal{C}_\bullet} d_{\mathrm{fpp}}(O,x)$ to the cycle $\mathcal{C}_\bullet$ instead of the maximal distance. 
However, they do not differ much since
\begin{equation}\label{eq:dfppdifference}
 0\leq \hat{d}_{\mathrm{fpp}}(\map_\bullet,\loopconf)-d_{\mathrm{fpp}}(\map_\bullet,\loopconf) \leq \sum_{e} x_e,
\end{equation} 
where the sum is over the edges $e$ adjacent to the marked vertex.

\section{The peeling process for Boltzmann loop-decorated maps}\label{sec:peelboltz}

In this section we examine the law of the peeling exploration when applied to a Boltzmann loop-decorated map.
We will describe both the untargeted exploration as well as the targeted exploration of a pointed Boltzmann loop-decorated map.
But first we summarize the gasket decomposition result of \cite{borot_recursive_2012}.

\subsection{Gasket decomposition}\label{sec:gasket}
Recall from Section \ref{sec:Onintro} that the gasket $\gasket (\map,\loopconf)$ of a loop-decorated map $(\map,\loopconf)\in\loopmaps^{(p)}$ is obtained from $\map$ by removing all edges intersected by a loop and retaining only the connected component containing the root.
Let us summarize the argument in \cite{borot_recursive_2012} showing that the gasket of Boltzmann loop-decorated map is itself a Boltzmann map.

Suppose $\map'$ is an (undecorated) map of perimeter $2p$.
Given an admissible triple $(\qseq,g,n)$, it is not too hard to determine the total weight of the set $\gasket^{-1}(\map')\subset \loopmaps^{(p)}$ of loop-decorated maps $(\map,\loopconf)$ that have $\map'$ as their gasket.
Indeed, each internal face $f$ of $\map'$ of degree $2k$ either corresponds to a face of $\map$, or to the contour of a loop of length $2k$ in $\loopconf$ (see e.g. Figure \ref{fig:gasket}).
In the latter case the partition function $F^{(k)}(\qseq,g,n)$ gives the total weight of all possible configuration in the interior of the ring.
Therefore
\begin{equation*}
\sum_{(\map,\loopconf)\in\gasket ^{-1}(\map')} w_{\qseq,g,n}(\map,\loopconf) = \prod_f \left(q_{\frac{1}{2}\deg(f)} + n \,g^{\deg(f)} F^{(\frac{1}{2}\deg(f))}(\qseq,g,n)\right),
\end{equation*}
where the product is over all internal faces of $\map'$.
In terms of the \emph{effective weight sequence} $\hat\qseq$ of \eqref{eq:effq}
this is equal to
\begin{equation*}
\prod_f \hat{q}_{\frac{1}{2}\deg(f)} = w_{\hat{\qseq}}(\map').
\end{equation*}
The partition function $F^{(p)}(\qseq,g,n)$ therefore satisfies the identity
\begin{equation}\label{eq:diskid}
F^{(p)}(\qseq,g,n)= \sum_{(\map,\loopconf)\in\loopmaps^{(p)}} w_{\qseq,g,n}(\map,\loopconf)= \sum_{\map'\in\maps^{(p)}} w_{\hat\qseq}(\map') = W^{(p)}(\hat{\qseq}),
\end{equation}
which implies in particular that $W^{(p)}(\hat{\qseq})<\infty$.
Hence $\hat{\qseq}$ is admissible.

\subsection{Untargeted exploration}

Let $(\qseq,g,n)$ be an admissible triple and let $(\map,\loopconf)$ be a $(\qseq,g,n)$-Boltzmann loop-decorated map of perimeter $2p > 0$.
It satisfies the following Markov property, which is a direct generalization of \cite[Proposition 6]{budd_peeling_2015}.

\begin{lemma}[Markov property]\label{thm:markov}
Let $(\emap,\loopconf')$ be a fixed loop-decorated map of perimeter $2p$ with holes $h_1,\ldots,h_k$. If $(\emap,\loopconf') \subset (\map,\loopconf)$ with positive probability, then conditionally on $(\emap,\loopconf') \subset (\map,\loopconf)$ the loop-decorated maps $\umap_1,\ldots,\umap_k$ filling in the holes $h_1,\ldots,h_k$ are distributed as independent $(\qseq,g,n)$-Boltzmann loop-decorated maps of perimeters $\deg(h_1),\ldots,\deg(h_k)$.
\end{lemma}
\begin{proof}
The rigidity of the gluing operation described in Section \ref{sec:undecoratedpeeling} implies that the set of loop-decorated maps $(\map,\loopconf)$ such that $(\emap,\loopconf') \subset (\map,\loopconf)$ is in bijection with tuples $(\umap_1,\ldots,\umap_k)$ of loop-decorated maps of perimeters $\deg(h_1),\ldots,\deg(h_k)$.
Due to the product structure \eqref{eq:loopweight}, there exists a constant $C\geq 0$ depending only on $(\emap,\loopconf')$ such that 
\begin{equation*}
w_{\qseq,g,n}(\map,\loopconf) = C \prod_{i=1}^k w_{\qseq,g,n}(\umap).
\end{equation*}
If $C>0$ then the right-hand side normalizes to a probability distribution on tuples $(\umap_1,\ldots,\umap_k)$ that agrees with that of independent $(\qseq,g,n)$-Boltzmann loop-decorated maps.
Notice finally that $C>0$ precisely when $(\emap,\loopconf') \subset (\map,\loopconf)$ with positive probability.
\end{proof}

In the following we assume a peel algorithm $\mathcal{A}$ is fixed, which may be deterministic or probabilistic. 
In the latter case we require that conditionally on $(\emap,\loopconf')\subset(\map,\loopconf)$
the choice of peel edge $\mathcal{A}(\emap,\loopconf')\in\mathsf{Active}(\emap)$ is independent of $(\map,\loopconf)$ and thus only depends on the structure of $(\emap,\loopconf')$. 
Let 
\begin{equation*}
(\emap_0,\loopconf_0)\subset (\emap_1,\loopconf_1) \subset \cdots \subset (\emap_n,\loopconf_n) = (\map,\loopconf)
\end{equation*} 
be the associated peeling exploration.
Conditionally on $(\emap_i,\loopconf_i)$ and $e=\mathcal{A}(\emap_i,\loopconf_i)$, let us determine the distribution of $(\emap_{i+1},\loopconf_{i+1})$.
Notice that in the case of rigid loops in order to specify $(\emap_{i+1},\loopconf_{i+1})$ it is sufficient to determine which of the peeling events $\mathsf{G}_{k_1,k_2}$, $\mathsf{C}_k$ or 
$\mathsf{L}_{k,k}$ occurs.
If the degree of the hole incident to $e$ is $2l$ then with the help of Lemma \ref{thm:markov} we easily find that these probabilities are
\begin{align}
\prob(\mathsf{G}_{k_1,k_2} | \emap_i,\loopconf_i,e) &=  \frac{F^{(k_1)}(\qseq,g,n) F^{(k_2)}(\qseq,g,n)}{F^{(l)}(\qseq,g,n)} & \text{for }&k_1+k_2+1=l, k_1,k_2\geq 0,\nonumber\\
\prob(\mathsf{C}_{k} | \emap_i,\loopconf_i,e) &=  \frac{q_k\, F^{(l+k-1)}(\qseq,g,n)}{F^{(l)}(\qseq,g,n)} & \text{for }&k\geq 1,\label{eq:peelprob}\\
\prob(\mathsf{L}_{k,k} | \emap_i,\loopconf_i,e) &=  \frac{n \,g^{2k}\, F^{(k)}(\qseq,g,n)F^{(l+k-1)}(\qseq,g,n)}{F^{(l)}(\qseq,g,n)} & \text{for }&k\geq 1,\nonumber
\end{align}
where we use the convention that $F^{(0)}(\qseq,g,n)=1$.

\subsection{Targeted exploration of pointed Boltzmann loop-decorated maps}

Recall that a pointed $(\qseq,g,n)$-Boltzmann loop-decorated map $(\map_\bullet,\loopconf)$ of perimeter $2p$ is a loop-decorated map with a marked vertex sampled with probability proportional to $w_{\qseq,g,n}(\map_\bullet,\loopconf)$. 
In order for this to make sense we need that the \emph{pointed partition function}
\begin{equation*} 
F^{(p)}_{\bullet}(\qseq,g,n) \coloneqq \sum_{(\map_\bullet,\loopconf)\in\loopmaps_\bullet^{(p)}} w_{\qseq,g,n}(\map_\bullet,\loopconf)
\end{equation*} 
is finite.
The equivalence with the admissibility of $(\qseq,g,n)$ will be a consequence of Theorem \ref{thm:admissibility}.
For the time being we assume the following stronger hypothesis.

\begin{definition}[Strong admissibility]\label{def:strongadmissible}
	A triple $(\qseq,g,n)$ is called \emph{strongly admissible} if $F^{(p)}_\bullet(\qseq,g,n) <\infty$ for all $p \geq 1$.
\end{definition}

Let $(\qseq,g,n)$ be strongly admissible and $(\map_\bullet,\loopconf)$ a pointed $(\qseq,g,n)$-Boltzmann loop-decorated map of perimeter $2p$.
The equivalent of the Markov property Lemma \ref{thm:markov} in the pointed case can be stated as follows. 
Let $(\emap,\loopconf')$ be a fixed loop-decorated map of perimeter $2p$ with a single hole $h$.
If with positive probability $(\emap,\loopconf')\subset (\map_\bullet,\loopconf)$ and the marked vertex of $\map_\bullet$ is not an \emph{inner vertex} of $\emap$, where the inner vertices of $\emap$ are the vertices that are not incident to a hole, then the pointed map $\umap_\bullet$ filling in the hole $h$ is distributed as a pointed $(\qseq,g,n)$-Boltzmann loop-decorated map of perimeter $\deg(h)$.

With this Markov property it is again straightforward to study the law of the targeted peeling exploration \eqref{eq:targetlooppeeling},
\begin{equation*}
(\emap_0,\loopconf_0)\subset (\emap_1,\loopconf_1) \subset \cdots \subset (\emap_n,\loopconf_n) = (\map_\bullet,\loopconf),
\end{equation*}
of a pointed $(\qseq,g,n)$-Boltzmann loop-decorated map $(\map_\bullet,\loopconf)$ of perimeter $2p$.
Conditionally on $(\emap_i,\loopconf_i)$, if the half-degree of the hole of $\emap_i$ is $P_i=l$ then the events $\mathsf{C}_k$, $\mathsf{G}_{k,\ast}$, $\mathsf{G}_{\ast,k}$, $\mathsf{L}_{k,\ast}$ and $\mathsf{L}_{\ast,k}$ occur with probabilities 
\begin{align*}
\prob(\mathsf{G}_{k,*} | \emap_i,\loopconf_i,e)=\prob(\mathsf{G}_{*,k} | \emap_i,\loopconf_i,e) &= \frac{F^{(k)} F_\bullet^{(l-k-1)}}{F_\bullet^{(l)}} & \text{for }&0\leq k \leq l-1,\\
\prob(\mathsf{C}_{k} | \emap_i,\loopconf_i,e) &=  \frac{q_k\, F_\bullet^{(l+k-1)}}{F_\bullet^{(l)}} & \text{for }&k\geq 1,\\
\prob(\mathsf{L}_{*,k} | \emap_i,\loopconf_i,e) &=  \frac{n \,g^{2k}\, F^{(k)}F^{(l+k-1)}_\bullet}{F_\bullet^{(l)}} & \text{for }&k\geq 1,\\
\prob(\mathsf{L}_{k,*} | \emap_i,\loopconf_i,e) &=  \frac{n \, g^{2k}\, F^{(l+k-1)}F^{(k)}_\bullet}{F_\bullet^{(l)}} & \text{for }&k\geq 1,
\end{align*}
with $F_\bullet^{(l)}=F_\bullet^{(l)}(\qseq,g,n)$ and $F_\bullet^{(0)}=1$ by convention.
Moreover, when a hole of degree $2k$ is filled in with a loop-decorated map $\umap$, then $\umap$ is distributed as a $(\qseq,g,n)$-Boltzmann loop-decorated map of perimeter $2k$ independently of $(\emap_i,\loopconf_i)$.

Notice that after the peeling operation the half-degree $P_{i+1}$ of the hole is given by
\begin{equation*}
P_{i+1} = \begin{cases}
P_i - k - 1 & \mathsf{G}_{k,*}\text{ or }\mathsf{G}_{*,k},\\
P_i + k - 1 & \mathsf{C}_{k}\text{ or }\mathsf{L}_{*,k},\\
k & \mathsf{L}_{k,*}. \\
\end{cases}
\end{equation*}
It follows that the perimeter and nesting process $(P_i,N_i)$ of $(\map_\bullet,\loopconf)$ is a Markov process with transition probabilities given explicitly by
\begin{align}
\prob( P_{i+1}=p, N_{i+1}=N_i|P_i=l) &= \frac{F_\bullet^{(p)}}{F_\bullet^{(l)}} \begin{cases}
q_{p-l+1} + n\, g^{2(p-l+1)} F^{(p-l+1)}& p \geq l, \label{eq:perimeterlaw}\\
2 F^{(l-p-1)} & p < l,
\end{cases}\\
\prob( P_{i+1}=p, N_{i+1}=N_i+1|P_i=l) & = \frac{F_\bullet^{(p)}}{F_\bullet^{(l)}} n\, g^{2p} F^{(l+p-1)},\nonumber
\end{align}
where we used \eqref{eq:effq} and \eqref{eq:diskid}.
Following \cite{budd_peeling_2015} we introduce the notation
\begin{equation}\label{eq:nuqhatdef}
\nu_{\hat{\qseq}}(k) \coloneqq\gamma_\hqseq^{2k}\,\begin{cases}
\hat{q}_{k+1}& k \geq 0, \\
2 W^{(-k-1)}(\hat{\qseq}) & k < 0.
\end{cases}
\end{equation}
Then the transition probabilities translate into
\begin{align}
\begin{split}
\prob( P_{i+1}=p, N_{i+1}=N_i|P_i=l) &= \frac{F_\bullet^{(p)}\gamma_\hqseq^{-2p}}{F_\bullet^{(l)}\gamma_\hqseq^{-2l}}\, \nu_{\hat{\qseq}}(p-l),\qquad\qquad\qquad(l>0,p\geq 0) \\
\prob( P_{i+1}=p, N_{i+1}=N_i+1|P_i=l)&= \frac{F_\bullet^{(p)}\gamma_\hqseq^{-2p}}{F_\bullet^{(l)}\gamma_\hqseq^{-2l}}\, \frac{n}{2} \left(g\gamma_\hqseq^2\right)^{2p} \nu_{\hat{\qseq}}(-p-l).\qquad(l>0,p>0)
\end{split}\label{eq:perimlawnu}
\end{align}
Since $\hat{\qseq}$ is admissible, it follows from \cite[Corollary 23]{curien_peeling_nodate} and \cite[Lemma 9]{curien_peeling_nodate} that $\nu_{\hat{\qseq}}$ is a probability measure on $\Z$.
Later we will interpret the law \eqref{eq:perimlawnu}, in the non-generic critical case $g\gamma_\hqseq^2=1$, as an $h$-transform of a random walk with law $\nu_{\hqseq}$ that is confined to the nonnegative integers by a particular boundary condition.
However, before doing so we are going to consider more general processes of the form \eqref{eq:perimlawnu} in which $\hat{\qseq}$ is allowed to be any admissible sequence, not necessarily arising from a strongly admissible triple $(\qseq,g,n)$.

\section{Ricocheted random walks}\label{sec:rw}

Suppose $\nu:\Z\to\R$ is non-negative and $\sum_{k=-\infty}^\infty\nu(k)=1$, such that it defines a probability measure on $\Z$.
We denote by $(W_i)_{i\geq0}$ under $\probric_p$ the random walk started at $p\in\Z$ with independent increments distributed according to $\nu$.
The random walk $(W_i)_{i\geq0}$ is said to \emph{drift to $\pm\infty$} if $\lim_{i\to\infty}W_i = \pm\infty$ almost surely.
If neither, then the random walk is said to \emph{oscillate}.

\subsection{Wiener-Hopf factorization}
The \emph{weak ascending ladder epochs} $(T_j^\geq)$ of the walk $(W_i)$ correspond to the successive times at which the walk attains its running maximum, i.e. $T_0^\geq = 0$ and 
\begin{equation*}
T_{j+1}^\geq = \inf\{ i > T_j^\geq : W_i \geq W_{T_j^\geq}\} \in \Z \cup \{\infty\}\qquad \text{for }j\geq 0.
\end{equation*}
The \emph{weak ascending ladder process} $(H_j^\geq)$ is then given by $H_j^\geq = W_{T_j^\geq}$ provided $T_j^\geq < \infty$ and otherwise we set $H_j^\geq = \dagger$ for some cemetery state $\dagger$. 
If $(W_i)$ drifts to $-\infty$ then the process $(H_j^\geq)$ defines a \emph{defective} random walk, i.e. it has i.i.d. increments in $\Z$ and is sent to the cemetery state after a geometrically distributed number of steps.
If $(W_i)$ oscillates or drifts to $\infty$ then $(H_j^\geq)$ is a proper random walk on $\Z$.
Similarly, the \emph{strict descending ladder epochs} $(T_j^<)$ are defined via $T_0^< = 0$ and 
\begin{equation*}
T_{j+1}^< = \inf\{ i > T_j^<: W_i < W_{T_j^<}\} \in \Z \cup \{\infty\}\qquad \text{for }j\geq 0.
\end{equation*}
The \emph{strict descending ladder process} $(H_j^<)$ is given by $H_j^< = -W_{T_j^<}$ provided $T_j^< < \infty$ and otherwise we set $H_j^< = \dagger$.
It is a defective random walk if $(W_i)$ drifts to $\infty$ and proper otherwise.

We define the characteristic function $\varphi(\theta)$ and the probability generating functions $G^\geq(z)$ and $G^<(z)$ by
\begin{equation*}
\varphi(\theta) \coloneqq \expecric_0\left[e^{i\theta W_1}\right],\quad G^\geq(z) \coloneqq \expecric_0\left[z^{H_1^\geq}\one_{\{H_1^\geq\neq\dagger\}}\right],\quad G^<(z) \coloneqq \expecric_0\left[z^{H_1^<}\one_{\{H_1^<\neq\dagger\}}\right].
\end{equation*}
A classic result in random walks (see e.g. \cite[Section XVIII.3]{feller_introduction_1966}) is that $\varphi(\theta)$ satisfies the Wiener-Hopf factorization
\begin{equation}\label{eq:wienerhopfchar}
1-\varphi(\theta) = (1-G^{\geq}(e^{i\theta}))(1-G^{<}(e^{-i\theta})),
\end{equation}
which is valid for any real $\theta$.

\subsection{Admissibility criteria}

Let $h_0^\downarrow:\Z\to\R$ be defined by
\begin{equation*}
h_0^\downarrow(p) \coloneqq 4^{-p}\binom{2p}{p} \one_{\{p\geq 0\}}.
\end{equation*}
We say $h_0^\downarrow$ is \emph{$\nu$-harmonic} on $\N$ if 
\begin{equation*}
\sum_{k=-\infty}^{\infty} h_0^\downarrow(l+k) \nu(k) = h_0^\downarrow(l)\quad\text{for all }l\geq 1.
\end{equation*}
In \cite{budd_peeling_2015} it was realized that the mapping $\hqseq \mapsto \nu_{\hqseq}$ in \eqref{eq:nuqhatdef} determines a one-to-one correspondence between admissible weight sequences $\hqseq$ and laws $\nu$ for which $h_0^\downarrow$ is $\nu$-harmonic on $\N$.
More precisely, we have the following.

\begin{proposition}\label{thm:admiss}
Let $\nu:\Z\to\R$ be a probability measure on $\Z$ with $\nu(-1)>0$. Then the following are equivalent:
\begin{enumerate}[(i)]
	\item $\nu = \nu_\hqseq$ for some admissible sequence $\hqseq$.
	\item The function $h^\downarrow_0$ is $\nu$-harmonic on $\N$.
	\item The strict descending ladder process of the random walk $(W_i)$ with law $\nu$ has probability generating function
	\begin{equation*}
	G^<(z) = 1 - \sqrt{1-z}.
	\end{equation*} 
\end{enumerate}
\end{proposition}

\begin{proof}
The equivalence of (i) and (ii) is a direct consequence of \cite[Proposition 3 \& 4]{budd_peeling_2015}. 
It remains to prove the equivalence of (ii) and (iii).
In both cases we may assume that $(W_i)$ does not drift to $\infty$.
This follows from \cite[Proposition 4]{budd_peeling_2015} in case (ii) and from $G^<(1)=1$ in case (iii), since that implies that the strict descending ladder process $(H_j^<)$ is proper.
If $(W_i)$ does not drift to $\infty$, then (ii) is equivalent to $h_0^\downarrow(p)$ being the probability under $\probric_p$ that $(W_i)$ hits the non-positive integers at $0$.
In terms of the strict descending ladder process this is equivalent to $h_0^\downarrow(p) = \probric_0( (H_j^<)\text{ visits }p)$.
Since $(H_j^<)$ is a renewal process, its law is characterized precisely by these probabilities for all $p \geq 1$.
Hence, it suffices to check that $h_0^\downarrow(p) = \probric_0( (H_j^<)\text{ visits }p)$ is satisfied when $G^<(z)=1-\sqrt{1-z}$.
Indeed we may explicitly calculate
\begin{equation*}
\probric_0( (H_j^<)\text{ visits }p) = \sum_{j=0}^\infty \probric_0( H_j^<=p)= \sum_{j=0}^\infty [z^p](G^<(z))^j = [z^p] \frac{1}{1-G^<(z)} = [z^p] \frac{1}{\sqrt{1-z}} = h_0^\downarrow(p),
\end{equation*}
proving the equivalence of (ii) and (iii).
\end{proof}

To ease the exposition we use the following terminology.

\begin{definition}
	$\nu:\Z\to\R$ is \emph{admissible} iff it satisfies any of the conditions of Proposition \ref{thm:admiss}. 
\end{definition}

We see that the law of the strict descending ladder process is universal in the sense that it is shared by any admissible law $\nu$.
This will be quite useful in the following.
We introduce the functions $H_l:\Z\to \R$ for $l\geq 0$ by setting 
\begin{equation*}
H_l(p) = \frac{p}{l+p}\, h^\downarrow_0(p)\,h^\downarrow_0(l)\,\one_{\{p>0\}} + \one_{\{l=p=0\}},
\end{equation*}
such that $H_0 = h_0^\downarrow$.

\begin{lemma}\label{thm:Hlpproof}
If $\nu$ is admissible, then for $l\geq 0,p\geq 1$ we have
\begin{equation*}
\probric_p[ (W_i)\text{ hits }\Z_{\leq 0}\text{ at }-l ] = H_l(p).
\end{equation*}
\end{lemma}
\begin{proof}
Let us denote 
\begin{equation*}
P_l(p) = \probric_0[(H_i^<)\text{ hits }\{p,p+1,\ldots\}\text{ at }p+l]
\end{equation*}
and $P_l(0) = \one_{\{l=0\}}$.
We aim to show that $P_l(p)=H_l(p)$.
To this end we notice that for $p,k\geq 0$ we have the decomposition
\begin{equation*}
\probric_0[(H_i^<)\text{ visits }p+k] = \sum_{l=0}^k \probric_0[(H_i^<)\text{ hits }\{p,p+1,\ldots\}\text{ at }p+l]\,\,\probric_0[(H_i^<)\text{ visits }k-l],
\end{equation*}
which implies that $P_l(p)$ satisfies 
\begin{equation*}
h^\downarrow_0(p+k) = \sum_{l=0}^k P_l(p)h^\downarrow_0(k-l)\quad \text{for all }p,k\geq 0.
\end{equation*}
	We can rewrite this as an equation for the generating function $P(x,y) := \sum_{p=0}^\infty \sum_{l=0}^\infty P_l(p) x^l y^p$, which necessarily converges for $|x|,|y|<1$, as follows.
	Since $\sum_{p=0}^\infty h^\downarrow_0(p) x^p = 1/\sqrt{1-x}$ for $|x|<1$, we have
	\begin{align*}
	\frac{P(x,y)}{\sqrt{1-x}} &= \sum_{k,p\geq 0} \sum_{l=0}^k P_l(p)h^\downarrow_0(k-l) x^k y^l = \sum_{k,p\geq 0} h^\downarrow_0(p+k)x^ky^p\\
	&= \sum_{l=0}^\infty\sum_{p=0}^l h^\downarrow_0(l) x^{l-p} y^p = \sum_{l=0}^\infty h^\downarrow_0(l) \left( \frac{x^{l+1}-y^{l+1}}{x-y}\right) =  \frac{\frac{x}{\sqrt{1-x}}-\frac{y}{\sqrt{1-y}}}{x-y}.
	\end{align*}
	To find the coefficients $P_l(p)$, notice that this implies that
	\begin{equation*}
	(x \partial_x+y \partial_y) P(x,y) = y\partial_y \left(\frac{1}{\sqrt{(1-x)(1-y)}}\right),
	\end{equation*}
	and therefore $(p+l)P_l(p) = p\, h^\downarrow_0(p) h^\downarrow_0(l)$ for $l,p\geq 0$.
	Combining with the fact that $P_0(0) = 1$, this shows that $P_p(l) = H_p(l)$ as claimed.
\end{proof}

For future reference, let us note that the proof of Lemma \ref{thm:Hlpproof} shows that $H_l(p)$ has the generating function 
\begin{equation}\label{eq:Hgen}
\sum_{p=0}^\infty \sum_{l=0}^\infty H_l(p) x^l y^p = \frac{x-y\sqrt{\frac{1-x}{1-y}}}{x-y}.\quad\quad (|x|,|y|\leq 1, y\neq 1),
\end{equation}
which should be understood as $\frac{1}{2}(1+1/(1-x))$ when $y=x$.

\subsection{Ricocheted random walks}\label{sec:ricochetedwalk}

For a general probability distribution $\nu$ on $\Z$ and a constant $\pr\in[0,1]$ we define the \emph{$\pr$-ricocheted random walk} to be a Markov process $(W_i^\ast,N^*_i)_{i\geq 0}$ on $\Z^2$ started at $(k,0)$ under $\probric^*_k$ with the transition probabilities
\begin{alignat}{3}
\probric^*_k( W^\ast_{i+1} = l,&\, N^*_{i+1}=N^*_i &|\, W^\ast_i = p ) &= \left(1-\pr\,\one_{\{l<0\}}\right)\nu(l-p),\quad && (p\geq 0,l\in \Z)\nonumber\\
\probric^*_k( W^\ast_{i+1} = l,&\, N^*_{i+1}=N^*_i+1\, &|\, W^\ast_i = p ) &= \pr\,\nu(-l-p),  && (p\geq 0,l>0)\label{eq:Wastprob}\\
\probric^*_k( W^\ast_{i+1} = p,&\, N^*_{i+1}=N^*_i\, &|\, W^\ast_i = p ) &= 1.  && (p\leq 0)\nonumber
\end{alignat}
One should think of this process as a random walk with law $\nu$ on $\N$ with a special boundary condition whenever it is about to jump outside $\N$ to, say, $-k < 0$: with (independent) probability $1-\pr$ it ``penetrates'' the ``wall'' $\Z_{\leq 0}$ and the process is trapped at $-k$ forever; otherwise with probability $\pr$ it ``ricochets off the wall'' (like a bullet) and lands at $k$; upon hitting $0$ the process is trapped with probability $1$.
The process $(N^*_i)$ simply counts the number of ricochets that have occurred so far.
If the process is trapped eventually, then we denote by $(W_\infty^*,N_\infty^*)$ its final position (see Figure \ref{fig:ricochet}).

\begin{figure}[h]
	\centering
	\includegraphics[width=.45\linewidth]{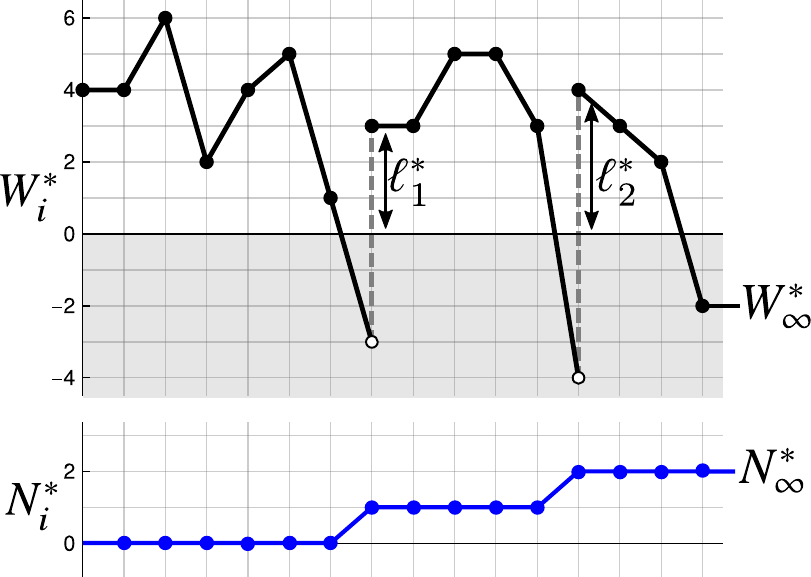}
	\caption{A $\pr$-ricocheted random walk $(W_i^*,N_i^*)_{i\geq0}$ that is eventually trapped at $(W_\infty^*,N_\infty^*)=(-2,2)$.
		\label{fig:ricochet}}
\end{figure}

To any such ricocheted walk $(W_i^\ast,N^*_i)_{i\geq 0}$ we can associate the \emph{ricochet sequence} $(\ell^*_n)_{n\geq0}$ that keeps track of the sizes of ricochets.
To be precise, we let $\ell^*_0=W_0^*$ and for $1\leq n\leq N_\infty^*$ set $\ell^*_n=W^\ast_{i}$ if the $n$th ricochet occurs at the $i$th step (i.e. $N^*_{i-1}=n-1$ and $N^*_{i}=n$).
If $n > N_\infty^*$ then we set $\ell^*_n = W^*_\infty$.
The sequence $(\ell^*_n)_{n=0}^m$ is seen to be a Markov process with transition probabilities given by
\begin{alignat*}{3}
\probric^*_k( \ell^*_{n+1}=l\, |\, \ell^*_n = p) &=& \pr\,& \probric_p( (W_i)\text{ hits }\Z_{\leq 0}\text{ at }-l ),&&(l>0,p>0)\\
\probric^*_k( \ell^*_{n+1}=-l \,|\, \ell^*_n = p) &=&\, (1-\pr\one_{l>0})&\probric_p( (W_i)\text{ hits }\Z_{\leq 0}\text{ at }-l ),\qquad&&(l\geq 0,p>0)\\
\probric^*_k( \ell^*_{n+1}=p \,|\, \ell^*_n = p) &=&\,1,&\qquad&&(p\leq 0)
\end{alignat*}
where $(W_i)$ under $\probric_p$ is a random walk with law $\nu$ started at $p>0$.
 
Let us now look at admissible laws $\nu$.
By Lemma \ref{thm:Hlpproof} the law of the ricochet sequence $(\ell^*_n)_{n\geq0}$ is 
\begin{equation}
\begin{alignedat}{2}
\probric^*_k( \ell^*_{n+1}=l\, |\, \ell^*_n = p) &= \pr \,H_l(p) \one_{\{l\geq 0\}} + (1-\pr)H_{|l|}(p)\one_{\{l\leq 0\}},\quad&&(p>0,l\in\Z)\label{eq:ricochetlaw}\\
\probric^*_k( \ell^*_{n+1}=p\, |\, \ell^*_n = p) &= 1, &&(p\leq 0)
\end{alignedat}
\end{equation}
and is thus independent of $\nu$.

\begin{lemma}\label{thm:trapped}
If $\nu$ is admissible and $\pr\in[0,1]$, then the $\pr$-ricocheted random walk almost surely gets trapped in finite time.
\end{lemma}
\begin{proof}
This is clear in the case $\pr<1$, since the random walk $(W_i)_i$ does not drift to $\infty$ and each time it is about to hit $\Z_{\leq 0}$ it is trapped with probability at least $1-\pr$.
We could prove the case $\pr=1$ by using a convenient choice of law $\nu$ and relying on recurrence criteria of random walks on $\Z$.
Instead we refer to Proposition \ref{thm:alternation} below which relates the ricochet sequence to certain alternations of a simple random walk on $\Z^2$.
Since the simple random walk is recurrent the ricochet sequence almost surely gets trapped at $0$.
The same is thus true for the $1$-ricocheted random walk.
\end{proof}

It turns out that the probability that the $\pr$-ricocheted random walk gets trapped at $0$ is given precisely by the family \eqref{eq:hdowndef} of functions $h^\downarrow_\pr:\Z_{\geq0}\to\R$ introduced in the introduction.

\begin{proposition}\label{thm:hdown}
	If $\nu$ is admissible and $\pr\in [0,1]$, then the probability that the $\pr$-ricocheted random walk $(W_i^\ast,N_i^*)_{i\geq 0}$ gets trapped at $0$ is
	\begin{equation}\label{eq:probtrap}
	\probric^*_p(W_\infty^\ast = 0 ) = h_\pr^\downarrow(p).
	\end{equation} 
	Moreover, the functions $h_\pr^\downarrow$ satisfy the following properties:
	\begin{enumerate}
		\item for $\pr\in(-1,1)$ it has generating function
			\begin{equation}\label{eq:hdowngen}
			\sum_{p=0}^{\infty}h^\downarrow_\pr(p)x^{2p} = \frac{1}{1+\pr}\left[\pr+\cosh\left(2(b-1)\mathrm{arctanh}\, x\right)\right] \quad (|x| \leq 1, x\neq \pm 1),
			\end{equation}
		\item $\pr \to h_\pr^\downarrow(p)$ is analytic on $(-1,1)$ and continuous on $[-1,1]$ for all $p\geq 0$;
		\item $h_\pr^\downarrow(p)\in(0,1]$ for all $\pr\in[0,1]$ and $p\geq 0$;
		\item for $\pr\in(-1,1]$ it satisfies the asymptotics
		\begin{equation}\label{eq:hdownasymp}
		h_\pr^\downarrow(p) \stackrel{p\to\infty}{\sim} \mathsf{r}_\pr\, p^{-b}\quad \text{with }\mathsf{r}_\pr = 2 \Gamma(b)\frac{4^{-b}}{\pi}\sqrt{\frac{1-\pr}{1+\pr}},
		\end{equation}
		uniformly in $\pr$ in any compact subinterval.
	\end{enumerate}
\end{proposition}
\begin{proof}
Let us start by verifying (i)-(iv).
The expression (\ref{eq:hdowngen}) is even and analytic in $x\in(-1,1)$ and can be rewritten as
\begin{align}\label{eq:hdowngen2}
\frac{1}{\pr+1}\left[\pr+ \cosh\left(2(b-1)\mathrm{arctanh}\, x\right)\right] &= \frac{1}{\pr+1}\left[\pr + \frac{1}{2}\left(\frac{1-x}{1+x}\right)^{1-b}+ \frac{1}{2}\left(\frac{1-x}{1+x}\right)^{b-1}\right].
\end{align}
The coefficient of $x^0$ is clearly equal to $h_\pr^\downarrow(0)=1$ while the coefficient of $x^{2l}$ for $l\geq 1$ equals
\begin{align}
\frac{1}{\pr+1} \left[x^{2l}\right] \left(\frac{1-x}{1+x}\right)^{1-b}
&= \frac{1}{\pr+1} \sum_{k=0}^{2l} (-1)^k \frac{\Gamma(2-b)}{\Gamma(k+1)\Gamma(2-b-k)} \frac{\Gamma(b)}{\Gamma(2l-k+1)\Gamma(b-2l+k)} \label{eq:hgenfunexpr}\\
&= \frac{1}{\pr+1}  \frac{\Gamma(b)}{\Gamma(2l+1)\Gamma(b-2l)} \sum_{k=0}^{2l} \binom{2l}{k} \frac{(b-1)_k}{(b-2l)_k}\nonumber\\
&= \frac{1}{\pr+1} \frac{\Gamma(b)}{\Gamma(2l+1)\Gamma(b-2l)} \,{_2F_1}(-2l,b-1;b-2l;-1) = h_\pr^\downarrow(l),\nonumber
\end{align}
where $(a)_k := \Gamma(a+k)/\Gamma(a)$ is the (rising) Pochhammer symbol.
This proves (i). 	

We also notice from \eqref{eq:hgenfunexpr} that $(\pr+1)h^\downarrow_\pr(p)$ is a polynomial in $b=\tfrac{1}{\pi}\arccos\pr$ for any $p\geq 1$.
Therefore $h^\downarrow_\pr(p)$ is analytic in $\pr\in (-1,1)$.
Using the generating function it is straightforward to check that the limiting values of $h^\downarrow_\pr(p)$ as $\pr\to\pm1$ are given by those of \eqref{eq:hdowndef}.
This shows that (ii) is also granted.
Part (iii) will be a direct consequence of the main identity \eqref{eq:probtrap} and the fact that $\probric_p^*(W^*_\infty = 0)>0$ for all $p\geq 0$.
The singular behaviour of \eqref{eq:hdowngen2} as $x\to\pm1$ is
\begin{equation*}
\frac{\pr}{1+\pr} + 2\frac{4^{-b}}{1+\pr} (1-x^2)^{1-b} + o(1),
\end{equation*} 
where the error is uniform in $\pr$ in any compact subinterval of $(-1,1]$.
Transfer theorems then imply the asymptotics of (iv).
	
If $\pr=0$ then by Proposition \ref{thm:Hlpproof} we have $\probric^*_p(W_\infty^*=0) = H_0(p)$ which agrees with $h^\downarrow_0(p)$.
If $\pr=1$ then Lemma \ref{thm:trapped} implies $\probric^*_p(W_\infty^*=0) = 1 = h^\downarrow_1(p)$ for every $p \geq 0$. 
In the following we therefore concentrate on proving \eqref{eq:probtrap} in the case $0<\pr<1$.

Let us denote $P_\pr(p) = \probric^*_p(W_\infty^*=0)$	the sought-after probability.
From the transition probabilities \eqref{eq:ricochetlaw} it follows that it can be expressed as
\begin{equation}\label{eq:trapgenfun}
P_\pr(p)= \sum_{N=0}^\infty \pr^N\sum_{(p_1,\ldots,p_N)\in\Z_{>0}^N} H_0(p_1)H_{p_1}(p_2)\cdots H_{p_{N-1}}(p_N)H_{p_N}(p).
\end{equation}
Since the second sum is bounded by $1$ it follows that $P_\pr(p)$ is analytic in $\pr$ around $0$ for all $p$.
It is easy to see that $P_\pr(p)$ is the unique such analytic solution to the system of equations
\begin{equation}\label{eq:hdowneq}
P_\pr(p) = H_0(p) + \pr\sum_{l=1}^\infty  P_\pr(l) H_l(p)\quad\quad(l\geq 0).
\end{equation}
According to (ii) $h^\downarrow_\pr(l)$ is analytic in $\pr$ around $0$, so it suffices to check that $P_\pr(p)=h^\downarrow_\pr(p)$ solves (\ref{eq:hdowneq}).

	By changing variables $x=\tanh^2\eta$ and $y=\tanh^2\sigma$ in \eqref{eq:Hgen} we deduce that for $\eta,\sigma \in \R + i[-\pi/4,\pi/4]$,
	\begin{equation*}
	\sum_{p,l\geq 1} H_p(l) (\tanh\eta)^{2p}(\tanh\sigma)^{2l} = \frac{\cosh \sigma \cosh \eta}{\cosh \sigma+\cosh\eta}\left(\cosh\sigma -1\right)\left(1-\frac{1}{\cosh\eta}\right),
	\end{equation*}
	while 
	\begin{equation*}
	\sum_{l=1}^\infty H_0(l) (\tanh\sigma)^{2l} = \cosh\sigma - 1. 
	\end{equation*}
	For any $m\in \Z$ we have
	\begin{equation*}
	\one_{\{m=0\}} = \frac{1}{2\pi} \int_0^{2\pi} e^{im\theta}\rmd\theta = \frac{i}{\pi} \int_{-\infty}^{\infty} \frac{\tanh^{2m}(x+\pi i/4)}{\sinh(x+\pi i /4)\cosh(x+\pi i/4)}\rmd x,
	\end{equation*}
	where we made the change of variables $e^{i\theta} = \tanh^2(x+\pi i/4)$.
	By setting $m=l-k$, this implies that for any $l,k\in\Z$,
	\begin{equation*}
	\frac{i}{\pi}\int_{\R +\frac{\pi i}{4}}\rmd\eta \frac{(\tanh \eta)^{2l}(\tanh \bar{\eta})^{2k}}{\sinh \eta\cosh\eta} = \one_{\{l=k\}}, 
	\end{equation*}
	where the integration is over the line $\mathrm{Im}(\eta)=\pi/4$.
	Writing $f(l) \coloneqq \pr\sum_{p=1}^{\infty}h^\downarrow_\pr(p)H_p(l)$ and using \eqref{eq:hdowngen} this allows us to express for $\sigma>0$,
	\begin{align}
	\sum_{l=1}^\infty f(l) (\tanh \sigma)^{2l} &= \frac{i}{\pi}\frac{\pr}{\pr+1} \cosh \sigma(\cosh\sigma-1) \int_{\R +\frac{\pi i}{4}}\!\!\rmd\eta \frac{1}{\sinh\eta}\left(1-\frac{1}{\cosh\eta}\right)\frac{\cosh(2(b-1)\bar\eta)}{\cosh \sigma+\cosh \eta}\nonumber\\
	&=  \frac{i}{2\pi}\frac{\pr}{\pr+1} \cosh \sigma(\cosh\sigma-1) \left[ e^{(1-b)\pi i} I(\sigma,b-1) + e^{(b-1)\pi i} I(\sigma,1-b)\right],\label{eq:fgenepxr}
	\end{align}
	where we defined
	\begin{equation*}
	I(\sigma,a) :=  \int_{\R +\frac{\pi i}{4}}\!\!\rmd\eta \frac{1}{\sinh\eta}\left(1-\frac{1}{\cosh\eta}\right)\frac{\exp(2a\eta)}{\cosh \sigma+\cosh \eta}.
	\end{equation*}
	The latter can be evaluated by noting that a shift in the contour from $\R+\frac{\pi i}{4}$ to $\R-\frac{7\pi i}{4}$ results in an integral equal to $\exp(-4\pi i a)I(\sigma,a)$.
	The integrand is analytic and falls off sufficiently fast as $\mathrm{Re}(\eta)\to\pm\infty$, while the only poles in the region $\mathrm{Im}(\eta)\in [-7\pi/4,\pi/4]$ are located at $\eta = -\frac{\pi i}{2},-\pi i,- \frac{3\pi i}{2}, -\pi i\pm \sigma$ with respective residues
	\begin{equation*}
	\frac{ \exp(-a\pi i)}{\cosh\sigma}, \quad \frac{2\exp(-2a\pi i)}{1-\cosh\sigma}, \quad \frac{\exp(-3a\pi i)}{\cosh\sigma}, \quad \frac{\exp(-2a\pi i \pm 2 a\sigma)}{\cosh\sigma(\cosh\sigma-1)}.
	\end{equation*}
	By the residue theorem the difference between the two integrals is
	\begin{equation*}
	\left(e^{-4 a\pi i}-1\right)I(\sigma,a) = 2\pi i e^{-2a\pi i} \left[ \frac{2\cos(\pi a)}{\cosh\sigma} + \frac{2}{1-\cosh\sigma}+ \frac{2\cosh(2a\sigma)}{\cosh\sigma(\cosh\sigma-1)}\right].
	\end{equation*}
	Hence
	\begin{equation*}
	I(\sigma,a) = -\frac{2\pi}{\sin(2\pi a)}\left[ \frac{\cos(\pi a)}{\cosh\sigma} + \frac{1}{1-\cosh\sigma}+ \frac{\cosh(2a\sigma)}{\cosh\sigma(\cosh\sigma-1)}\right].
	\end{equation*}
	Plugging into (\ref{eq:fgenepxr}) yields
	\begin{align*}
	\!\!\!\!\sum_{l=1}^\infty f(l) (\tanh \sigma)^{2l} &= \frac{\pr}{\pr+1} \frac{\cosh\sigma(\cosh\sigma-1)}{\cos(\pi b)}\left[ \frac{-\cos(\pi b)}{\cosh\sigma} + \frac{1}{1-\cosh\sigma}+ \frac{\cosh(2(b-1)\sigma)}{\cosh\sigma(\cosh\sigma-1)}\right]\\
	&\!\!\!= \frac{\pr}{\pr+1} + \frac{1}{\pr+1}\cosh(2(b-1)\sigma)-\cosh(\sigma) = \sum_{l=1}^{\infty} (h^\downarrow_n(l)-H_0(l)) (\tanh \sigma)^{2l},
	\end{align*}
	implying $f(l) + H_0(l)= h^\downarrow_n(l)$ for $l\geq 1$ as desired.
\end{proof}

\subsection{Conditioning to be trapped at $0$}

By comparing the probability $\probric^*_p(W_\infty^*=0)$ before and after the first step of the ricocheted walk, we deduce that $h^\downarrow_\pr$ satisfies for $p>0$,
\begin{equation}\label{eq:hdownharm}
\sum_{l=0}^\infty h^\downarrow_\pr(l) \left( \nu(l-p) + \pr\one_{\{l>0\}} \nu(-l-p)\right) = h^\downarrow_\pr(p).
\end{equation}

Let us define the \emph{$\pr$-ricocheted random walk conditioned to be trapped at $0$} to be the process $(W^\downarrow_i,N^\downarrow_i)_{i\geq0}$ under $\probric^\downarrow_p$ that has the law of the $\pr$-ricocheted random walk $(W^\ast_i,N^*_i)_{i\geq 0}$ conditional on $W^*_\infty = 0$ under $\probric^*_p$.
Then $(W^\downarrow_i,N^\downarrow_i)_{i\geq0}$ is the Markov process obtained from $(W^\ast_i,N^*_i)_{i\geq 0}$ by the  $h$-transform with respect to the harmonic function $(p,n)\to h^\downarrow_\pr(p)\one_{\{p\geq0\}}$, i.e.
\begin{alignat}{3}
\probric^\downarrow_k( W^\downarrow_{i+1} = l,&\, N^\downarrow_{i+1}=N^\downarrow_i &|\, W^\downarrow_i = p ) &= \frac{h^\downarrow_\pr(l)}{h^\downarrow_\pr(p)}\nu(l-p),\quad && (p> 0,l\geq 0)\nonumber\\
\probric^\downarrow_k( W^\downarrow_{i+1} = l,&\, N^\downarrow_{i+1}=N^\downarrow_i+1 &|\, W^\downarrow_i = p ) &= \frac{h^\downarrow_\pr(l)}{h^\downarrow_\pr(p)}\pr\,\nu(-l-p), \quad && (p>0,l>0)\label{eq:markovdown}\\
\probric^\downarrow_k( W^\downarrow_{i+1} = 0,&\, N^\downarrow_{i+1}=N^\downarrow_i\, &|\, W^\downarrow_i = 0 ) &= 1.  && \nonumber
\end{alignat}
Since $(W^\ast_i,N^*_i)_{i\geq 0}$ is almost surely trapped after a finite time, the same is true for $(W^\downarrow_i,N^\downarrow_i)_{i\geq0}$, i.e., almost surely $W^\downarrow_i=0$ for some finite $i$.
The ricochet sequence of $(W^\downarrow_i,N^\downarrow_i)_{i\geq0}$ is denoted $(\ell^\downarrow_n)_{n=0}^\infty$, for which the transition probabilities follow from \eqref{eq:ricochetlaw},
\begin{equation}
\begin{alignedat}{2}
\probric^\downarrow_k( \ell^\downarrow_{n+1}=l\, |\, \ell^\downarrow_n = p) &= \frac{h^\downarrow_\pr(l)}{h^\downarrow_\pr(p)}\pr \,H_l(p) + \frac{h^\downarrow_0(p)}{h^\downarrow_\pr(p)} (1-\pr)\one_{\{l=0\}},\quad&&(p>0,l\geq 0)\label{eq:ricochetdownlaw}\\
\probric^\downarrow_k( \ell^\downarrow_{n+1}=0\, |\, \ell^\downarrow_n = 0) &= 1. &&
\end{alignedat}
\end{equation}

One may already observe the similarity between \eqref{eq:markovdown} and the transition probabilities \eqref{eq:perimlawnu} of the perimeter process of a strongly admissible triple $(\qseq,g,n)\in\Ddomain$ when $\nu=\nu_\hqseq$ and $\pr = n/2$.
We could proceed to prove that $F_\bullet^{(p)}\gamma^{-2p} = h^\downarrow_\bullet(p)$ by showing that both sides define harmonic functions for $(W^\ast_i,N^*_i)_{i\geq 0}$ and relying on a uniqueness argument.
Instead we will observe that $F_\bullet^{(p)}\gamma^{-2p} = h^\downarrow_\bullet(p)$ holds for any admissible triple $(\qseq,g,n)$ as a consequence of the proof of Theorem \ref{thm:admissibility} presented in Section \ref{sec:proofadmiss} below.

\subsection{Nesting and winding statistics}\label{sec:winding}

The explicit law \eqref{eq:ricochetdownlaw} of the ricochet sequence show that the law of the total number $N_\infty^\downarrow$ of ricochets only depends on $\pr$ and not on $\nu$.

\begin{proposition}\label{thm:ricochets}
Let $\nu$ be admissible, $\pr\in(0,1]$, and $(W_i^\downarrow,N_i^\downarrow)_{i\geq 0}$ be the $\pr$-ricocheted random walk of law $\nu$ conditioned to be trapped at $0$. Then the total number $N_\infty^\downarrow$ of ricochets has probability generating function
\begin{equation}\label{eq:numricochets}
\expecric_k^\downarrow[ x^{N^\downarrow_\infty} ] = \frac{h^\downarrow_{x\pr}(k)}{h^\downarrow_{\pr}(k)},\qquad x\in[-1/\pr,1/\pr].
\end{equation}
If $\pr=1$ it satisfies the convergence in distribution
\begin{equation}\label{eq:ricochetsconvdist}
\frac{N^\downarrow_\infty}{\tfrac{1}{\pi^2}\log^2 k} \xrightarrow[k\to\infty]{\mathrm{(d)}} \mathcal{L},
\end{equation}
where $\mathcal{L}$ is a L\'evy random variable with density $e^{-1/(2x)}/\sqrt{2\pi x^3}\,\rmd x$.
For $\pr<1$ we have the convergence in probability 
\begin{equation*}
\frac{N_\infty^\downarrow}{\log k} \xrightarrow[k\to\infty]{\mathrm{(p)}} \frac{\pr}{\pi\sqrt{1-\pr^2}}
\end{equation*}
and the large deviation property
\begin{align*}
\frac{\log\probric_k^{\downarrow}\left[ N^\downarrow_\infty < \lambda\log k\right]}{\log k}& \xrightarrow{k\to\infty } -\frac{1}{\pi} J_\pr(\pi \lambda) \quad\text{for}\quad 0< \lambda < \tfrac{\pr}{\sqrt{1-\pr^2}}, \\
\frac{\log\probric_k^{\downarrow}\left[ N^\downarrow_\infty > \lambda\log k\right]}{\log k}& \xrightarrow{k\to\infty } -\frac{1}{\pi} J_\pr(\pi \lambda) \quad\text{for}\quad \lambda >\tfrac{\pr}{\sqrt{1-\pr^2}},
\end{align*} 
where 
\begin{equation*}
J_\pr(x) = x \log \left( \frac{x}{\pr\sqrt{1+x^2}}\right) + \arccot x - \arccos\pr.
\end{equation*}
\end{proposition}

\begin{proof}
It is easily seen, e.g. from \eqref{eq:trapgenfun}, that for $\pr \in (0,1]$ we have
\begin{equation*}
\probric_k^*(W_\infty^*=0,N_\infty^*=n) = \pr^n [x^n]h^\downarrow_x(k)
\end{equation*}
The conditioning of the $\pr$-ricocheted random walk to be trapped at $0$ corresponds to an $h$-transform with respect to $h^\downarrow_\pr$.
Therefore 
\begin{equation*}
\probric_k^\downarrow(N_\infty^\downarrow=n) = \frac{h^\downarrow_\pr(0)}{h^\downarrow_\pr(k)} \probric_k^*(W_\infty^*=0,N_\infty^*=n) = \frac{\pr^n}{h^\downarrow_\pr(k)} [x^n]h^\downarrow_x(k).
\end{equation*}
We conclude that the generating function of $N_\infty^\downarrow$ satisfies \eqref{eq:numricochets}.

If $\pr=1$ then according to Proposition \ref{thm:hdown}(iv) we have for any $\lambda \geq 0$ as $k\to\infty$,
\begin{equation*}
\expecric^\downarrow_k\left[ \exp\left(-\lambda \frac{\pi^2N_\infty^\downarrow}{\log^2 k}\right) \right] \sim \mathsf{r}_1 k^{-\frac{1}{\pi}\arccos\exp\left( - \frac{\pi^2\lambda}{\log^2 k}\right)} \sim k^{-\frac{1}{\pi}\sqrt{\frac{2\pi^2\lambda}{\log^2 k}}} = e^{-\sqrt{2\lambda}},
\end{equation*}
but that is precisely the Laplace transform of the L\'evy random variable with density $e^{-1/(2x)}/\sqrt{2\pi x^3}\,\rmd x$.
We may therefore conclude the convergence in distribution 
The convergence in distribution \eqref{eq:ricochetsconvdist} thus follows.

If $\pr\in(0,1)$ Proposition \ref{thm:hdown}(iv) shows that 
\begin{equation*}
\expecric_k^\downarrow[ x^{N^\downarrow_\infty} ] = \frac{h^\downarrow_{x\pr}(k)}{h^\downarrow_{\pr}(k)} = A(x) \, B(x)^{\log k} (1+o(1))
\end{equation*}
with the error uniform in a neighbourhood of $x=1$, where 
\begin{equation*}
A(x) \coloneqq \frac{\mathsf{r}_{x\pr}}{\mathsf{r}_{\pr}}, \quad B(x) \coloneqq e^{b - \frac{1}{\pi}\arccos(\pr x)}.
\end{equation*}
Since $A(x)$ and $B(x)$ are analytic around $x=1$, $A(1)=B(1)=1$, and $B''(1)+B'(1)-B'(1)^2 \neq 0$, it follows from \cite[Theorem IX.8]{flajolet_analytic_2009} that $N^\downarrow_\infty/\log k$ converges in probability to $B'(1) = \pr /(\pi\sqrt{1-\pr})$.
A simple calculation shows furthermore that for $\lambda > 0$,
\begin{equation*}
\inf_{x\in(0,1/\pr)} \log\left( \frac{B(x)}{x^\lambda} \right) = \frac{1}{\pi} J_\pr(\pi\lambda).
\end{equation*}
The large deviations property then follows from \cite[Theorem IX.15]{flajolet_analytic_2009}. 
\end{proof}

In the case $\pr=1$ we have a surprising interpretation of the ricochet sequence $(\ell_n^\downarrow)_{n\geq 0}$ and thus of the number $N_\infty^\downarrow$ of ricochets.
Let $(X_i)_{i\geq 0}$ be a simple random walk on $\Z^2$ started at $(k,k)$ for some $x\geq 1$.
Denote the positive and negative diagonals by $\Delta_\pm \coloneqq \{ (y,y) : \pm y \geq 0\}$.
We define the \emph{alternation times} $(t_n)_{n\geq 0}$ by setting $t_0 = 0$ and $t_{i+1} = \inf\{ t \geq t_i : X_t\in \Delta_- \}$ for $i$ even and $t_{i+1} = \inf\{ t \geq t_i : X_t\in \Delta_+ \}$ for $i$ odd.
The \emph{alternation sequence} $(x_n)_{n\geq 0}$ is defined by setting $x_n$ equal to the absolute value of the first coordinate of $X_{t_n}$ (see Figure \ref{fig:alternation}).

\begin{figure}[th]
	\centering
	\includegraphics[width=.35\linewidth]{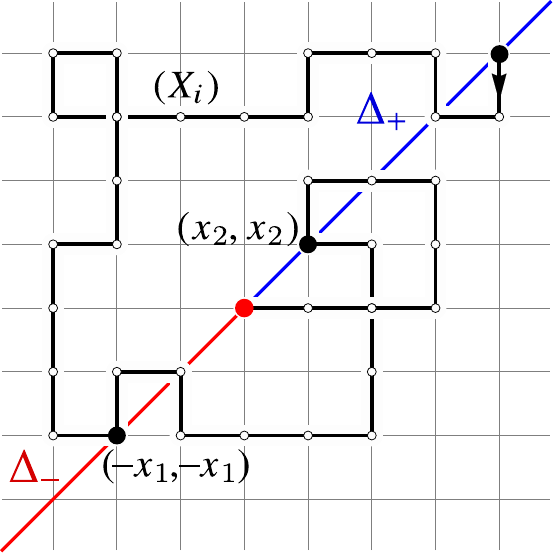}
	\caption{An example of a walk $(X_i)$ with two alternations between the diagonals. \label{fig:alternation}}
\end{figure}

\begin{proposition}\label{thm:alternation}
	For $\nu$ admissible and $\pr=1$ the ricochet sequence $(\ell^\downarrow_n)_{n\geq 0}$ under $\probric^\downarrow_k$ is identical in law to the alternation sequence $(x_n)_{n\geq 0}$ of a simple random walk on $\Z^2$ started at $(k,k)$.
\end{proposition}
\begin{proof}
	Since the law of the ricochet sequence does not depend on $\nu$, we may restrict to one particular admissible law $\nu$.
	To this end define $\nu(k)$ for $k\in\Z$ to be probability that the random walk $(X_i)_{i\geq 0}$ on $\Z^2$ started at the origin returns to the (full) diagonal for the first time at $(k,k)$. 
	By summing the probabilities that this happens after $n\geq 1$ steps we arrive at the corresponding characteristic function 
	\begin{equation*}
	\varphi(\theta) = \sum_{n=1}^\infty 4^{-2n} \, 2\,C_{n-1}\, (e^{i\theta/2}+e^{-i\theta/2})^{2n} = 1 - |\sin(\theta/2)| = 1 - \frac{1}{2}\sqrt{1-e^{i\theta}}\sqrt{1-e^{-i\theta}},
	\end{equation*}
	where $C_n = \frac{1}{n+1}\binom{2n}{n}$ are the Catalan numbers.
	The uniqueness of the Wiener-Hop factorization \eqref{eq:wienerhopfchar} implies that $G^<(z) = 1 - \sqrt{1-z}$ and we conclude by Proposition \ref{thm:admiss} that $\nu$ is admissible.
	
	Let us look at the sequence $(y_i)_{i\geq0}$ of first coordinates of the points at which the simple random walk $(X_i)_{i\geq 0}$ intersects the diagonal when started at $(k,k)$.
	Clearly the increments of $(y_i)_{i\geq0}$ are independent and distributed according to the symmetric law $\nu$.
	A quick look at \eqref{eq:markovdown} reveals that the $1$-ricocheted random walk $(W^\downarrow_i,N_i^\downarrow)_{i\geq 0}$ is identical in law to $(|y_i|, n_i)_{i\geq 0}$ until $y_i=0$, where $n_i$ is the number of sign changes in $y_0,\ldots,y_i$.
	The corresponding ricochet sequence therefore agrees with the alternation sequence $(x_n)_{n\geq 0}$.
\end{proof}

\begin{remark}\label{rem:specialparams}
	The explicit measure $\nu$ appearing in the proof is one of the few symmetric admissible measures, which were already determined in \cite[Section 6.1]{budd_peeling_2015} in the more general non-bipartite setting.
	Here we have
	\begin{equation*}
	\nu(k) = \frac{2}{\pi}\frac{1}{4k^2-1} + \one_{\{k=0\}}.
	\end{equation*}
	Since $\nu(k-1) \geq \nu(-k-1)$ we will see in Section \ref{sec:proofadmiss} that there exists an admissible triple $(\qseq,g,n)$ given by 
	\begin{equation*}
	q_k = \left(\frac{\nu(-1)}{2}\right)^{k-1}(\nu(k-1)-\nu(-k-1)) = (3\pi)^{1-k} \frac{2k}{(k-\tfrac{3}{2})_4}+\one_{\{k=1\}},\quad g = \frac{1}{3\pi}, \quad n = 2,
	\end{equation*}
	such that the perimeter and nesting process of the pointed $(\qseq,g,n)$-Boltzmann loop-decorated map has the law of $(W_i^\downarrow,N_i^\downarrow)_{i\geq 0} \stackrel{\mathrm{(d)}}{=} (|y_i|,n_i)_{i\geq 0}$. 
\end{remark}

In particular, the total number of alternations of the simple random walk started at $(k,k)$ before hitting the origin is equal in law to the number $N^\downarrow_\infty$ of ricochets under $\probric^\downarrow_k$.
We will now show that this is all we need to know to study the distribution of the winding angle of $(X_i)$.

\begin{proof}[Proof of Theorem \ref{thm:winding}]
	Let $\theta_i$ be the winding angle around the origin of the walk $(X_i)$ up to time $i$.
	Conditionally on the number $N$ of alternations of $(X_i)$, the sequence $(\theta_{t_n})_{n=0}^{N}$ has the distribution of a simple random walk on $\pi\Z$ started at $0$.
	Indeed, in between alternation times $t_n$ and $t_{n+1}$ the walk travels from the positive to the negative diagonal or vice versa, incrementing the winding angle by $\theta_{t_{n+1}} - \theta_{t_n} = \pm\pi$. 
	By symmetry the increments $\pi$ and $-\pi$ occur with equal probability and independently of previous increments.
	Finally, the last section of the path from $(x_N,x_N)$ to the origin yields a contribution $\pm \tfrac{3}{4}\pi$ or $\pm \tfrac{1}{4}\pi$ with the sign again uniformly distributed.
	Hence,
	\begin{align*}
	\expec\left[ e^{i\beta [\Theta]}\middle|\, N \right] &=
	 \frac{e^{i\beta\pi/2}+e^{-i\beta\pi/2}}{2} \left(\frac{e^{i\beta\pi}+e^{-i\beta\pi}}{2}\right)^N\\
&= \cos\left(\tfrac{\pi\beta}{2}\right)\,\cos^N\left(\pi\beta\right)	.
	\end{align*}
	Using the relation to the number of ricochets and \eqref{eq:numricochets}, we thus find
	\begin{equation*}
	\expec\left[ e^{i\beta [\Theta]} \right] =\sum_{N=0}^\infty \cos\left(\tfrac{\pi \beta}{2}\right) \expecric_{k}^\downarrow\left[ \cos^{N_\infty^\downarrow}\left(\pi\beta\right)\right] = \cos\left(\tfrac{\pi \beta}{2}\right) h^\downarrow_{\cos(\pi \beta)}(k).
	\end{equation*} 
	
	Using Proposition \ref{thm:hdown}(iv) and the fact that $\mathsf{r}_\pr\to 1$ as $\pr\to 1$, we find for $k$ sufficiently large
	\begin{align*}
	\expec\left[\exp\left(i\frac{\beta[\Theta]}{\log k}\right)\right] &= \cos\left(\frac{\pi\beta}{2\log k}\right) h^\downarrow_{\cos\left(\frac{\pi\beta}{\log k}\right)}(k) \\
	&\stackrel{k\to\infty}{\sim} k^{-|\beta|/\log(k)} = e^{-|\beta|},
	\end{align*} 
	where we used that $\tfrac{1}{\pi}\arccos(\cos(\tfrac{\pi\beta}{\log k})) = |\beta|/\log k$.
	Since $\beta\to e^{-|\beta|}$ is the characteristic function of the standard Cauchy random variable $\mathcal{C}$, we obtain the convergence in distribution $[\Theta]/\log k\xrightarrow[k\to\infty]{\mathrm{(d)}} \mathcal{C}$.
	The same result holds with $[\Theta]$ replaced by $\Theta$, because $([\Theta]-\Theta)/\log k$ converges to $0$ in probability, 
\end{proof}

\section{Proof of Theorem \ref{thm:admissibility}}\label{sec:proofadmiss}

Suppose $\hqseq$ is admissible and $g\geq 0$, $n\in(0,2]$ are such that 
\begin{equation}\label{eq:qfromqhat}
q_k = \hat{q}_k - n\,g^{2k} W^{(k)}(\hqseq) \geq 0\qquad\text{for all }k\geq 1.
\end{equation}
Equivalently, $\hat{q}_k = (\frac{1}{2}\nu(-1))^{k-1}\nu(k-1)$ for an admissible law $\nu$ that satisfies for $k\geq 1$
\begin{equation}\label{eq:nuinequality}
\nu(k-1) \geq \frac{n}{2} \left(\frac{2g}{\nu(-1)}\right)^{2k} \nu(-k-1).
\end{equation}
In particular this implies that $g \leq \nu(-1)/2 = \gamma_\hqseq^{-2}$.

Our strategy will be to explicitly construct a random loop-decorated map that we will demonstrate to be distributed as a $(\qseq,g,n)$-Boltzmann loop-decorated map, where $\qseq$ is given by \eqref{eq:qfromqhat}.
To this end we rely on the fact already observed in Section \ref{sec:looppeeling} that one may specify a loop-decorated map $(\map,\loopconf)$ of perimeter $2p$ by fixing a (deterministic) peel algorithm $\mathcal{A}$ and providing a valid sequence of peeling events among $\mathsf{G}_{k_1,k_2},\mathsf{C}_k,\mathsf{L}_{k,k}$.

Let $(\emap_0,\loopconf_0)$ be the hollow map of perimeter $2p$.
Iteratively we construct the random sequence
\begin{equation}\label{eq:submapseq}
(\emap_0,\loopconf_0)\subset (\emap_1,\loopconf_1) \subset (\emap_2,\loopconf_2) \subset \cdots
\end{equation}
by letting $(\emap_{i+1},\loopconf_{i+1})=(\emap_{i},\loopconf_{i})$ if $\emap_i$ has no holes and otherwise $(\emap_{i+1},\loopconf_{i+1})$ is obtained from $(\emap_{i},\loopconf_i)$ by one of the peeling events.
If the peel edge $\mathcal{A}(\emap_{i})$ is incident to a hole of degree $2l$, we choose the event $\mathsf{E}\in\{\mathsf{G}_{k_1,k_2},\mathsf{C}_k,\mathsf{L}_{k,k}\}$ with to-be-determined probability $P_l(\mathsf{E})$ independently of everything else.
We make an educated guess for the probabilities $P_l(\mathsf{E})$ by expressing the corresponding probabilities \eqref{eq:peelprob} for the peeling process of $(\qseq,g,n)$-Boltzmann loop-decorated maps in terms of $\nu_\hqseq$ (assuming $(\qseq,g,n)$ is admissible).
This leads us to the choice
\begin{align}
P_l(\mathsf{G}_{k_1,k_2}) &=  \frac{\nu(-k_1-1)\nu(-k_2-1)}{2\nu(-l-1)} & \text{for }&k_1+k_2+1=l, k_1,k_2\geq 0,\nonumber\\
P_l(\mathsf{C}_{k}) &=  \left(\nu(k-1) - \frac{n}{2}\left(\frac{2g}{\nu(-1)}\right)^{2k}\nu(-k-1)\right) \frac{\nu(-l-k)}{\nu(-l-1)} & \text{for }&k\geq 1,\label{eq:constructprob}\\
P_l(\mathsf{L}_{k,k}) &=  \frac{n}{2}\left(\frac{2g}{\nu(-1)}\right)^{2k}\nu(-k-1) \frac{\nu(-l-k)}{\nu(-l-1)}& \text{for }&k\geq 1.\nonumber
\end{align}
These are nonnegative because of \eqref{eq:nuinequality}.
To see that they sum to one, one may observe that the sum is independent of $n$ and that for $n=0$ these are precisely the probabilities of the events $\mathsf{G}_{k_1,k_2}$ and $\mathsf{C}_k$ occurring in the peeling process of a $\hqseq$-Boltzmann planar map.
Since the events with non-vanishing probability are always valid, \eqref{eq:submapseq} is a well-defined random infinite sequence of loop-decorated maps with holes.

\begin{proposition}\label{thm:stabilize}
The sequence \eqref{eq:submapseq} stabilizes almost surely and therefore gives rise to a well-defined random loop-decorated map $(\map,\loopconf)$ of perimeter $2p$.
\end{proposition}

The proof will rely on the identification of an explicit superharmonic function for the process \eqref{eq:submapseq}. 
For a loop-decorated planar map $(\emap,\loopconf)$ with holes set
\begin{align}\label{eq:superharmonicfun}
V(\emap,\loopconf) = |\emap| + \sum_{h}f^{\downarrow}(\deg(h)/2)\quad\text{with}\quad f^\downarrow(p) \coloneqq \frac{\nu(-1)h^\downarrow_\pr(p)}{\nu(-p-1)},
\end{align}
where $|\emap|$ is the number of inner vertices of $\emap$, i.e. the vertices not incident to a hole, and we set $\pr = n/2$.

\begin{lemma}\label{thm:supermartingale}
We have
\begin{equation*}
\expec[ V(\emap_{i+1},\loopconf_{i+1}) | (\emap_i,\loopconf_i) ] \leq V(\emap_i,\loopconf_i)
\end{equation*}
with equality if $g = \nu(-1)/2$.
\end{lemma}
\begin{proof}
Let $(\emap_i,\loopconf_i)$ be fixed and suppose that the peel edge $\mathcal{A}(\emap_i)$ is incident to a hole $h$ of degree $2p$.
Then
\begin{equation*}
V(\emap_{i+1},\loopconf_{i+1})-V(\emap_i,\loopconf_i) = -f^\downarrow(p) +|\emap_{i+1}|-|\emap_i| + \sum_{h'} f^\downarrow(\tfrac12\deg(h')),
\end{equation*}
where the sum is over the (zero, one or two) holes $h'$ of $\emap_{i+1}$ that originate from the peeling operation on the hole $h$.
Depending on the event this equals 
\begin{equation*}
V(\emap_{i+1},\loopconf_{i+1})-V(\emap_i,\loopconf_i) = -f^\downarrow(p) + \begin{cases} 
f^\downarrow(k_1)+f^\downarrow(k_2) & \text{in case }\mathsf{G}_{k_1,k_2}, \\
f^\downarrow(p+k-1) & \text{in case }\mathsf{C}_k,\\
f^\downarrow(p+k-1) + f^\downarrow(k) & \text{in case }\mathsf{L}_{k,k},
\end{cases}
\end{equation*} 
where we used that $|\emap_{i+1}|\neq|\emap_i|$ only in case of $\mathsf{G}_{l-1,0}$ or $\mathsf{G}_{0,l-1}$ in which case the contribution $|\emap_{i+1}|-|\emap_i|$ is accounted for by the fact that $f^\downarrow(0)=1$.
Using the probabilities \eqref{eq:constructprob} and the definition \eqref{eq:superharmonicfun} we find
\begin{align*}
\expec&[ V(\emap_{i+1},\loopconf_{i+1})-V(\emap_i,\loopconf_i) | (\emap_i,\loopconf_i) ] = \\
&= -f^\downarrow(p) + 2 \sum_{k_1=0}^{p-1}\frac{\nu(k_1-p)\nu(-1)h^\downarrow_\pr(k_1)}{2\nu(-p-1)}+ \sum_{k\geq1}\frac{\nu(k-1)\nu(-1)h^\downarrow_\pr(p+k-1)}{\nu(-p-1)} \\
&\qquad\qquad + \pr \sum_{k\geq1}\left(\frac{2g}{\nu(-1)}\right)^{2k}\frac{\nu(-p-k)\nu(-1)h^\downarrow_\pr(k)}{\nu(-p-1)}\\
&= \frac{\nu(-1)}{\nu(-p-1)}\left(-h^\downarrow_\pr(p) + \sum_{k=0}^\infty h^\downarrow_\pr(k) \left( \nu(k-p) + \left(\frac{2g}{\nu(-1)}\right)^{2k}\pr\one_{\{k>0\}} \nu(-k-p) \right)\right).
\end{align*}
The claim then follows from \eqref{eq:hdownharm} and the fact that $g \leq \nu(-1)/2$.
\end{proof} 

\begin{proof}[Proof of Proposition \ref{thm:stabilize}]
One may check that there exists a $c>0$ such that $P_l(\mathsf{G}_{0,l-1}) > c$ uniformly in $l$. 
This means that, as long as there is at least one hole, at each step the number of inner vertices increases with a probability at least $c$.
Therefore, for any $j \geq 0$ we have
\begin{align}\label{eq:expecvertinequality}
f^{\downarrow}(p) = \expec[ V(\emap_0,\loopconf_0)] \geq \expec[ V(\emap_j,\loopconf_j)] \geq \expec[ |\emap_j| ] \geq j\, c\, \prob[ \emap_{j}\text{ has at least one hole}].
\end{align}
But this implies that the probability that the peeling exploration has not stabilized after $j$ steps is $O(j^{-1})$, hence it stabilizes after an almost surely finite time.
\end{proof}

\begin{lemma}\label{thm:constructprob}
	The probability of obtaining any particular loop-decorated map $(\map,\loopconf)$ of perimeter $2p$ in this way is 
	\begin{equation*}
	\frac{w_{\qseq,g,n}(\map,\loopconf)}{W^{(p)}(\hqseq)}.
	\end{equation*}
\end{lemma}
\begin{proof}
	A loop-decorated map $(\map,\loopconf)$ corresponds to a unique finite sequence $\mathsf{E}_1,\ldots,\mathsf{E}_m$ of events in $\{\mathsf{G}_{k_1,k_2},\mathsf{C}_k,\mathsf{L}_{k,k}\}$.
	The probability of such a sequence occurring can be deduced from \eqref{eq:constructprob} to be 
	\begin{equation}\label{eq:probseq1}
	\frac{\nu(-1)^{|\map|}}{\nu(-p-1)}\prod_{i=1}^m \begin{cases}
	\tfrac12 & \text{if }\mathsf{E}_i=\mathsf{G}_{k_1,k_2}, \\
	\nu(k-1) - \frac{n}{2}\left(\frac{2g}{\nu(-1)}\right)^{2k}\nu(-k-1) & \text{if }\mathsf{E}_i=\mathsf{C}_{k}, \\
	\frac{n}{2}\left(\frac{2g}{\nu(-1)}\right)^{2k} & \text{if }\mathsf{E}_i=\mathsf{L}_{k,k},
	\end{cases}
	\end{equation}
	where $|\map|$ is the number of vertices of $\map$.
	This can be seen by observing that all factors of $\nu(-k-1)$ with $k \geq 1$ appearing in the probabilities \eqref{eq:constructprob} cancel in the product, except for an overall $1/\nu(-p-1)$ in $P_p(\mathsf{E}_1)$.
	A factor of $\nu(-1)$ is produced by $\mathsf{G}_{0,k_2}$ or $\mathsf{G}_{k_1,0}$ for each vertex of $\map$.
	
	Since each step adds exactly one inactive edge that is not crossed by a loop, $m = |\mathsf{Edges}(\map)| - |\loopconf|$ where $|\loopconf|$ is the total length of all the loops.
	Among the $m$ events there is exactly one $\mathsf{C}_k$ for each internal face $f$ of $\map$ of degree $2k$ that is not visited by a loop, and one $\mathsf{L}_{k,k}$ for each loop of length $2k$.
	This allows us to rewrite \eqref{eq:probseq1} as
	\begin{equation*}
	\frac{\nu(-1)^{|\mathsf{Vertices}(\map)|}}{\nu(-p-1)} \left(\tfrac12\right)^{|\mathsf{Edges}(\map)|-|\loopconf|} \bigg[\prod_{f} 2 \left(\frac{\nu(-1)}{2}\right)^{1-\tfrac12\deg(f)}q_{\tfrac12\deg(f)}\bigg]\,\bigg[ \prod_{\ell\in\loopconf} n\left(\frac{2g}{\nu(-1)}\right)^{|\ell|}\bigg].
	\end{equation*}
	With the help of Euler's formula this reduces to
	\begin{equation*}
	\frac{2}{\nu(-p-1)}\left(\frac{\nu(-1)}{2}\right)^{p} \bigg[\prod_{f} q_{\tfrac12\deg(f)}\bigg]\,\bigg[ \prod_{\ell\in\loopconf} n\,g^{|\ell|}\bigg] = \frac{w_{\qseq,g,n}(\map,\loopconf)}{W^{(p)}(\hqseq)},
	\end{equation*}
	as claimed.
\end{proof}

Proposition \ref{thm:stabilize} and Lemma \ref{thm:constructprob} together imply that
\begin{equation*}
F^{(p)}(\qseq,g,n) = \sum_{(\map,\loopconf)\in\loopmaps^{(p)}} w_{\qseq,g,n}(\map,\loopconf) = W^{(p)}(\hqseq).
\end{equation*}
Hence, $(\qseq,g,n)$ is an admissible triple and our random loop-decorated map $(\map,\loopconf)$ is distributed as a $(\qseq,g,n)$-Boltzmann loop-decorated map of perimeter $2p$.
Moreover, it follows from \eqref{eq:expecvertinequality} that the expected number of vertices
of $\map$ is at most $f^\downarrow(p)$.
This finishes the proof of Theorem \ref{thm:admissibility}.

\subsection{Corollaries}

Let us discuss a few consequences.
The first one is direct:

\begin{corollary}
If $(\qseq,g,n)$ is admissible and $n\in[0,2]$ then $(\qseq,g,n)$ is strongly admissible in the sense of Definition \ref{def:strongadmissible}.
\end{corollary}

If $(\qseq,g,n)$ is non-generic critical, meaning that $g = \nu(-1)/2 = \gamma_\hqseq^{-2}$, then it follows from Lemma \ref{thm:supermartingale} that $f^\downarrow(p)$ is exactly the expected number of vertices in the $(\qseq,g,n)$-Boltzmann loop-decorated map. 
In that case
\begin{equation}\label{eq:fbulletexpr}
F^{(p)}_\bullet(\qseq,g,n) = f^\downarrow(p) F^{(p)}(\qseq,g,n) = \gamma_\hqseq^{2p}\, h^\downarrow_{\frac{n}{2}}(p),
\end{equation}
with $h^\downarrow_\pr(p)$ given explicitly in Proposition \ref{thm:hdown}.

\begin{proposition}\label{thm:expecvertices}
	If $(\qseq,g,n)\in \Ddomain$ is non-generic critical (dilute or dense) with exponent $\alpha$, then the number $|\map|$ of vertices of a $(\qseq,g,n)$-Boltzmann loop-decorated map $(\map,\loopconf)$ of perimeter $2p$ has expectation value
	\begin{equation}
	\expec^{(p)}[ |\map| ] = \frac{\gamma_\hqseq^{2p}h^\downarrow_{n/2}(p)}{F^{(p)}(\qseq,g,n)} \stackrel{p\to\infty}{\sim} C\,p^{\max(2,2\alpha-1)} 
	\end{equation}
	for some $C>0$.
\end{proposition}
\begin{proof}
The expression for the expectation value is given in \eqref{eq:fbulletexpr}, while Proposition \ref{thm:hdown}(iv) and \eqref{eq:Fasymp} provide the asymptotics.
\end{proof}

Comparing the transition probabilities \eqref{eq:perimlawnu} to \eqref{eq:markovdown} using with \eqref{eq:fbulletexpr} we may also conclude the following.

\begin{proposition}\label{thm:perimlaw}
If $(\qseq,g,n)$ is non-generic critical, then the perimeter and nesting process $(P_i,N_i)_{i\geq 0}$ of the pointed $(\qseq,g,n)$-Boltzmann loop-decorated map of perimeter $2p$ is equal in distribution to the $\frac{n}{2}$-ricocheted random walk $(W^\downarrow_i,N^\downarrow_i)_{i\geq 0}$ of law $\nu_\hqseq$ conditioned to be trapped at $0$ under $\probric^\downarrow_p$.
\end{proposition}

In particular, we may now wrap up Theorem \ref{thm:nesting}.

\begin{proof}[Proof of Theorem \ref{thm:nesting}]
With the help of Proposition \ref{thm:perimlaw} the results of Proposition \ref{thm:ricochets} on the number of ricochets directly translate into the same properties for the number of nested loops. 
\end{proof}

\section{Ricocheted stable processes}\label{sec:ricocheted-stable-processes}

\subsection{Definition}\label{sec:stabproc}
Let 
\begin{equation*}
\Adomain = \{(\theta,\rho) : \theta \in (0,1), \rho\in(0,1)\}\, \cup \,\{(\theta,\rho) : \theta \in (1,2), \rho\in(1-1/\theta,1/\theta)\}.
\end{equation*}
For $(\theta,\rho)\in\Adomain$ let $(S_t)_{t\geq 0}$ be the $\theta$-stable L\'evy process started at $0$ with positivity parameter $\rho$. 
More precisely, $(S_t)_{t\geq 0}$ is characterized by
\begin{equation*}
\expec\exp(i\lambda S_t)) = \exp(-t\psi(\lambda)),
\end{equation*}
with characteristic exponent
\begin{equation*}
\psi(\lambda) := e^{\pi i \theta(\tfrac12-\rho)} \lambda^\theta\one_{\{\lambda>0\}}+e^{-\pi i \theta(\tfrac12-\rho)} |\lambda|^\theta\one_{\{\lambda<0\}}.
\end{equation*}
In L\'evy--Khintchine form this corresponds to
\begin{equation*}
\psi(\lambda) = i a \lambda + \int_{-\infty}^{\infty} (1-e^{i\lambda y}+i \lambda y \one_{\{|y|<1\}})\pi(y)\rmd y
\end{equation*}
with L\'evy measure
\begin{equation*}
\pi(y) = c_+ y^{-\theta-1}\one_{\{y>0\}}+c_-|y|^{-\theta-1} \one_{\{y<0\}},
\end{equation*}
where 
\begin{equation*}
a = \frac{c_+-c_-}{\theta-1}, \quad
c_+ = \Gamma(1+\theta)\frac{\sin(\pi\theta\rho)}{\pi}, \quad c_- = \Gamma(1+\theta)\frac{\sin(\pi\theta(1-\rho))}{\pi}.
\end{equation*}
The infinitesimal generator $\mathcal{A}$ acting on twice differentiable measurable functions $f:\R\to\R$ defined as $\mathcal{A}f(x) = \lim_{t\to 0} \frac{1}{t}\expec(f(x+S_t)-f(x))$ is given by
\begin{equation}\label{eq:infgen}
\mathcal{A}f(x) = -a f'(x) + \int_{-\infty}^{\infty} ( f(x+y)-f(x) - y f'(x) \one_{\{|y|<1\}}) \pi(y) \rmd y.
\end{equation}

For $\pr\in[0,1]$ we define the \emph{$\pr$-ricocheted stable process} $(X^*_t)_{t\geq 0}$ to have infinitesimal generator $\mathcal{A}^*$ acting on twice-differentiable measurable functions $f: (0,\infty) \to \R$ given in terms of (\ref{eq:infgen}) for $x>0$ by 
\begin{equation*}
\mathcal{A}^*f(x) = \mathcal{A}f^*(x),\quad\quad
f^*(x) := f(x)\one_{\{x\geq 0\}} + \pr f(-x)\one_{\{x< 0\}}.
\end{equation*} 

\begin{lemma}\label{thm:infgenstar}
The infinitesimal generator $\mathcal{A}^*$ can be written as
\begin{align}
\mathcal{A}^* f(x) =& x^{-\theta}\int_{0}^{\infty}(f(u x)-f(x)-x f'(x)(u-1) \one_{\{|u-1|<1\}})\mu(u)\rmd u \nonumber\\
&+ \left(\pr -1\right)c_-\theta^{-1}x^{-\theta}f(x) + \left(\pr c_-\kappa-a\right) x^{1-\theta}f'(x), \label{eq:infgennast}
\end{align}
where
\begin{equation*}
\mu(u) := \pi(u-1) + \pr \pi(-u-1) = \pi(u-1) + \pr \,c_- (u+1)^{-\theta-1},
\end{equation*}
and $\kappa := \int_0^2(u-1)(u+1)^{-\theta-1}\rmd u$.
\end{lemma}
\begin{proof}
In the case $\pr=0$ the infinitesimal generator was computed in \cite[Theorem 2]{caballero_conditioned_2006} giving for $x>0$,\footnote{Be aware that our definitions of $a$ differ by a minus sign.}
\begin{align}
(\mathcal{A}f(\cdot)\one_{\{\cdot>0\}})(x) &= x^{-\theta}\int_0^{\infty}(f(u x)-f(x)-x f'(x)(u-1) \one_{\{|u-1|<1\}})\pi(u-1)\rmd u \nonumber\\
&\quad- c_- \theta^{-1} x^{-\theta} f(x) - a x^{1-\theta} f'(x) .\label{eq:infgen0}
\end{align}
On the other hand we find for $x>0$ that
\begin{align*}
(\mathcal{A}f(-\,\cdot)\one_{\{\cdot<0\}})(x) & = \int_{-\infty}^{-x}f(-x-y)\pi(y)\rmd y = x^{-\theta} \int_0^{\infty} f(ux)\pi(-u-1) \rmd u \\
&= x^{-\theta}\int_0^{\infty}(f(u x)-f(x)-x f'(x)(u-1) \one_{\{|u-1|<1\}})\pi(-u-1)\rmd u \nonumber\\
&\quad + c_-\theta^{-1} x^{-\theta} f(x) + c_-\kappa x^{1-\theta} f'(x),
\end{align*}
with $\kappa$ as given above. 
Putting both parts together we reproduce (\ref{eq:infgennast}).
\end{proof}

The process $(X^\ast_t)_{t\geq0}$ is an example of a \emph{positive self-similar Markov process} (pssMp) with index $\theta$, meaning that $(\lambda X^\ast_{\lambda^{-\theta}t})_{t\geq 0}$ started at $X^\ast_0=x$ is identical in law to $(X^\ast_t)_{t\geq 0}$ started at $\lambda x$ for any $\lambda>0$ (see \cite[Chapter 13]{kyprianou_fluctuations_2014} for a nice introduction).
It is a well-known result by Lamperti \cite{lamperti_semi-stable_1972} that any such pssMp $(X_t)_{t\geq 0}$ may be expressed as the exponential of a time-change of a (killed) L\'evy process $(\xi_s)_{s\geq0}$ started at $\xi_0=0$ in the following way.
Suppose $X_0=x$ and let $T_0 = \inf\{ t>0 : X_t = 0 \}$ (which may be $\infty$), then
\begin{equation*}
X_t = x \exp(\xi_{s(t x^{-\theta})}), \quad s(t) := \inf\{s\geq 0 : \int_0^s \exp(\theta \xi_u)\rmd u \geq t\}, \quad (0\leq t \leq T_0).
\end{equation*}
The L\'evy process $(\xi_s)_{s\geq0}$ is called the \emph{Lamperti representation} of the pssMp $(X_t)_{t\geq 0}$.

\begin{proposition}\label{thm:lampertistar}
The Lamperti representation $(\xi^*_s)_{s\geq0}$ of $(X^\ast_t)_{t\geq0}$ is given by the sum of
\begin{itemize}
\item the Lamperti representation $(\hat{\xi}_s)_{s\geq 0}$ of the $\theta$-stable process killed upon hitting $(-\infty,0)$ with Laplace exponent
\begin{equation}\label{eq:hatpsi}
\hat{\Psi}(z) := \log\expec\exp(z\hat{\xi}_1) =  \frac{1}{\pi}\Gamma(\theta-z)\Gamma(1+z)\sin(\pi z-\pi\theta(1-\rho));
\end{equation}
\item a compound Poisson process $(\xi^\diamond_s)_{s\geq 0}$ with rate $2\pr\, \Gamma(\theta)\sin(\pi \theta(1-\rho))/(2\pi)$ and Laplace exponent
\begin{equation*}
\Psi^\diamond(z) := \log\expec\exp(z \xi^\diamond_1) = \frac{\pr}{\pi}\Gamma(\theta-z)\Gamma(1+z)\sin(\pi\theta(1-\rho)).
\end{equation*}
\end{itemize}
With the notation $\sigma:= 1/2-\theta(1-\rho)\in(-1/2,1/2)$ and $b:=\frac{1}{\pi}\arccos(\pr \cos(\pi\sigma))\in[|\sigma|,1/2]$, the Laplace exponent $\Psi^\ast(z) = \hat{\Psi}(z)+\Psi^\diamond(z)$ of $(\xi^*_s)_{s\geq0}$ can be written in factorized form as
\begin{equation}\label{eq:psiast}
\Psi^\ast(z) = 2^{\theta} \frac{\Gamma\left(\frac{1+z}{2}\right)\Gamma\left(\frac{2+z}{2}\right)\Gamma\left(\frac{\theta-z}{2}\right)\Gamma\left(\frac{1+\theta-z}{2}\right)}{\Gamma\left(\frac{\sigma+b+z}{2}\right)\Gamma\left(\frac{\sigma-b+z}{2}\right)\Gamma\left(\frac{2-\sigma+b-z}{2}\right)\Gamma\left(\frac{2-\sigma-b-z}{2}\right)}.
\end{equation}
\end{proposition}
\begin{proof}
Comparing \cite[Theorem 6.1]{lamperti_semi-stable_1972} (see also \cite[Theorem 1]{caballero_conditioned_2006}) to $\mathcal{A}^*f(x)$ in (\ref{eq:infgennast}), one finds 
\begin{equation}\label{eq:levykhinast}
\Psi^*(i\lambda) = \left(\pr -1\right)c_-\theta^{-1} + i \lambda \left(\pr \,c_-\,\kappa-a\right) + \int_{-\infty}^{\infty}\left[e^{i\lambda y}-1-i \lambda(e^y-1)\one_{\{|e^y-1|<1\}}\right]\Pi^*(\rmd y),
\end{equation} 
where the L\'evy measure $\Pi^*(\rmd y) = \mu(e^y)e^y \rmd y$ is given by
\begin{equation}\label{eq:levymeasast}
\Pi^*(\rmd y) = \left(c_+ (e^y-1)^{-\theta-1}\one_{\{y>0\}} + c_- (1-e^{y})^{-\theta-1}\one_{\{y<0\}} + \pr c_- (e^y+1)^{-\theta-1}\right)e^y\rmd y.
\end{equation}
Clearly $\Psi^*$ splits into a part $\hat{\Psi}$ that is independent of $\pr$ and a part $\Psi^\diamond$ that is linear in $\pr$.
The first is simply the Laplace exponent of the $\theta$-stable process killed upon hitting $(-\infty,0)$, for which the explicit formula (\ref{eq:hatpsi}) can be found in \cite[Theorem 1]{kuznetsov_fluctuations_2010}.
The remaining part $\Psi^\diamond$ is 
\begin{align*}
\Psi^\diamond(i\lambda) &= \pr c_-\left(\theta^{-1} + i \lambda \kappa +\int_{-\infty}^{\infty}\left[e^{i\lambda y}-1-i \lambda(e^y-1)\one_{\{|e^y-1|<1\}}\right]\frac{e^y\rmd y}{(e^y+1)^{\theta+1}}\right)\nonumber\\
&=\pr c_-\int_{-\infty}^{\infty}e^{i\lambda y}\frac{e^y\rmd y}{(e^y+1)^{\theta+1}} = \pr c_- \int_0^1 x^{\theta-i\lambda-1}(1-x)^{i\lambda} \rmd x \nonumber\\
&= \pr c_-\frac{\Gamma(\theta-i\lambda)\Gamma(1+i\lambda)}{\Gamma(\theta+1)} = \frac{\pr}{\pi} \Gamma(\theta-i\lambda)\Gamma(1+i\lambda) \sin(\pi\theta(1-\rho)).
\end{align*}
From the L\'evy--Khintchine representation of $\Psi^\diamond(i\lambda)$ one sees that it must correspond to a compound Poisson process with jumps arriving at a finite rate $\Psi^\diamond(0)$.

Adding both parts we get
\begin{align*}
\Psi^*(z) &= \frac{1}{\pi}\Gamma(\theta-z)\Gamma(1+z)\left[\cos(\pi b)-\cos(\pi (z+\sigma))\right]\\
&= -\frac{2}{\pi} \Gamma(\theta-z)\Gamma(1+z) \sin\left(\frac{\pi}{2}(b+\sigma+z)\right)\sin\left(\frac{\pi}{2}(b-\sigma-z)\right).
\end{align*}
Using the doubling formula $\Gamma(2x) = \frac{1}{\sqrt{\pi}}2^{2x-1}\Gamma(x)\Gamma(x+1/2)$ and the reflection formula $\Gamma(1-x)\Gamma(x)=\pi/\sin(\pi x)$ one may recover (\ref{eq:psiast}).
\end{proof}

One may easily check that for all allowed parameters $(\theta,\rho)\in\mathbb{A}$ and $\pr\in[0,1]$, the Laplace exponent $\Psi^*$ is analytic on $(-1,\theta)$ and has precisely two zeros at $z=\pm b-\sigma$, except when $\pr=1$ and $\rho=1-1/(2\theta)$, in which case both zeros coincide.
For now, we will exclude the latter scenario, and assume that $\pr<1$ when $\rho=1-1/(2\theta)$.
Then the functions $\hdownc,\hupc : (0,\infty)\to\R$ given by
\begin{equation}\label{eq:hdownc}
\hdownc(x) := x^{-b-\sigma} \quad\text{and}\quad \hupc(x) := x^{b-\sigma}
\end{equation}
are $\mathcal{A}^\ast$-harmonic on $(0,\infty)$.
This allows us to introduce the corresponding $h$-transforms $(X^\downarrow_t)_{t\geq 0}$ and $(X^\uparrow_t)_{t\geq 0}$ with generators
\begin{equation}\label{eq:infgen2}
\mathcal{A}^\downarrow f(x) = \frac{1}{\hdownc(x)}(\mathcal{A}^\ast \hdownc f)(x),\quad\quad\mathcal{A}^\uparrow f(x) = \frac{1}{\hupc(x)}(\mathcal{A}^\ast \hupc f)(x),
\end{equation}
which are again pssMp's with index $\theta$.
We denote their Lamperti representations by $(\xi^\downarrow_t)_{t\geq0}$ and $(\xi^\uparrow_t)_{t\geq0}$ respectively.
The corresponding Laplace exponents are then simply given by shifting $\Psi^*(z)$, i.e.
\begin{align}
\Psi^\downarrow(z) &= \Psi^*(z-b-\sigma) = \frac{1}{\pi}\Gamma(\theta+b+\sigma-z)\Gamma(1-b-\sigma+z)(\cos(\pi b)-\cos(\pi(z-b)) \nonumber\\
&=2^{\theta} \frac{\Gamma\left(\frac{1-b-\sigma+z}{2}\right)\Gamma\left(\frac{2-b-\sigma+z}{2}\right)\Gamma\left(\frac{\theta+b+\sigma-z}{2}\right)\Gamma\left(\frac{1+\theta+b+\sigma-z}{2}\right)}{\Gamma\left(\frac{z}{2}\right)\Gamma\left(\frac{-2b+z}{2}\right)\Gamma\left(\frac{2+2b-z}{2}\right)\Gamma\left(\frac{2-z}{2}\right)},\label{eq:psidown}\\
\Psi^\uparrow(z) &= \Psi^*(z+b-\sigma) = \frac{1}{\pi}\Gamma(\theta-b+\sigma-z)\Gamma(1+b-\sigma+z)(\cos(\pi b)-\cos(\pi(z+b)) \nonumber\\
&=2^{\theta} \frac{\Gamma\left(\frac{1+b-\sigma+z}{2}\right)\Gamma\left(\frac{2+b-\sigma+z}{2}\right)\Gamma\left(\frac{\theta-b+\sigma-z}{2}\right)\Gamma\left(\frac{1+\theta-b+\sigma-z}{2}\right)}{\Gamma\left(\frac{z}{2}\right)\Gamma\left(\frac{2b+z}{2}\right)\Gamma\left(\frac{2-2b-z}{2}\right)\Gamma\left(\frac{2-z}{2}\right)}.\label{eq:psiup}
\end{align}
The corresponding L\'evy measures are $\Pi^\downarrow(\rmd y) = e^{-(b+\sigma)y}\Pi^*(\rmd y)$ and $\Pi^\uparrow(\rmd y) = e^{(b-\sigma)y}\Pi^*(\rmd y)$.

\begin{figure}[h]
	\centering
	\includegraphics[width=.65\linewidth]{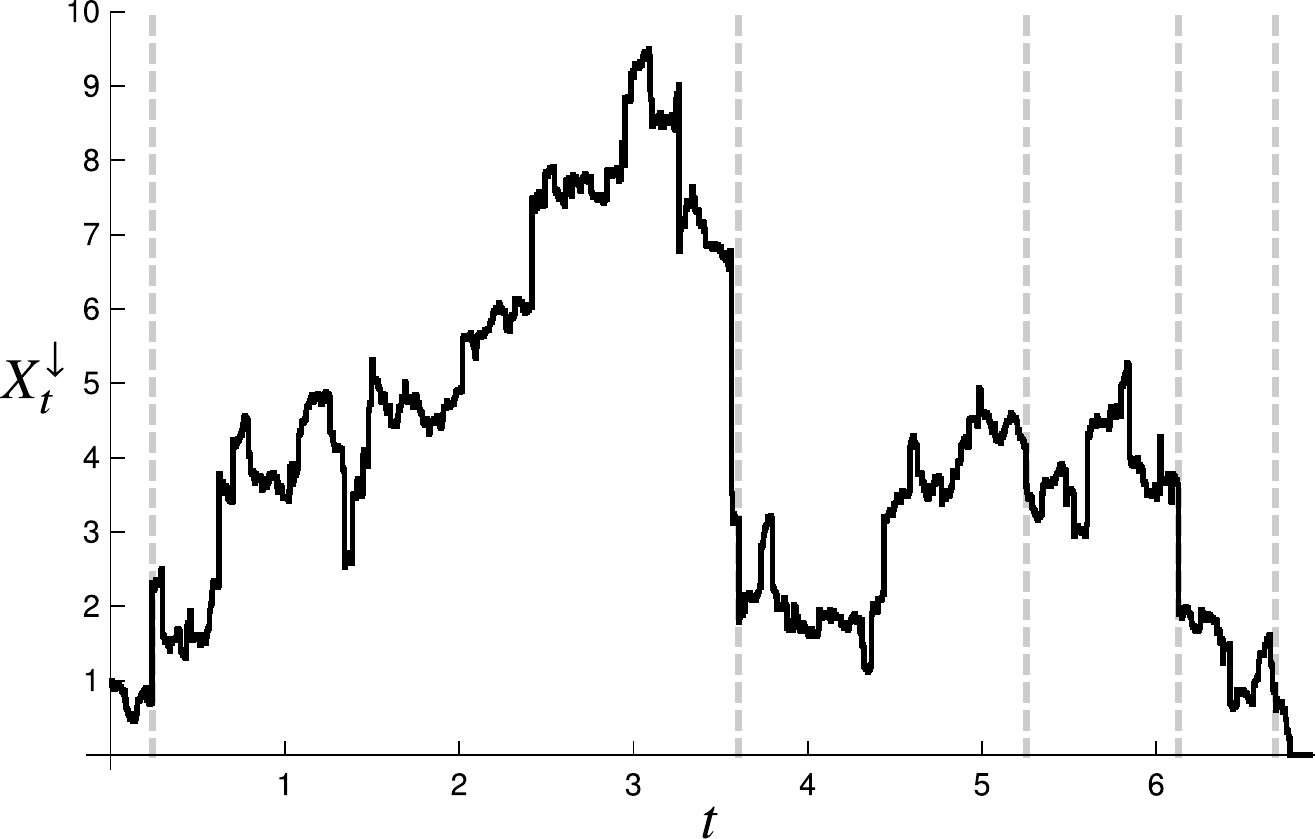}
	\caption{ A simulation of the $\pr$-ricocheted stable process $(X_t^\downarrow)_{t\geq 0}$ conditioned to die continuously at zero for $\sigma=0$, $\theta = 1+b = 1.4$. The dashed lines indicate times at which ricochets happen, i.e. jumps associated to the compound Poisson process $(\xi^\diamond_s)$. In reality there is an infinite number of them concentrating near the endpoint, but here only the first few are shown.
		\label{fig:ricochetproc}}
\end{figure}

Since both satisfy $\Psi^\uparrow(0)=\Psi_\downarrow(0)=0$, $(\xi^\downarrow_s)_{s\geq 0}$ and $(\xi^\uparrow_s)_{s<0}$  are L\'evy processes with infinite life times.
Moreover, $\expec \xi^\downarrow_1 = ({\psi^\downarrow})'(0) < 0$ and $\expec \xi^\uparrow_1 = ({\psi^\uparrow})'(0) > 0$, meaning that $(\xi^\downarrow_s)_{s\geq0}$ and $(\xi^\uparrow_s)_{s\geq0}$ drift respectively to $-\infty$ and $+\infty$ almost surely.
This implies that the pssMp $(X^\downarrow_t)_{t\geq0}$ hits zero continuously, while the pssMp $(X^\uparrow_t)_{t\geq0}$ never hits zero.
By convention we will set $X_t^\downarrow = 0$ after it has hit the origin.

\subsection{Scaling limit of the ricocheted random walk}\label{sec:scalinglimitwalk}

Of course, we introduced the ricocheted stable processes to serve as the scaling limits of the ricocheted random walks discussed in Section \ref{sec:ricochetedwalk}.
Although we expect such scaling limits to hold for quite general ricocheted random walks, we will restrict our attention to the special class of random walks that appear in the perimeter process of non-generic critical loop-decorated maps.
Let us make this precise and say $\nu$ is \emph{non-generic critical of index $\theta\in(\tfrac12,\tfrac32)$} if $\nu$ is admissible, the random walk associated to $\nu$ oscillates, and $\nu$ has tails $\nu(-k) \sim c\,k^{-\theta-1}$ and $\nu(k) \sim -c\,\cos(\pi \theta)\,k^{-\theta-1}$ as $k\to\infty$ for some $c>0$.
A random walk of such a law $\nu$ is in the domain of attraction of a $\theta$-stable process with $c_+/c_- = -\cos(\pi\theta)$, which is equivalent to a positivity parameter $\rho = 1 - 1/(2\theta)$.
In the remainder of the paper we therefore restrict to processes of such form, and thus set $\sigma = 0$.
To summarize we have the following identifications,
\begin{equation*}
\rho = 1-1/(2\theta), \quad \sigma = 0,\quad c_- = \frac{1}{\pi}\Gamma(\theta+1), \quad c_+ = -\cos(\pi\theta)c_-,\quad b = \tfrac{1}{\pi}\arccos(\pr).
\end{equation*}

\begin{proposition}\label{thm:conv}
Suppose $\nu$ is non-generic critical of index $\theta\in(1/2,3/2)$ and $\pr\in(0,1)$. 
Let $(W^\downarrow_i)_{i\geq 0}$ under $\probric^\downarrow_p$ be the associated $\pr$-ricocheted random walk of law $\nu$ started at $p$ and conditioned to be trapped at zero.
Then the following limit holds in distribution in the Skorokhod topology as $p\to\infty$,
\begin{equation}\label{eq:ldownconv}
\left( \frac{W^\downarrow_{\lfloor C t p^{\theta}\rfloor}}{p} \right)_{t\geq 0} \xrightarrow[p\to\infty]{(\mathrm{d})} (X^\downarrow_t)_{t\geq0},
\end{equation}
where $(X^\downarrow_t)_{t\geq0}$ is the $\pr$-ricocheted stable process of Section \ref{sec:stabproc} with parameters $(\theta,\rho=1-1/(2\theta),\pr)$ started at $X^\downarrow_0 = 1$, and $C \coloneqq \Gamma(1+\theta)/( \pi \lim_{k\to\infty}k^{\theta+1}\nu(-k))$.
\end{proposition}

The proof of Proposition \ref{thm:conv} relies on a recently established invariance principle \cite{bertoin_self-similar_2016}, which requires a number of assumptions to be verified. 
The main ingredient in this direction is the following Lemma.

\begin{lemma}\label{thm:infgenconv}
Under the assumptions of Theorem \ref{thm:conv} we have for any twice-differentiable function $f:(0,\infty)\to\R$ with compact support the convergence
\begin{equation}\label{eq:infgenconv}
C\,p^\theta \,\expecric^\downarrow_p\left[ f\left(\frac{W_1^\downarrow}{p}\right) -f(1)\right] \xrightarrow{p\to\infty} \mathcal{A}^\downarrow f(1),
\end{equation}
where $\mathcal{A}^\downarrow$ is the infinitesimal generator \eqref{eq:infgen2}.
\end{lemma}
\begin{proof}
Let $(W_i)$ under $\probric_p$ be a random walk with law $\nu$ started at $p$. 
Its scaling limit has been determined in \cite[Proposition 3.2]{budd_geometry_2017} for $\theta\neq 1$ and in \cite[Proposition 2]{budd_infinite_2017} for $\theta=1$ (see also \cite[Proposition 6.6]{bertoin_martingales_2017}),
\begin{equation*}
\left( \frac{W_{\lfloor C t p^\theta\rfloor}}{p} \right)_{t\geq 0} \xrightarrow[p\to\infty]{(\mathrm{d})} (S_t)_{t\geq 0},
\end{equation*} 
where $(S_t)_{t\geq 0}$ is the stable process with parameters $(\theta,\rho=1-1/(2\theta))$ started at $S_0=1$.

In particular, it follows (e.g. from \cite{skorokhod_limit_1957}) that for every twice-differentiable function $g:\R\to\R$ with compact support that  we have
\begin{equation}\label{eq:splim}
C\, p^\theta \sum_{k=-\infty}^\infty \left[g\left(x+\frac{k}{p}\right)-g(x)\right] \nu(k) \xrightarrow{p\to\infty} \mathcal{A}g(x),
\end{equation}
where $\mathcal{A}$ is the infinitesimal generator (\ref{eq:infgen}) of $(S_t)_{t\geq0}$.

Let $f:(0,\infty)\to\R$ be a twice-differentiable function with compact support. 
The left-hand side of \eqref{eq:infgenconv} reads
\begin{align}
C\, p^\theta\,&\expecric^\downarrow_p\left[ f\left(\frac{W_1^\downarrow}{p}\right) - f(1) \right] \nonumber\\
&=C\, p^\theta\left[-f(1)+\sum_{k=-p}^{\infty} \frac{h_\pr^\downarrow(p+k)}{h_\pr^\downarrow(p)} f\left(1+\frac{k}{p}\right) \left(\nu(k) + \pr\nu(-2p-k)\one_{\{k> -p\}}\right)\right].\label{eq:limsum0}
\end{align}
From the asymptotic behaviour \eqref{eq:hdownasymp} it follows that
\begin{equation*}
\frac{h_n^\downarrow(p+k)}{h_n^\downarrow(p)}\left(1+k/p\right)^{b} \xrightarrow[p\to\infty]{} 1
\end{equation*}
uniformly in $k$ for which $1+k/p$ is in the support of $f$.
Hence, as $p\to\infty$,
\begin{align}
C\, p^\theta\,&\expecric^\downarrow_p\left[ f\left(\frac{W_1^\downarrow}{p}\right) - f(1) \right]\nonumber\\
&\sim C\, p^\theta\left[-f(1)+\sum_{k=-p}^{\infty}\left(1+\frac{k}{p}\right)^{-b} f\left(1+\frac{k}{p}\right) \left(\nu(k) + \pr\nu(-2p-k)\one_{\{k> -p\}}\right)\right]\nonumber\\
&\sim C\, p^\theta \sum_{k=-\infty}^\infty \left[f^\downarrow\left(1+\frac{k}{p}\right)-f^\downarrow(1)\right]\nu(k),\label{eq:fdownconv}
\end{align}
where we defined
\begin{equation*}
f^\downarrow(x) := \hdownc(x) f(x) \one_{\{x\geq 0\}} + \pr\, \hdownc(-x)f(-x) \one_{\{x<0\}}
\end{equation*}
and $\hdownc(x) = x^{-b}$ as in \eqref{eq:hdownc}.
Since $f^\downarrow:\R\to\R$ is twice-differentiable and has compact support, we can apply (\ref{eq:splim}) to $f^\downarrow$ at $x=1$ to obtain
\begin{equation*}
C\, p^\theta\,\expecric^\downarrow_p\left[ f\left(\frac{W_1^\downarrow}{p}\right) - f(1) \right] \xrightarrow{p\to\infty} \mathcal{A}f^\downarrow(1).
\end{equation*}
The result \eqref{eq:infgenconv} then follows by noticing that $\mathcal{A}f^\downarrow(1) = \hdownc(1)^{-1}\mathcal{A}^\ast(\hdownc f)(1) = \mathcal{A}^\downarrow f(1)$.
\end{proof}

\begin{proof}[Proof of Proposition \ref{thm:conv}]
We wish to use \cite[Theorem 2]{bertoin_self-similar_2016} to prove the convergence \eqref{eq:ldownconv} to the self-similar Markov process whose Lamperti representation $\xi^\downarrow$ has Laplace exponent $\Psi^\downarrow(z)$ given in \eqref{eq:psidown}.
The Laplace exponent has L\'evy-Khintchine form
\begin{equation*}
\Psi^\downarrow(z) = b^\downarrow z + \int_{-\infty}^{\infty} \left(e^{z y} - 1 - zy\one_{\{|y|\leq 1\}}\right) \Pi^\downarrow(\rmd y),
\end{equation*}
where we recall that $\Pi^\downarrow(\rmd y) = e^{-by}\Pi^*(\rmd y) = e^{(1-b)y}\mu(e^y)\rmd y$ and $b^\downarrow$ is some explicit constant that one can deduce from \eqref{eq:levymeasast}. 
In order to use \cite[Theorem 2]{bertoin_self-similar_2016} we need to check three assumptions, (A1)-(A3).

From Lemma \ref{thm:infgenconv} and Lemma \ref{thm:infgenstar} we deduce that for any twice-differentiable function $f:(0,\infty)\to\R$ with compact support in $(0,1)\cup(1,\infty)$,
\begin{equation}\label{eq:A3condition}
C\,p^\theta \expecric^\downarrow_p\left[ f\left(\frac{W_1^\downarrow}{p}\right) -f(1)\right] \xrightarrow{p\to\infty} \mathcal{A}^\downarrow f(1) = \int_{0}^\infty \hdownc(u) f(u)\mu(u) \rmd u = \int_\R f(e^y)\Pi^\downarrow(\rmd y).
\end{equation}
This is assumption (A1) of \cite{bertoin_self-similar_2016}, except that it should be checked for continuous instead of twice-differentiable functions. This can be done by directly comparing the large-$p$ limit of \eqref{eq:fdownconv} to the right-hand side of \eqref{eq:A3condition}.
Assumption (A2) on the other hand is seen to be equivalent to Lemma \ref{thm:infgenconv} with $f$ taken such that $f(z) = \log(z)$ respectively $f(z)=\log^2(z)$ for $z$ in some small neighbourhood of $1$.
Finally assumption (A3) follows from the fact that for $\beta\in(0,b)$ we have as $p\to \infty$,
\begin{align*}
p^\theta \expecric^\downarrow_p\left[ \left(\frac{W^\downarrow_1}{p}\right)^\beta \one_{\{\log(W^\downarrow_1/p)>1\}}\right] &\lesssim p^\theta \sum_{k=2p}^\infty \frac{h^\downarrow(k)}{h^\downarrow(p)} \left(\frac{k}{p}\right)^\beta \nu(k-p) \\
&\lesssim p^\theta \sum_{k=2p}^\infty \left(\frac{k}{p}\right)^{\beta-b} (k-p)^{-\theta-1} \lesssim 1.
\end{align*}
Since $\xi^\downarrow$ drifts to $-\infty$, we can apply \cite[Theorem 2]{bertoin_self-similar_2016} to obtain the convergence (\ref{eq:ldownconv}) of $(W^\downarrow_i)_{i\geq 0}$.
\end{proof}

\begin{proof}[Proof of Theorem \ref{thm:perimeterscaling}]
Suppose $(\qseq,g,n)\in\Ddomain$ is admissible and non-generic critical with exponent $\alpha$.
According to Proposition \ref{thm:perimlaw} the perimeter process $(P_i)_{i\geq0}$ of the $(\qseq,g,n)$-Boltzmann loop-decorated map with perimeter $2p$ agrees in law with the $\tfrac{n}{2}$-ricocheted random walk $(W_i^\downarrow)_{i\geq0}$ of law $\nu_\hqseq$ under $\probric^\downarrow_p$.

Recall from Section \ref{sec:Onintro} that in the non-generic critical phase $g\gamma_\hqseq^{2}=1$ and $\hqseq$ is critical, which implies that the random walk with law $\nu_\hqseq$ oscillates \cite[Proposition 4]{budd_peeling_2015}. 
Since $\qseq$ has finite support, \eqref{eq:effq} together with \eqref{eq:Fasymp} furthermore shows that $\nu_\hqseq(-k) \sim c k^{-\alpha-1/2}$ and $\nu_\hqseq(k) \sim -c\,\cos(\pi (\alpha-1/2))k^{-a}$ for some $c>0$. 
Hence, we are precisely in the setting of Proposition \ref{thm:conv} with $\theta= \alpha -1/2$, $\pr = n/2$, and $b = |\alpha-3/2|$.
\end{proof}

\subsection{Comments on the ricocheted stable process conditioned to survive}\label{sec:ricochetedsurvive}

We have seen that the ricocheted stable process $(X_t^\downarrow)$ conditioned to die continuously appears as the scaling limit of the perimeter process of a loop-decorated map with a marked vertex.
Let us discuss on a heuristic level how the other process $(X_t^\uparrow)$, the ricocheted stable process conditioned to survive, appears in the context of planar maps. 

Let $(\qseq,g,n)$ be admissible and in the non-generic critical (dilute or dense) phase and let $(\map,\loopconf)\in\loopmaps^{(p)}$ be a $(\qseq,g,n)$-Boltzmann loop-decorated map. 
For any sufficiently large integer $k$ there is a nonzero probability that $\map$ has precisely $k$ vertices, $|\map|=k$.
Therefore we may consider the conditional probability distributions $\prob(\,\cdot\, | \,|\map| = k )$ on $(\map,\loopconf)$ as $k\to\infty$. 
It is expected that these converge weakly in an appropriate local sense to a random infinite loop-decorated map $(\map_\infty,\loopconf)$ (see \cite{bjornberg_recurrence_2014} for a proof in the case $n=0$).
 
Just like the pointed loop-decorated map, $(\map_\infty,\loopconf)$ admits a targeted peeling exploration in which after each peeling step the holes are filled in that contain only a finite region of $\map_\infty$. 
Assuming the local limit exists, one may work out explicitly the law of the corresponding perimeter process of $(\map_\infty,\loopconf)$.
It is again distributed as an $h$-transform $(W^\uparrow_i,N^\uparrow_i)_{i\geq 0}$ of the $(\pr=\frac{n}{2})$-ricocheted random walk $(W^*_i,N^*_i)_{i\geq 0}$, meaning that it has transition probabilities of the form \eqref{eq:markovdown}, except that $h^\downarrow_\pr$ is replaced by a new harmonic function $h^\uparrow_\pr$.
The latter is defined through the generating function
\begin{equation*}
\sum_{p=1}^\infty h^\uparrow_\pr(p) x^{2p} = \frac{1}{4b} \sinh(2b\arctanh x)\sinh( 2\arctanh x),
\end{equation*}  
where as usual $b = \frac{1}{\pi}\arccos \pr$. 
Explicitly, $h^\uparrow_0(p) = 2p\, h^\downarrow_0(p)$, $h^\uparrow_1(p) = \sum_{k=1}^p\frac{1}{2k-1}$ and for $\pr\in(0,1)$,
\begin{equation*}
h^\uparrow_\pr(p) = \frac{1}{4b} \frac{\Gamma(b)}{\Gamma(2p)\Gamma(b-2p+1)} {_2F_1}(-2p,b;b-2p+1;-1).
\end{equation*}
The Markov process $(W^\uparrow_i,N^\uparrow_i)_{i\geq 0}$ has the interpretation as the weak limit (in the sense of finite-dimensional marginals) of the process $(W^*_i,N^*_i)_{i\geq 0}$ conditioned to not be trapped for a long time, and thus deserves to be called the ricocheted random walk conditioned to survive.

Proposition \ref{thm:conv} can then be straightforwardly adapted (this time relying on \cite[Theorem 1]{bertoin_self-similar_2016} and the fact that $h_\pr^\uparrow \sim c\, p^b$ as $p\to\infty$) to show the convergence
\begin{equation*}
\left( \frac{W^\uparrow_{\lfloor C t p^{\theta}\rfloor}}{p} \right)_{t\geq 0} \xrightarrow[p\to\infty]{(\mathrm{d})} (X^\uparrow_t)_{t\geq0},
\end{equation*}
in the Skorokhod topology as $p\to\infty$, where $(X^\uparrow_t)_{t\geq0}$ is the $\pr$-ricocheted stable process of index $\theta=\alpha-1/2$ started at $1$ and conditioned to survive.
Actually one expects the fixed perimeter-$p$ scaling limit
\begin{equation*}
\left( \frac{W^\uparrow_{\lfloor C t \lambda^\theta\rfloor}}{\lambda}\right)  \xrightarrow[\lambda\to\infty]{(\mathrm{d})} (\tilde{X}^\uparrow_t)_{t\geq0}
\end{equation*}
to hold, where $(\tilde{X}^\uparrow_t)_{t\geq0}$ corresponds to the pssMp $(X^\uparrow_t)_{t\geq0}$ with its starting point $x$ taken to zero.
That $(\tilde{X}^\uparrow_t)_{t\geq0}$ is well-defined, meaning that it has a self-similar entrance law, follows from \cite[Theorem 1]{bertoin_entrance_2002}.
However, proving such convergence requires a stronger invariance principle than the one in \cite{bertoin_self-similar_2016} and is beyond the scope of this paper.

\section{The asymptotics of the FPP distance}\label{sec:the-asymptotics-of-the-fpp-distance}

The goal of this section is to prove Theorem \ref{thm:fpp} concerning the scaling limit of the distance between the root face and a marked vertex in the sense of first passage percolation.
Let $(\qseq,g,n)$ be admissible and in the non-generic critical and dilute phase and $(\map_\bullet,\loopconf) \in \loopmaps_\bullet^{(p)}$ a $(\qseq,g,n)$-Boltzmann loop-decorated map of perimeter $2p$.
We have already identified the law of the first-passage percolation distance in terms of the perimeter process of $(\map_\bullet,\loopconf)$ in \eqref{eq:fppperim}.
Combining with Proposition \ref{thm:perimlaw} we thus have 
\begin{equation}\label{eq:bpmfppid}
\hat{d}_{\mathrm{fpp}}(\map_\bullet,\loopconf) \overset{\mathrm{(d)}}{=} \sum_{i=0}^\infty \frac{\mathbf{e}_i}{2P_i}\one_{\{P_i>0\}} \overset{\mathrm{(d)}}{=} \sum_{i=0}^\infty \frac{\mathbf{e}_i}{2W^\downarrow_i}\one_{\{W^\downarrow_i>0\}}.
\end{equation}
To establish the convergence in distribution of $p^{-b}\hat{d}_{\mathrm{fpp}}(\map_\bullet,\loopconf)$ we will use the scaling limit of $(W_i^\downarrow)_{i\geq0}$ established in Proposition \ref{thm:conv} together with a first-moment bound on $\hat{d}_{\mathrm{fpp}}(\map_\bullet,\loopconf)$.
Most of the effort in this section is devoted to the latter.

\subsection{Expected first passage percolation distance}

Suppose $\nu$ is non-generic critical of index $\theta=1+b \in (1,3/2)$ in the sense of Section \ref{sec:scalinglimitwalk} and $\pr=\cos(\pi b)$.
Let $(W_i^\downarrow)_{i\geq0}$ be the $\pr$-ricocheted random walk conditioned to be trapped at $0$ started at $p$ under $\probric^\downarrow_p$.
We will prove the following estimate.

\begin{proposition}\label{thm:expdfpp}
	There exists a $C>0$, depending on the law $\nu$ and $\pr$, such that for all $p\geq0$,
	\begin{equation*}
	\expecric^\downarrow_p \left[ \sum_{i=0}^\infty \frac{1}{W^\downarrow_i} \one_{\{W^\downarrow_i > 0\}}\right] \leq C\, p^b.
	\end{equation*}
\end{proposition}

It is convenient to decompose the contributions into the parts in between ricochets. 
Recalling the definition of the ricochet sequence $(\ell^\downarrow_n)_{n\geq 0}$ from section \ref{sec:ricochetedwalk}, we may write 
\begin{equation}\label{eq:dfppD}
\expecric^\downarrow_p \left[ \sum_{i=0}^\infty \frac{1}{W^\downarrow_i} \one_{\{W^\downarrow_i > 0\}}\right] = \sum_{n=0}^\infty \expecric^\downarrow_p \left[ \sum_{i=0}^\infty \frac{1}{W^\downarrow_i} \one_{\{W^\downarrow_i > 0, N^\downarrow_i=n\}}\right] = \sum_{n=0}^\infty \expecric^\downarrow_p \left[ \mathcal{D}_{\ell^\downarrow_n,\ell^\downarrow_{n+1}} \one_{\{N^\downarrow_\infty \geq n\}}\right], 
\end{equation}
where $\mathcal{D}_{p,l}$ with $p\geq 1$, $l\geq 0$ is the following conditional expectation value for the random walk $(W_i)$ under $\probric_p$,
\begin{equation}\label{eq:Dpldef}
\mathcal{D}_{p,l} \coloneqq \expecric_p\left[ \sum_{i=0}^\infty \frac{1}{W_i} \one_{\{W_0,\ldots,W_i > 0\}}  \middle| (W_i)_{i\geq 0}\text{ hits $\Z_{\leq 0}$ at }-l\right].
\end{equation}

The quantity $\mathcal{D}_{p,l}$ can be interpreted in the loop-decorated map as the expected distance between a pair of consecutive nested loops with conditioned lengths $p$ and $l$.
The combinatorial setup suggests that it is symmetric in $p$ and $l$. 
Let us prove this fact using just the random walk.

\begin{lemma}\label{thm:Dsym}
	$\mathcal{D}_{p,l} = \mathcal{D}_{l,p}$ for $p,l\geq 1$.
\end{lemma}
\begin{proof}
	The conditioning in \eqref{eq:Dpldef} corresponds to an $h$-transform with respect to $H_l$ as can be seen from Lemma \ref{thm:Hlpproof}.
	Therefore 
	\begin{equation}\label{eq:dfppbpm2}
	\mathcal{D}_{p,l} = \sum_{i=0}^\infty\sum_{k=1}^{\infty}\frac{1}{k} \frac{H_l(k)}{H_l(p)} \probric_p[ W_1,\ldots,W_{i-1} > 0, W_i = k].
	\end{equation}
	Using that $H_l(k)/k = H_k(l)/l = \frac{1}{l}\probric_l((W_i)_i\text{ hits }\Z_{\leq0}\text{ at }-k)$ and the duality
	\begin{equation*}
	\probric_p[ W_1,\ldots,W_{i-1} > 0, W_i = k] = \probric_{-k}[ W_1,\ldots,W_{i-1} < 0, W_i = -p],
	\end{equation*}
	this becomes
	\begin{align*}
	\mathcal{D}_{p,l} &= \frac{1}{lH_l(p)} \sum_{k=1}^\infty \probric_l((W_i)_i\text{ hits }\Z_{\leq0}\text{ at }-k)	\sum_{i=0}^\infty \probric_{-k}[ W_1,\ldots,W_{i-1} < 0, W_i = -p]\\
	&=\frac{l+p}{lp}\frac{1}{H_0(l)H_0(p)}\sum_{j=0}^\infty \probric_l\left[ W_j=-p, (W_i)_{i=0}^j\text{ changes sign exactly once}\right].
	\end{align*}
	The probabilities on the right-hand side are easily seen to be symmetric in $p$ and $l$ by reversing the increments.
\end{proof}

\begin{lemma}\label{thm:Dboundl}
	There exists a $C>0$ such that $\mathcal{D}_{1,l} \leq C\, (1+l^b)$ for all $l\geq 0$.
\end{lemma}
\begin{proof}
	Under our assumptions on $\nu$ its characteristic function satisfies $|1-\phi(\theta)| \gtrsim |\theta|^{1+b}$ as $\theta\to 0$.
	By \eqref{eq:wienerhopfchar} and Proposition \ref{thm:admiss}(iii) we therefore have
	\begin{equation*}
	\frac{1}{|1-G^\geq(e^{i\theta})|} = \frac{\left|\sqrt{1-e^{-i\theta}}\right|}{|1-\phi(\theta)|} \lesssim  |\theta|^{-b-\tfrac{1}{2}}\quad \text{as }\theta\to 0.
	\end{equation*}
	Standard transfer theorems then imply that
	\begin{equation*}
	[z^{k-1}] \frac{1}{1-G^\geq(z)} \lesssim k^{b-1/2} \quad \text{as }k\to\infty.
	\end{equation*}
	The left-hand side is the expected number of visits of the weak ascending ladder process $(H_j^\geq)$ to height $k-1$.
	Therefore
	\begin{align*}
	[z^{k-1}] \frac{1}{1-G^\geq(z)} &= \sum_{i=0}^\infty \probric_0[ W_1,\ldots,W_{i-1} < k, W_i = k-1 ]\\
	&=\sum_{i=0}^\infty \probric_1[ W_1,\ldots,W_{i-1} > 0, W_i = k ],
	\end{align*}
	where the second equality follows from duality.
	Combining with \eqref{eq:dfppbpm2} we find for some $C>0$ the inequality
	\begin{equation*}
	\mathcal{D}_{1,l} \leq C \sum_{k=1}^\infty \frac{H_{l}(k)}{k\,H_l(1)} k^{b-1/2}.
	\end{equation*}
	Using that 
	\begin{equation*}
	\frac{H_{l}(k)}{k\,H_l(1)} = 2\frac{l+1}{l+k} H_0(k) \leq c \frac{l+1}{l+k}k^{-1/2}
	\end{equation*}
	for some $c>0$ and all $k\geq 1$ and $l\geq0$, the result follows from a simple estimation of the sum.
\end{proof}

\begin{lemma}\label{thm:dfppbound}
	There exists a $C>0$ such that $\mathcal{D}_{p,l} \leq C\,(p^b+l^b)$ for all $p\geq 1$ and $l\geq 0$.
\end{lemma}
\begin{proof}
Due to Lemma \ref{thm:Dboundl} and Lemma \ref{thm:Dsym} the result is granted if $p=1$ or $l=1$.
One may easily extend this to the case $l=0$ by comparing for $l=0$ and $l=1$ the expectation value of (\ref{eq:bpmfppid}) conditional on the value of $W_i$ just before it hits $\Z_{\leq 0}$. 
Since $H_1(k) \leq H_0(k) \leq 4H_1(k)$ they differ by at most an absolute constant.

For the general case $p\geq 1,l\geq 0$, consider the walk $(W_i)_i$ started at $p$ and let $m \in \{1,2,\ldots,p\}$ be the minimum of $(W_i)_i$ before it hits $\Z_{\leq 0}$.
Then conditionally on $m$ and on $(W_i)_i$ hitting $\Z_{\leq 0}$ at $-l$ the expectation of $\sum_{i=0}^\infty \frac{1}{W_i} \one_{\{W_1,\ldots,W_i>0\}}$ is
\begin{equation*}
\expecric_p\left[\sum_{i=0}^\infty \frac{1}{W_i} \one_{\{W_1,\ldots,W_i>m\}}\middle| (W_i)_i\text{ hits }\Z_{\leq m}\text{ at }m\right]+\expecric_m\left[\sum_{i=0}^\infty \frac{1}{W_i} \one_{\{W_1,\ldots,W_i\geq m\}}\middle| (W_i)_i\text{ hits }\Z_{< m}\text{ at }-l\right].
\end{equation*}
By shifting the starting point of the walk downwards it is easy to see that the two terms are bounded by $\mathcal{D}_{p-m,0}$ respectively $\mathcal{D}_{1,m+l-1}$.
Hence
\begin{equation*}
\mathcal{D}_{p,l} \leq \max_{1\leq m\leq p} \left( \mathcal{D}_{p-m,0} + \mathcal{D}_{1,m+l-1}\right) \leq C p^b + C (l+p)^b,
\end{equation*}
from which the result follows.
\end{proof}

\begin{proof}[Proof of Proposition \ref{thm:expdfpp}]
	Combining Lemma \ref{thm:dfppbound} with \eqref{eq:dfppD}, there exists a $C>0$ such that
	\begin{equation*}
\expecric^\downarrow_p \left[ \sum_{i=0}^\infty \frac{1}{W^\downarrow_i} \one_{\{W^\downarrow_i > 0\}}\right] \leq C \sum_{n=0}^\infty \expecric^\downarrow_p\left[ (\ell_n^\downarrow)^b \,\one_{\{\ell^\downarrow_n>0\}}\right] \leq C \sum_{n=0}^\infty \expecric_p^\downarrow\left[ \frac{1}{h_\pr^{\downarrow}(\ell^\downarrow_n)} \,\one_{\{\ell_n>0\}}\right],
	\end{equation*}
	where in the second inequality we used \eqref{eq:hdownasymp}.
	Using that $(W_i^\downarrow,N_i^\downarrow)$ under $\probric^\downarrow_p$ and $(W_i^*,N_i^*)$ under $\probric^*_p$ are related by an $h$-transform with respect to $h^\downarrow_\pr$ we may evaluate
	\begin{align*}
	\expecric^\downarrow_p\left[\frac{1}{h_n^\downarrow(\ell^\downarrow_n)} \one_{\{\ell_n >0\}} \right] &= \sum_{k=1}^\infty \frac{1}{h_\pr^\downarrow(k)}\probric^\downarrow_p[\ell^\downarrow_j=k] \\
	&= \sum_{k=1}^\infty \frac{1}{h_n^\downarrow(p)}\probric^*_p[\ell^*_n=k]\\
	&= \frac{1}{h_\pr^\downarrow(p)} \prob( N^*_\infty \geq n ) \leq \frac{1}{h_\pr^\downarrow(p)} \pr^n.
	\end{align*}
	Therefore 
	\begin{equation*}
	\expecric^\downarrow_p \left[ \sum_{i=0}^\infty \frac{1}{W^\downarrow_i} \one_{\{W^\downarrow_i > 0\}}\right] \leq C \frac{1}{h_\pr^\downarrow(p)} \frac{1}{1-\pr},
	\end{equation*}
	and the desired bound follows again from \eqref{eq:hdownasymp}.
\end{proof}

\subsection{Proof of Theorem \ref{thm:fpp}}

First we will show that the exponential variables in \eqref{eq:bpmfppid} do not affect the distribution of $\hat{d}_{\mathrm{fpp}}(\map,\loopconf)$ in the large-$p$ limit.
This is a consequence of
\begin{align*}
\expecric_p^\downarrow \left[\left( \sum_{i=0}^\infty \frac{\mathbf{e}_i}{2W^\downarrow_i}\one_{\{W^\downarrow_i>0\}}-\sum_{i=0}^\infty \frac{1}{2W^\downarrow_i}\one_{\{W^\downarrow_i>0\}} \right)^2\right]
&= \frac{1}{4}\expecric_p^\downarrow\left[ \sum_{i=0}^\infty \frac{\one_{\{W^\downarrow_i>0\}}}{(W^\downarrow_i)^2} \right]\\
&\leq \expecric_p^\downarrow\left[ \sum_{i=0}^\infty \frac{\one_{\{W^\downarrow_i>0\}}}{W^\downarrow_i} \right] \leq C p^b.
\end{align*}
Therefore we have the convergence in probability
\begin{equation*}
p^{-b}\sum_{i=0}^\infty \frac{\mathbf{e}_i}{2W^\downarrow_i}\one_{\{W^\downarrow_i>0\}}-p^{-b}\sum_{i=0}^\infty \frac{1}{2W^\downarrow_i}\one_{\{W^\downarrow_i>0\}} \xrightarrow[p\to\infty]{\probric^\downarrow_{p}} 0.
\end{equation*}

It follows from the weak convergence in the Skorokhod topology of Proposition \ref{thm:conv} that for any $\epsilon>0$ under $\probric^\downarrow_p$ we have
\begin{equation*}
\frac{p^{-b}}{C}\sum_{i=0}^\infty \frac{1}{W^\downarrow_i}\one_{\{W^\downarrow_i>p\epsilon\}} = 
\int_0^\infty \frac{p\,\rmd t}{W^\downarrow_{\lfloor Ct p^\theta\rfloor}}\one_{\{W^\downarrow_{\lfloor Ct p^\theta\rfloor} > \epsilon p\}} \xrightarrow[p\to\infty]{\mathrm{(d)}} \int_0^\infty \frac{\rmd t}{X_t^\downarrow} \one_{\{X_t^\downarrow > \epsilon\}}.
\end{equation*}
On the other hand, by Proposition \ref{thm:expdfpp} we have for all $p \geq 1$
\begin{equation*}
p^{-b}\expecric^\downarrow_p\left[\sum_{i=0}^\infty \frac{1}{W^\downarrow_i}\one_{\{1\leq W^\downarrow_i\leq p\epsilon\}}\right] \leq p^{-b} \max_{1\leq m \leq \lfloor \epsilon p\rfloor}\expecric^\downarrow_m\left[ \sum_{i=0}^\infty \frac{1}{W^\downarrow_i}\one_{\{W_i^\downarrow>0\}}\right]  \leq C \epsilon^{b}.
\end{equation*}
Since this can be made arbitrarily small and $\int_0^\infty\frac{\rmd t}{X_t^\downarrow}\one_{\{X_t^\downarrow>0\}}$ is a.s. finite, we obtain the convergence 
\begin{equation*}
\frac{p^{-b}}{2C} \hat{d}_{\mathrm{fpp}}(\map_\bullet,\loopconf) \stackrel{\mathrm{(d)}}{=} \frac{p^{-b}}{C}\sum_{i=0}^\infty \frac{\mathbf{e}_i}{W^\downarrow_i}\one_{\{W^\downarrow_i>0\}} \xrightarrow[p\to\infty]{\mathrm{(d)}} \int_0^{\infty}\frac{\rmd t}{X_t^\downarrow}\one_{\{X_t^\downarrow>0\}} = R^\downarrow.
\end{equation*}
It remains to check the convergence in probability
\begin{equation}\label{eq:dffpdiffinprob}
p^{-b} \hat{d}_{\mathrm{fpp}}(\map_\bullet,\loopconf) - p^{-b} d_{\mathrm{fpp}}(\map_\bullet,\loopconf) \xrightarrow[p\to\infty]{\prob} 0.
\end{equation}
By \eqref{eq:dfppdifference} we have that
\begin{equation*}
\expec^{(p)}_\bullet\left[ |\hat{d}_{\mathrm{fpp}}(\map_\bullet,\loopconf) - d_{\mathrm{fpp}}(\map_\bullet,\loopconf)| \right] \leq \expec^{(p)}_\bullet\left[ \text{degree of the marked vertex of }\map_\bullet\right].
\end{equation*}
By the following Lemma this is bounded uniformly in $p$, implying \eqref{eq:dffpdiffinprob} and thus finishing the proof of Theorem \ref{thm:fpp}.

\begin{lemma} 
The expected degree of the marked vertex is uniformly bounded in $p$.
\end{lemma}
\begin{proof}
One way to see this is to note that 
\begin{equation}\label{eq:edgvertratio}
\expec^{(p)}_\bullet\left[ \text{degree of the marked vertex of }\map_\bullet\right] = \frac{ 2\expec^{(p)}\left[|\edges(\map)|\right]}{\expec^{(p)}\left[|\vertices(\map)|\right]}, 
\end{equation}
where $(\map,\loopconf)$ is a (unpointed) $(\qseq,g,n)$-Boltzmann loop-decorated map of perimeter $2p$.
To every pair consisting of a loop-decorated map $(\map,\loopconf)$ and a marked edge $e\in\edges(\map)$ that is not crossed by a loop one can injectively assign a pointed loop-decorated $(\map_\bullet',\loopconf')$ by ``unzipping'' the edge $e$ into a bigon and inserting a loop and a marked vertex as in Figure \ref{fig:bigoninsert}.
Then $w_{\qseq,g,n}(\map,\loopconf) = w_{\qseq,g,n}(\map_\bullet',\loopconf')/(n g^2)$.
Since at least half of the edges is not crossed by a loop we find
\begin{align*}
\expec^{(p)}[|\edges(\map)|] &= \frac{1}{F^{(p)}}\sum_{(\map,\loopconf)\in \loopmaps^{(p)}} \sum_{e\in\edges(\map)} w_{\qseq,g,n}(\map,\loopconf) \\
&\leq \frac{2}{ng^2} \frac{1}{F^{(p)}}\sum_{(\map_\bullet,\loopconf)\in \loopmaps_\bullet^{(p)}} w_{\qseq,g,n}(\map_\bullet,\loopconf)=\frac{2}{ng^2} \expec^{(p)}\left[|\vertices(\map)|\right].
\end{align*}
Together with \eqref{eq:edgvertratio} this finishes the proof.
\end{proof}

\begin{figure}[h]
	\centering
	\includegraphics[width=.5\linewidth]{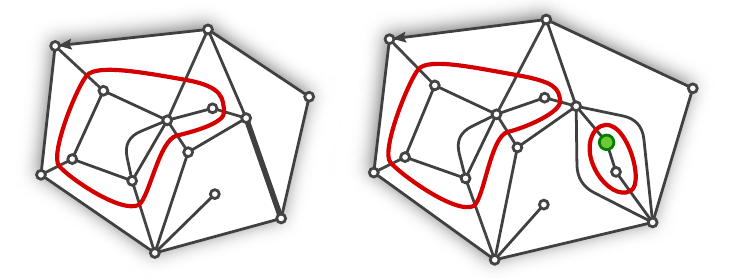}
	\caption{Illustration of the unzipping of a marked edge and the insertion of a loop surrounding a marked vertex.\label{fig:bigoninsert}}
\end{figure}

\section{Exponential integrals}\label{sec:expintgr}

\subsection{The Mellin transform} Given a (unkilled) Levy process $(\xi_t)_{t\geq0}$ that drifts to $+\infty$ one can define the exponential integral
\begin{equation*}
I(\xi) := \int_0^{\infty}e^{-\xi_t}\rmd t,
\end{equation*}
which converges almost surely.
Let $\mathcal{M}(s)$ be the Mellin transform of the distribution of this exponential integral, i.e.
\begin{equation*}
\mathcal{M}(s) = \expec[ I(\xi)^{s-1}],
\end{equation*} 
and $\psi(z)=\log\expec\exp(z \xi_1)$ the Laplace exponent of $(\xi_t)_{t\geq0}$.
One says $(\xi_t)_{t\geq0}$ satisfies \emph{Cram\'er's condition} if there exist $0<\Theta<z_0$ such that $\psi(z)$ is finite for all $z\in(-z_0,0)$ and $\psi(-\Theta)=0$.
In this case the following \emph{verification result} proved in \cite[Proposition 2]{kuznetsov_fluctuations_2010} allows one  to fix $\mathcal{M}(s)$ in terms of $\psi(z)$.

\begin{proposition}[\emph{Verification result} {\cite[Proposition 2]{kuznetsov_fluctuations_2010}} ]\label{thm:verificationresult}
	If Cram\'er's condition is satisfied for some $\Theta>0$ and $f(s)$ satisfies the properties
	\begin{enumerate}
		\item[(i)] $f(s)$ is analytic and non-zero for $\mathrm{Re}(s)\in(0,\Theta+1)$,
		\item[(ii)] $f(1)=1$ and $f(s+1) = -s f(s)/\psi(-s)$ for $s\in(0,\Theta)$,
		\item[(iii)] $\lim_{y\to\pm\infty} |f(x+i y)|^{-1}e^{-2\pi|y|} = 0$ for $x\in(0,\Theta+1)$,
	\end{enumerate}
	then $f(s) = \mathcal{M}(s)$ for $\mathrm{Re}(s)\in(0,\Theta+1)$.
\end{proposition}

In particular we can apply this to the rescaled L\'evy processes $(-\frac{1}{\delta}\xi^\downarrow_t)_{t\geq0}$ and $(\frac{1}{\delta}\xi^\uparrow_t)_{t\geq0}$ for $\delta > 0$, since they satisfy Cram\'er's condition with $\Theta = 2b\delta$.
Therefore it makes sense to define the Mellin transforms of the resulting exponential integrals
\begin{equation}
\mathcal{M}^\downarrow(s) = \mathcal{M}^\downarrow(s;\theta,b,\delta) := \expec[ I(-\compactfrac{1}{\delta}\xi^\downarrow_t)^{s-1}],\quad\mathcal{M}^\uparrow(s) = \mathcal{M}^\uparrow(s;\theta,b,\delta) := \expec[ I(\compactfrac{1}{\delta}\xi^\uparrow_t)^{s-1}].\label{eq:mellinexpint}
\end{equation}
Based on the similarity of the laplace exponents $\Psi^\downarrow(z)$ and $\Psi^\uparrow(z)$ in (\ref{eq:psidown}) and (\ref{eq:psiup}) to the hypergeometric versions in \cite{kuznetsov_fluctuations_2010}, one can simply guess a solution to the difference equation of Proposition \ref{thm:verificationresult}(ii) and then one has to check the remaining conditions.

It turns out one can construct such a solution in terms of the Barnes double gamma function   $G(z;\tau)$, $|\arg\tau|<\pi$, which is an entire function in $z$ with simple zeros on the lattice $\tau\Z_{\leq0} +\Z_{\leq0}$ (see \cite{barnes_genesis_1899} and \cite{kuznetsov_fluctuations_2010} for details).
It satisfies $G(1;\tau)=1$ and the three modular-type functional identities
\begin{align}
G(z+1;\tau) &= \Gamma\left(\frac{z}{\tau}\right) G(z;\tau), \label{eq:Gshift1}\\
G(z+\tau;\tau) &= (2\pi)^{(\tau-1)/2}\tau^{-z+1/2}\Gamma(z)G(z;\tau),\label{eq:Gshifttau} \\
G(z;\tau) &= (2\pi)^{z(1-1/\tau)/2}\tau^{-\frac{z^2}{2\tau}+\frac{z}{2}(1+1/\tau)-1}G\left(\frac{z}{\tau};\frac{1}{\tau}\right).\nonumber
\end{align}

\begin{proposition}\label{thm:mellin}
	The Mellin transforms $\mathcal{M}^\downarrow$ and $\mathcal{M}^\uparrow$ in (\ref{eq:mellinexpint}) for $\theta\in[1/2,3/2]$, $b\in(0,1/2]$, $\delta>0$ are positive and analytic on $(-(1-b)\delta, 1+2b\delta)$ and  $(-\delta,1+2b\delta)$, respectively, and are
	given by
	\begin{align}
	\mathcal{M}^\downarrow(s;\theta,b,\delta) =& C^\downarrow_{\theta,b,\delta} 2^{(1-\theta)s}\delta^s\frac{G(2\delta+s;2\delta)}{G(1+2\delta-s;2\delta)}\frac{G(2(1-b)\delta+s;2\delta)}{G(1+2b\delta-s;2\delta)}\nonumber\\
	&\times\frac{G(1+(\theta+b)\delta-s;2\delta)}{G((1-b)\delta+s;2\delta)}\frac{G(1+(1+\theta+b)\delta-s;2\delta)}{G((2-b)\delta+s;2\delta)},\label{eq:mellindown}\\
	\mathcal{M}^\uparrow(s;\theta,b,\delta) =& C^\uparrow_{\theta,b,\delta} 2^{(1-\theta)s}\delta^s\frac{G(2\delta+s;2\delta)}{G(1+2\delta-s;2\delta)}\frac{G(2(1-b)\delta+s;2\delta)}{G(1+2b\delta-s;2\delta)}\nonumber\\
	&\times\frac{G(1+(1+b)\delta-s;2\delta)}{G((\theta-b)\delta+s;2\delta)}\frac{G(1+(2+b)\delta-s;2\delta)}{G((1+\theta-b)\delta+s;2\delta)},\label{eq:mellinup}
	\end{align}
	where $G(z;\tau)$ is the Barnes double gamma function described above and $C^\downarrow_{\theta,b,\delta},C^\uparrow_{\theta,b,\delta}$ are $(\theta,b,\delta)$-dependent constants chosen to ensure $\mathcal{M}^\downarrow(1)=\mathcal{M}^\uparrow(1)=1$.
\end{proposition}
\begin{proof}
	Let's denote the right-hand sides of (\ref{eq:mellindown}) and (\ref{eq:mellinup}) by $f^\downarrow(s)$ and $f^\uparrow(s)$, respectively.
	We will verify that $f^\downarrow(s)$ and $f^\uparrow(s)$ satisfy the conditions of Proposition \ref{thm:verificationresult} with $\Theta = 2b\delta$ and $\psi^\downarrow(s) = \Psi^\downarrow(-s/\delta)$, respectively, $\psi^\uparrow(s) = \Psi^\uparrow(s/\delta)$.
	
	Since all double gamma functions appearing in $f^\downarrow(s)$ and $f^\uparrow(s)$ are either of the form $G(p+s;2\delta)$ with $p>0$ or of the form $G(1+p-s)$ with $p \geq 2b\delta=\Theta$, $f^\downarrow(s)$ and $f^\uparrow(s)$ are analytic and non-zero on the strip $\mathrm{Re}(s)\in(0,\Theta+1)$, showing that they satisfy (i).
	Condition (ii) follows from direct substitution and multiple applications of (\ref{eq:Gshift1}).
	
	It remains to evaluate the asymptotics of condition (iii).
	From \cite{kuznetsov_fluctuations_2010}, Lemma 1, we know that for $p,q\in\R$ fixed and $y\to\pm\infty$
	\begin{equation*}
	\log\left|\frac{G(p+i y;2\delta)}{G(q+i y;2\delta)}\right| = \frac{p-q}{2\delta}\mathrm{Re}[(q+i y)\log(q+i y)] + o(|y|) = -\pi\frac{p-q}{4\delta} |y| + o(|y|).
	\end{equation*}
	Applying this to $\log|f^\downarrow(x+iy)|$ and $\log|f^\uparrow(x+iy)|$ multiple times and using that $\log(|2^{(1-\theta)(x+iy)}\delta^{x+iy}|)$ is independent of $y$, we find 
	\begin{equation*}
	\log|f^\downarrow(x+iy)|= -\frac{\pi}{2}\theta |y| + o(|y|),\quad\quad
	\log|f^\uparrow(x+iy)| = -\frac{\pi}{2}(2-\theta)|y| + o(|y|),
	\end{equation*}
	and therefore condition (iii) is satisfied in both cases.
\end{proof}

\subsection{$R^\downarrow$ as an exponential integral}\label{sec:expintradius}

Our main interest lies in the case $\theta = 1+b$ and the random variable $R^\downarrow$ appearing in Theorem \ref{thm:fpp},
\begin{equation*}
R^\downarrow = \int_{0}^{T^\downarrow} \frac{\rmd t}{X^\downarrow_t} = \int_0^\infty e^{-b\xi^\downarrow_s}\rmd s = I(-b\xi^\downarrow).
\end{equation*}
Its Mellin transform is therefore given by
\begin{equation}\label{eq:rdownmel}
\expec (R^\downarrow)^{s-1} = \mathcal{M}^\downarrow(s;b+1,b,1/b).
\end{equation}

The expression \eqref{eq:mellindown} is sufficiently explicit to deduce the asymptotics of the distribution as described in Proposition \ref{thm:rdownasymp}.

\begin{proof}[Proof of Proposition \ref{thm:rdownasymp}]
According to Proposition \ref{thm:mellin} the Mellin transform $\mathcal{M}^\downarrow(s;b+1,b,1/b)$ is analytic on the interval $(1-1/b,3)$. 
To find the leading divergence as $s\to 1-1/b$, we use that
\begin{equation*}
\mathcal{M}^\downarrow(s;b+1,b,1/b) = - \frac{1}{s} \Psi^\downarrow(b s)\mathcal{M}^\downarrow(s+1;b+1,b,1/b).
\end{equation*}
$\Psi^\downarrow(b s)$ has a simple pole at $s=1-1/b$ with residue $\frac{2}{\pi b}\cos^2(\frac{\pi b}{2}) \Gamma(b+2)$.
A tedious but straightforward calculation using the shift relations \eqref{eq:Gshift1} and \eqref{eq:Gshifttau} yields
\begin{equation*}
\mathcal{M}^\downarrow(2-1/b;b+1,b,1/b) = b^{-2+\frac{1}{b}}\Gamma(2-b).
\end{equation*}
Hence
\begin{equation*}
\mathcal{M}^\downarrow(s;b+1,b,1/b) \stackrel{s\to 1-\tfrac{1}{b}}{\sim}b^{-1+1/b} (b+1) \cot\left(\frac{\pi b}{2}\right)\,(s-1+\tfrac{1}{b})^{-1},
\end{equation*}
which is easily seen to imply
\begin{equation*}
\prob( R^\downarrow < r )\stackrel{r\to 0}{\sim} b^{1/b} (b+1) \cot\left(\frac{\pi b}{2}\right) r^{1/b}.
\end{equation*}

Similarly, $\Psi^\downarrow(b(s-1))$ has a pole at $s=3$ with residue $-\frac{\pi b}{\Gamma(-b)}$ and
\begin{equation*}
\mathcal{M}^\downarrow(2;b+1,b,1/b) = -\Gamma(-b) \cot\left(\frac{\pi b}{2}\right),
\end{equation*}
such that
\begin{align*}
\mathcal{M}^\downarrow(s;b+1,b,1/b) &= \frac{1-s}{\Psi(b(s-1))} \mathcal{M}^\downarrow(s-1;b+1,b,1/b) \\
&\stackrel{s\to 3}{\sim} 2\frac{\Gamma(-b)^2}{\pi b} \cot\left(\frac{\pi b}{2}\right)\,(3-s)^{-1},
\end{align*}
which implies
\begin{equation*}
\prob( R^\downarrow > r )\stackrel{r\to \infty}{\sim} \frac{\Gamma(-b)^2}{\pi b} \cot\left(\frac{\pi b}{2}\right) \, r^{-2},
\end{equation*}
as claimed.
\end{proof}

On the other hand, we may consider the \emph{first Lamperti transform} $(\hat{L}^\uparrow_r)_{r\geq0}$ of $(X^\uparrow_t)_{t\geq0}$ started at $x>0$, which is a pssMp with index $\theta-1=b>0$ determined by the time change
\begin{equation*}
\hat{L}^\uparrow_r = X^\uparrow_{t(r)}, \quad t(r) := \inf\left\{t > 0 : \int_0^t \frac{1}{X_u}\rmd u > r\right\}.
\end{equation*}
It has the same Lamperti representation $\xi^{\uparrow}_t$ as $(X^{\uparrow}_t)_{t\geq0}$ but a different index of similarity.
According to \cite[Theorem 1(i)]{bertoin_entrance_2002}, $(\hat{L}^\uparrow_r)_{r\geq 0}$ converges in distribution (in the sense of finite-dimensional marginals) as $x\to0$ to a well-defined pssMp, which we denote by $(L^\uparrow_r)_{r\geq0}$.

\begin{remark}\label{rem:hull}
In the light of Section \ref{sec:ricochetedsurvive} the process $(L^\uparrow_r)_{r\geq0}$ has the following conjectural interpretation.  
Let $(\map_\infty,\loopconf)$ be the tentative infinite Boltzmann loop-decorated map in the non-generic critical dilute phase with perimeter $p$.
Let $\overline{\mathrm{Ball}}_t(\map_\infty,\loopconf)$ be the hull of the ball of fpp-radius $t$ as in Section \ref{sec:first-passage-percolation} and $|\partial\overline{\mathrm{Ball}}_t(\map_\infty,\loopconf)|$ the perimeter of its boundary, which in $\map_\infty$ separates the root from infinity.
Then it is expected that for some $c>0$ we have the convergence in distribution in the Skorokhod topology
\begin{equation*}
\left(\frac{|\partial\overline{\mathrm{Ball}}_{\lambda^{b} r}(\map_\infty,\loopconf)|}{c\lambda}\right)_{r\geq 0} \xrightarrow[\lambda\to\infty]{\mathrm{(d)}} (L^\uparrow_r)_{r\geq0}.
\end{equation*}
\end{remark}

The process $(L^\uparrow_r)_{r\geq0}$ is started at $0$ and has entrance law determined by \cite[Theorem 1(iii)]{bertoin_entrance_2002}
\begin{equation*}
\expec f(L^\uparrow_r) = \frac{1}{b\, \expec\xi^\uparrow_1}\expec\left[ \frac{1}{I(b\xi^\uparrow)}f\left(\frac{ r^{1/b}}{I(b\xi^\uparrow)^{1/b}}\right)\right]
\end{equation*} 
for any non-negative measurable function $f:\R_{>0}\to\R_{\geq 0}$.
In particular, its Mellin transform can be expressed as
\begin{equation}\label{eq:entrlawmellin}
\expec ( L^{\uparrow}_r)^{s-1} = \frac{r^{(s-1)/b}}{b\,\expec\xi^\uparrow_1}\mathcal{M}^\uparrow\left(\frac{1-s}{b};b+1,b,\frac{1}{b}\right).
\end{equation}
As a consequence of Proposition \ref{thm:mellin} and relation \eqref{eq:Gshift1} one finds the following simple relation between the two Mellin transforms.

\begin{lemma}\label{thm:mupdownrel2}
	For $b\in(0,1/2]$, and $s \in (-1/b,3)$ we have
	\begin{equation*}
	\frac{\mathcal{M}^\downarrow(s;b+1,b,1/b)}{\mathcal{M}^\uparrow(s;b+1,b,1/b)} = \frac{\Gamma(1+2b-bs)\Gamma(1-b+bs)}{\Gamma(1+b)}
	\end{equation*}
\end{lemma}

This allows one to determine an integral relation between the distributions of $R^\downarrow$ and $L^{\uparrow}_r$.

\begin{proposition}\label{thm:RdownLuprel}
	For $b\in(0,1/2]$ the distribution function of $R^\downarrow$ is given by 
	\begin{equation*}
		\prob( R^\downarrow > r ) = b r \sin(\pi b)\Gamma(1+b)\expec\left[ \frac{ (L_r^\uparrow)^{-b}}{(1+L_r^\uparrow)^{b+1}}\right]
	\end{equation*}
\end{proposition}
\begin{proof}
	We claim that the probability density function of $R^\downarrow$ is given by
		\begin{equation}\label{eq:RdownLuprel}
		\prob( R^\downarrow \in \rmd r ) = 
		\sin(\pi b)\Gamma(2+b) \,\expec\left[ \frac{(L_r^{\uparrow})^{1-b}}{(1+L^\uparrow_r)^{b+2}}\right]\rmd r.
		\end{equation}
	In that case the claimed result follows from integration using the self-similarity $L_r^\uparrow \stackrel{\mathrm{(d)}}{=} r^{1/b}L_1^\uparrow$,
	\begin{align*}
	\int_{r}^\infty \expec\left[ \frac{(L_u^{\uparrow})^{1-b}}{(1+L^\uparrow_u)^{b+2}}\right]\rmd u &= \expec\left[\int_r^{\infty}  \frac{u^{\frac{1}{b}-1}(L_1^\uparrow)^{1-b}}{(1+u^{\frac{1}{b}}L_1^\uparrow)^{b+2}}\rmd u\right] \\
	&= \frac{b}{1+b}\expec\left[ \frac{ (L_1^\uparrow)^{-b}}{(1+r^{\frac{1}{b}}L_1^\uparrow)^{b+1}}\right] = \frac{b\, r}{1+b}\expec \left[\frac{ (L_r^\uparrow)^{-b}}{(x+L_r^\uparrow)^{b+1}}\right].
	\end{align*}	
	
	To verify (\ref{eq:RdownLuprel}) we will compute the Mellin transform of its right-hand side and check that it agrees with the Mellin transform (\ref{eq:rdownmel}) of the left-hand side for $s\in (1-1/b, 3)$.
	The equality of the densities in (\ref{eq:RdownLuprel}) then follows from the uniqueness of the Mellin transform.
	
	Using again the self-similarity (and Fubini's theorem) we have 
	\begin{align}\label{eq:lupintgr}
	\int_0^\infty \rmd r \,r^{s-1}\, \expec\left[ \frac{(L_r^{\uparrow})^{1-b}}{(1+L^\uparrow_r)^{2+b}}\right] &= \expec\left[\int_0^\infty \rmd r\, r^{s-1}  \frac{(r^{1/b}L_1^{\uparrow})^{1-b}}{(1+r^{1/b}L^\uparrow_1)^{2+b}}\right],
	\end{align}
	which converges for $s\in (1-1/b, 3)$.
	Since in the integrand is positive, we may change the integration variable to $z = 1/(r^{1/b}L^\uparrow_1)$, yielding
	\begin{align*}
	\int_0^\infty \rmd r \,r^{s-1} &\expec\left[ \frac{(L_r^{\uparrow})^{1-b}}{(1+L^\uparrow_r)^{2+b}}\right] =b\,\expec\left[  \int_0^\infty \frac{\rmd z}{z}\, \left(z L_1^\uparrow\right)^{-bs} z^{-2b-1}(z+1)^{-b-2}\right] ,\\
	&= b\,\expec\left[ (L_1^\uparrow)^{-bs} \int_0^\infty \frac{\rmd z}{z} z^{1+2b-bs}(z+1)^{-b-2}\right]\\
	&= \frac{\Gamma(1+2b-bs)\Gamma(1-b+bs)}{\Gamma(2+b) \expec \xi_1^\uparrow} \mathcal{M}^\uparrow(s;b+1,b,1/b),
	\end{align*}
	where in the last equality we used (\ref{eq:entrlawmellin}).
	This is finite precisely when $s\in(1-1/b, 3)$, in which case by Lemma \ref{thm:mupdownrel2} and $\expec\xi^\uparrow_1 = \Gamma(1+b)\sin(\pi b)$ it equals
	\begin{equation*}
	\frac{1}{\Gamma(2+b)\sin(\pi b)} \mathcal{M}^\downarrow(s,b+1,b,1/b),
	\end{equation*}
	in agreement with (\ref{eq:RdownLuprel}) and (\ref{eq:rdownmel}).
\end{proof}

\subsection{Further simplification in the rational case}

When $b=1/m$ is the reciprocal of an integer $m\geq 2$ (and $\theta=1+b$, $\delta=1/b$), the Barnes double gamma functions in the Mellin transforms of Proposition \ref{thm:mellin} can expressed in terms of more familiar functions.
We will do this first for $L_r^\uparrow$ and then use Proposition \ref{thm:RdownLuprel} to find a relatively simple integral expression for the distribution function of $R^\downarrow$.

\begin{proposition}\label{thm:explexpr}
	Suppose $\theta=1+b$, $b=1/m$, $m = 2,3,\ldots$, then
	\begin{itemize}
		\item[(i)] the Mellin transform of $L_r^\uparrow$ for $s\in(1-3/m,2)$ is given by
		\begin{align}
		\expec (L_r^\uparrow)^{s-1} &= -\left(\frac{r}{m}\right)^{m(s-1)}2^{-m+1}\sqrt{m} \frac{\Gamma\left(-\frac{1}{m}\right)\Gamma\left(\frac{1}{m}+s\right)}{\pi\sin\left(\frac{\pi}{2m}\right)}\prod_{k=3}^{m} \frac{1}{\sin\left(\frac{\pi k}{2m}+\frac{\pi (s-1)}{2}\right)};\label{eq:lupmellinspec}
		\end{align}
		\item[(ii)] the "tilted" Laplace transform $G_r(Z) := \expec (L_r^\uparrow)^{-b}e^{-Z L_r^\uparrow}$ is given by 
		\begin{equation}\label{eq:GrZ}
		G_r(Z) =  \frac{- \Gamma(-1/m)}{\pi} \frac{m}{r} \left[1- B_m\left( \frac{m}{r}Z^{-1/m} \right)\right],\quad\quad  B_m(y):=   \frac{1 + y\cot\left(\frac{\pi}{2m}\right)}{\prod_{k=0}^m (1-iy e^{i\pi k/m})};
		\end{equation}
	\end{itemize}
\end{proposition}
\begin{proof}
	To simplify formulas in the proof of (i) and (ii), we will set $r=m$ without loss of generality, since the expressions for general $r>0$ follow from self-similarity.
	To obtain (i) let us evaluate (\ref{eq:entrlawmellin}) without keeping track of $s$-independent factors, i.e.
	\begin{align*}
	\expec (L_m^{\uparrow})^{s-1} &\propto m^{m(s-1)}\mathcal{M}^\uparrow(m-ms;1+1/m,1/m,m) \\
	&\propto 2^{s-1}\frac{G(2+m+ms;2m)}{G(1+m+ms;2m)}\frac{G(2+ms;2m)}{G(3-m+ms;2m)}\frac{G(3m-2-ms;2m)}{G(2m-ms;2m)}.
	\end{align*}
	Since the arguments of all three ratios of double gamma functions differ by an integer we can use (\ref{eq:Gshift1}) to obtain
	\begin{align*}
	\expec (L_m^{\uparrow})^{s-1} &\propto 2^{s-1} \Gamma\left(\frac{1}{2}+\frac{1}{2m}+\frac{s}{2}\right) \left[\prod_{k=3}^{m+1} \Gamma\left( \frac{k}{2m}+\frac{s-1}{2}\right)\right]\left[\prod_{k=3}^{m} \Gamma\left( 1-\frac{k}{2m}-\frac{s-1}{2}\right)\right]\\
	&= 2^{s-1} \Gamma\left(\frac{1}{2}+\frac{1}{2m}+\frac{s}{2}\right) \Gamma\left( \frac{1}{2m}+\frac{s}{2}\right) \prod_{k=3}^{m} \frac{1}{\sin\left(\frac{\pi k}{2m}+\frac{\pi (s-1)}{2}\right)}\\
	&\propto \Gamma\left( \frac{1}{m} + s\right) \prod_{k=3}^{m} \frac{1}{\sin\left(\frac{\pi k}{2m}+\frac{\pi (s-1)}{2}\right)},
	\end{align*}
	where we used the reflection and doubling formulas of the gamma function.
	As a consequence of the well-known identity $\prod_{k=1}^{2m-1} \sin\left(\frac{\pi k}{2m}\right) = 2^{2-2m}m$, one has
	\begin{equation}\label{eq:sinprod}
	\prod_{k=3}^m \sin\left(\frac{\pi k}{2m}\right) = \frac{1}{\sin\left(\frac{\pi}{2m}\right) \sin\left(\frac{\pi}{m}\right)} \sqrt{ \prod_{k=1}^{2m-1} \sin\left(\frac{\pi k}{2m}\right)} = \frac{2^{1-m}\sqrt{m}}{\sin\left(\frac{\pi}{2m}\right) \sin\left(\frac{\pi}{m}\right)}.
	\end{equation}
	Since $\expec (L_r^{\uparrow})^{s-1}=1$ at $s=1$, we obtain
	\begin{equation*}
	\expec (L_m^{\uparrow})^{s-1} = \frac{2^{1-m}\sqrt{m}}{\sin\left(\frac{\pi}{2m}\right) \sin\left(\frac{\pi}{m}\right)} \frac{\Gamma(s+1/m)}{\Gamma(1+1/m)} \prod_{k=3}^{m} \frac{1}{\sin\left(\frac{\pi k}{2m}+\frac{\pi (s-1)}{2}\right)}.
	\end{equation*}
	Another application of the double formula then yields (\ref{eq:lupmellinspec}) for $r=m$.
	
	To show (ii) we compute the Mellin transform $M(q)$ of $G_r(Z)$ for $r=m$, 
	\begin{align*}
	M(q) &:= \int_0^{\infty} \rmd Z\, Z^{q-1} G_m(Z) = \int_0^{\infty} \rmd Z\, Z^{q-1} \expec (L_m^\uparrow)^{-1/m} e^{-Z L_r^\uparrow} = \Gamma(q) \expec(L_m^\uparrow)^{-q-1/m}\\
	&=-\frac{2^{1-m}\sqrt{m}\Gamma(-1/m)}{\sin\left(\frac{\pi}{2m}\right)\sin(\pi q)}\prod_{k=2}^{m-1} \frac{1}{\sin\left(\frac{\pi k}{2m}-\frac{\pi q}{2}\right)} = \frac{2^{-m}\sqrt{m}\Gamma(-1/m)}{\sin\left(\frac{\pi}{2m}\right)} \prod_{\substack{k=0\\k\neq 1}}^{m} \frac{1}{\sin\left(\frac{\pi k}{2m}-\frac{\pi q}{2}\right)},
	\end{align*}
	which is valid whenever $q>0$ and $\expec(L_m^\uparrow)^{-q-1/m}<\infty$, i.e., when $q \in (0,2/m)$.
	Because $|M(q)|$ falls off exponentially as $\mathrm{Im}(q)\to\infty$, by the Mellin inversion theorem $G_m(Z)$ is given by
	\begin{equation*}
	G_m(Z) = \frac{1}{2\pi i} \int_{c-i\infty}^{c+i\infty} \rmd q\, Z^{-q} M(q),\quad\quad c\in(0,2/m), Z>0.
	\end{equation*}
	Due to the periodicity $M(q+2) = (-1)^m M(q)$ of $M$, shifting $c\to c+2$ yields the same integral multiplied by $(-1)^m Z^{-2}$.
	It follows from the residue theorem that
	\begin{equation*}
	((-1)^m Z^{-2}-1)G_m(Z) = \sum_{q\in(0,2]} Z^{-q} \mathrm{Res}(M,q),
	\end{equation*}
	where the sum is over the poles $q$ with $\mathrm{Re}(q)\in(0,2]$, which are located at $q=2$ and $q=k/m$ for $k=2,\ldots, m$.
	Using (\ref{eq:sinprod}) the residue at $q=2$ is easily found to be 
	\begin{equation*}
	\mathrm{Res}(M,2) = -(-1)^{m}\frac{2^{-m+1}\sqrt{m}\Gamma(-1/m)}{\pi\sin\left(\frac{\pi}{2m}\right)} \prod_{k=2}^{m} \frac{1}{\sin\left(\frac{\pi k}{2m}\right)} = -(-1)^m \frac{\Gamma(-1/m)}{\pi}.
	\end{equation*}
	It follows that
	\begin{equation*}
	((-1)^m Z^{-2}-1)\frac{G_m(Z)}{(-1)^m\mathrm{Res}(M,2)} - (-1)^mZ^{-2} = \sum_{p=2}^m Z^{-p/m}\frac{\mathrm{Res}(M,p/m)}{(-1)^m\mathrm{Res}(M,2)} = f(Z^{-1/m})
	\end{equation*}
	where
	\begin{equation*}
	f(y) := (-1)^m\sum_{p=2}^m y^p\frac{\mathrm{Res}(M,p/m)}{\mathrm{Res}(M,2)} = \left[\prod_{k=2}^m \sin\left(\frac{\pi k}{2m}\right)\right] \sum_{p=2}^m y^p \prod_{\substack{k=0\\k\neq 1,p}}^{m} \frac{1}{\sin\left(\frac{\pi (k-p)}{2m}\right)}
	\end{equation*}
	is a polynomial of order $m$ with a double zero at the origin.
	We claim that 
	\begin{equation*}
	f(y) = -1 + \left[1+\cot\left(\frac{\pi}{2m}\right)y\right]\prod_{k=1}^{m-1}\left(1+i y e^{i\pi k/m}\right),
	\end{equation*}
	which can be seen to imply (ii) for $r=m$.
	Since $\sum_{k=1}^{m-1} i e^{i\pi k/m} = -\cot(\pi/2m)$, the right-hand side also has a double zero at the origin, and therefore it suffices to verify the identity at $m-1$ distinct values of $y\neq 0$.
	In particular we will show that $f(i e^{-i\pi l/m})=-1$ for $l=1,\ldots,m-1$.
	
	Let us apply the method described in \cite{overflow}, which relies on the trigonometric version of Lagrange's interpolation formula.
	According to the latter, any trigonometric polynomial $P$ of order at most $K\geq 0$, i.e. a function of the form $P(x) = \sum_{k=-K}^{K} P_k e^{i k x}$, can be represented as
	\begin{equation*}
	P(x) = \sum_{p=1}^{2K+1} P\left(x_p\right) \prod_{\substack{k=1\\ k\neq p}}^{2K+1} \frac{\sin\left((x_k-x)/2\right)}{\sin\left((x_k-x_p)/2\right)}
	\end{equation*}
	for any sequence $(x_k)_{k=1}^{2K+1}$ of distinct real numbers in $[0,2\pi)$.
	
	Assuming $m$ is even, let $K=(m-2)/2$ and set $x_k = \pi(k+1)/m$ for $k=1,\ldots m-1$.
	Then we get
	\begin{equation*}
	P(x) = \sum_{p=2}^m P\left(\frac{\pi p}{m}\right) \prod_{\substack{k=2\\ k\neq p}}^{m} \frac{\sin\left(\frac{\pi k}{2m}-\frac{x}{2}\right)}{\sin\left(\frac{\pi(k-p)}{2m}\right)}
	\end{equation*}
	for any trigonometric polynomial $P$ of degree at most $\frac{1}{2}m-1$.
	Setting $x=0$ this becomes
	\begin{equation*}
	P(0) = -\left[\prod_{k=2}^m \sin\left(\frac{\pi k}{2m}\right)\right] \sum_{p=2}^m P\left(\frac{\pi p}{m}\right) \prod_{\substack{k=0\\k\neq 1,p}}^m \frac{1}{\sin\left(\frac{\pi(k-p)}{2m}\right)}.
	\end{equation*}
	In particular, for $P(x) = \exp(i (\frac{1}{2}m-l)x)$ with $1\leq l\leq m-1$ this gives $P(0) = -f(i e^{-i\pi l/m})=1$, as desired.
	
	On the other hand, if $m$ is odd, we let $K=(m-1)/2$ and set $x_1=0$, $x_k = \pi k/m$ for $k=2,\ldots,m$, such that
	\begin{equation*}
	P(x) = \sum_{\substack{p=0\\p\neq 1}}^{m} P\left(\frac{\pi p}{m}\right) \prod_{\substack{k=0\\ k\neq 1,p}}^{m} \frac{\sin\left(\frac{\pi k}{2m}-\frac{x}{2}\right)}{\sin\left(\frac{\pi(k-p)}{2m}\right)}.
	\end{equation*}
	When $P(0)=0$, the derivative at zero reads
	\begin{equation*}
	P'(0) = -\frac{1}{2}\left[\prod_{k=2}^m \sin\left(\frac{\pi k}{2m}\right)\right]\sum_{p=2}^{m} \frac{P\left(\frac{\pi p}{m}\right)}{\sin\left(\frac{\pi p}{2m}\right)} \prod_{\substack{k=0\\k\neq 1,p}}^{m} \frac{1}{\sin\left(\frac{\pi (k-p)}{2m}\right)}.
	\end{equation*}
	Therefore, choosing $P(x) = 2\exp(i (\frac{1}{2}m-l)x)\sin (\frac{1}{2}x)$ for $l=1,2,\ldots,m-1$ we again find $P'(0) = -f(i e^{-i\pi l/m})=1$, completing the proof of (ii).
\end{proof}

From here it is only a small step to the claim of Theorem \ref{thm:specRdist}.

\begin{proof}[Proof of Theorem \ref{thm:specRdist}] 
	The distribution function of Proposition \ref{thm:RdownLuprel} can be expressed in terms of $G_r(Z) = \expec\left[(L_r^\uparrow)^{-b}e^{-Z L_r^\uparrow}\right]$ as
	\begin{align*}
	\prob( R^\downarrow > r ) &= b r \sin(\pi b)\Gamma(1+b)\expec \frac{ (L_r^\uparrow)^{-b}}{(1+L_r^\uparrow)^{b+1}}\\
	&= b r \sin(\pi b) \int_0^\infty \rmd Z\,Z^b e^{-Z} G_r(Z).
	\end{align*}
	Setting $b=1/m$ and using Proposition \ref{thm:explexpr}(ii) this evaluates to
	\begin{align*}
	\prob( R^\downarrow > r ) = \frac{1}{\Gamma(1+1/m)} \int_0^\infty \rmd Z\, Z^{1/m} e^{-Z}\left[1-B_m\left(\frac{m}{r}Z^{-1/m}\right)\right],
	\end{align*}
	where we used the reflection formula for the gamma function.
	The integral over the first term in square brackets evaluates to $1$.
	Therefore
	\begin{align*}
	\prob(R^\downarrow < r) &= \frac{1}{\Gamma(1+1/m)} \int_0^\infty \rmd Z\, Z^{1/m} e^{-Z}B_m\left(\frac{m}{r}Z^{-1/m}\right) \\
	&= \frac{m}{\Gamma(1+1/m)} \int_0^\infty \rmd x\, x^m e^{-x^m}B_m\left(\frac{m}{r x}\right),
	\end{align*}
	where we made change of integration variable $Z=x^m$.
\end{proof}

Let us finish by checking that in the case $m=2$ we obtain expected results.\footnote{Actually we excluded the case $b=1/2$ (i.e. $n=\pr=0$ in the dilute phase) from many of the results (notably Proposition \ref{thm:conv} and Theorem \ref{thm:fpp}) for ease of exposition and because most already appear elsewhere in the literature, but it is straightforward to check that the results go through unaltered.}
In this case $b=1/2$ and thus $\pr=0$, meaning that $(X^\downarrow_t)_{t\geq 0}$ is the spectrally negative $3/2$-stable processes conditioned to die continuously at zero, while $(X^\uparrow_t)_{t\geq 0}$ is the same process conditioned to survive.
Moreover, the first Lamperti transform $(L^\uparrow_r)$ now has the interpretation as a time-reversal of a continuous-state branching process with branching mechanism $\psi(u) = u^{3/2}$ (see \cite[Section 2.1]{curien_hull_2016} and
\cite[Section 4.4]{curien_scaling_2017}).
According to Proposition \ref{thm:explexpr}(ii) we have for $m=2$,
\begin{equation*}
G_r(Z) = \expec\left[\frac{e^{-Z L_r^\uparrow}}{\sqrt{L_r^\uparrow}}\right] = \frac{16}{r\sqrt{\pi}} \frac{1}{4+r^2 Z}.
\end{equation*}
By taking the inverse Laplace transform we thus find the probability density
\begin{equation*}
\prob( L_r^{\uparrow} \in \rmd \ell ) = \frac{16}{\sqrt{\pi}} \frac{\sqrt{\ell}}{r^3}e^{-4 \ell/r^2}.
\end{equation*}
This is also precisely the law of the boundary of a filled geodesic ball of radius $r$ in the Brownian plane \cite[Proposition 1.2(i)]{curien_hull_2016} and is thus in agreement with the heuristics of Remark \ref{rem:hull}, since the Brownian plane is expected to be the scaling limit in the local Gromov-Hausdorff sense of Boltzmann quadrangulations.

On the other hand, by Theorem \ref{thm:specRdist} we have for $m=2$
\begin{align*}
\prob( R^\downarrow < r ) &= \frac{2}{\sqrt{\pi}} \int_0^\infty \rmd Z\, \sqrt{Z} \frac{r^2 Z}{4+r^2 Z} e^{-Z}\\
&= 1 - \frac{8}{r^2} + 16 \sqrt{\pi}\, r^{-3} e^{4r^{-2}} \,\mathrm{erfc}(2/r)
\end{align*}
where $\mathrm{erfc}(x) = 1-\frac{2}{\sqrt{\pi}} \int_0^x e^{-t^2}\rmd t$ is the complementary error function.
On differentiation with respect to $r$ one may recognize precisely the $\mu\to 0$ limit of \cite[Equation (4.59)]{bouttier_distance_2009}, which gives the distribution of the graph distance from a marked vertex to the boundary of a critical quadrangulation of large perimeter.
Since one expects the first passage percolation distance and the graph distance in a quadrangulation to differ at large distances by deterministic constant (see \cite{curien_first-passage_2015}), this is in agreement with Theorem \ref{thm:fpp}.

\appendix

\section{The phase diagram of Boltzmann loop-decorated maps (by Linxiao Chen)}\label{sec:phasediagram}


\renewcommand*{\natural}{\ensuremath{\mathbb{N}}}	
\newcommand*{\integer}{\ensuremath{\mathbb{Z}}}
\newcommand*{\rational}{\ensuremath{\mathbb{Q}}}
\newcommand*{\real}{\ensuremath{\mathbb{R}}}
\newcommand*{\complex}{\ensuremath{\mathbb{C}}}
\newcommand*{\ocomplex}{\ensuremath{\overline{\mathbb{C}}}}
\newcommand*{\id}{\ensuremath{\mathds{1}}}
\newcommand*{\idd}[1]{\ensuremath{\mathds{1}_{\m.{#1}}}}



\newcommand*{\is}{\overset{\mathtt{def}}{=}}
\renewcommand{\iff}{\ensuremath{\Leftrightarrow}}
\newcommand{\Iff}{if and only if}
\newcommand{\dd}{\mathrm{d}}							

\newcommand*{\cvg}[2]{\xrightarrow[#2]{\text{#1}}}
\newcommand*{\cv}[2][(d)]{\xrightarrow[\raisebox{0.1em}{$\scriptstyle #2\to\infty$}]{#1}}
\newcommand*{\eqv}[2][\infty]{\underset{#2\to#1}\sim}
\newcommand*{\la}{\preccurlyeq}
\newcommand*{\ga}{\succcurlyeq}

\newcommand*{\qtq}[1]{\qquad\text{#1}\qquad}
\newcommand*{\tq}[1]{\text{#1}\qquad}
\newcommand*{\qt}[1]{\qquad\text{#1}}
\newcommand*{\lesser}{\wedge}               
\newcommand*{\greater}{\vee}                

\newcommand*{\od}[3][]{\frac{\dd^{#1}#2}{\dd\!\left.#3\right.^{#1}}}
\newcommand*{\pd}[3][]{\frac{\partial^{#1}#2}{\partial\!\left.#3\right.^{#1}}}

\newcommand*{\match}[4]{%
	\ifx#1(%
	#3(#2#4)%
	\else\ifx#1[%
	#3[#2#4]%
	\else\ifx#1.%
	#3\{#2#4\}%
	\else\ifx#1<%
	#3\langle#2#4\rangle%
	\else\ifx#1/%
	#3\lfloor#2#4\rfloor%
	\else\ifx#1!%
	#3\lceil#2#4\rceil%
	\else%
	#3#1#2#4#1%
	\fi\fi\fi\fi\fi\fi}

\newcommand*{\m} [2]{\match{#1}{#2}{\left}{\right}}
\newcommand*{\mb}[2]{\match{#1}{#2}{\big }{\big }}
\newcommand*{\mB}[2]{\match{#1}{#2}{\Big }{\Big }}
\newcommand*{\mh}[2]{\match{#1}{#2}{\bigg}{\bigg}}
\newcommand*{\mH}[2]{\match{#1}{#2}{\Bigg}{\Bigg}}
\newcommand*{\abs}[1]{\m|{#1}}	
\newcommand*{\norm}[1]{\m\|{#1}}
\newcommand*{\floor}[1]{\m/{#1}}
\newcommand*{\ceil}[1]{\m!{#1}}

\newcommand*{\mm} [3]{\m #1{#2\,\middle|\,#3}}
\newcommand*{\mmb}[3]{\mb#1{#2\,\big |\,#3}}
\newcommand*{\mmB}[3]{\mB#1{#2\,\Big |\,#3}}
\newcommand*{\mmh}[3]{\mh#1{#2\,\bigg|\,#3}}
\newcommand*{\mmH}[3]{\mH#1{#2\,\Bigg|\,#3}}

\newcommand*{\set} {\mm.}
\newcommand*{\setb}{\mmb.}
\newcommand*{\setB}{\mmB.}
\newcommand*{\seth}{\mmh.}
\newcommand*{\setH}{\mmH.}

\newcommand*{\Set} [2]{\m .{#1:\,#2}}
\newcommand*{\Setb}[2]{\mb.{#1:\,#2}}
\newcommand*{\SetB}[2]{\mB.{#1:\,#2}}
\newcommand*{\Seth}[2]{\mh.{#1:\,#2}}
\newcommand*{\SetH}[2]{\mH.{#1:\,#2}}

\newcommand{\infs}[2]{\inf\set{#1}{#2}}
\newcommand{\infS}[2]{\inf\Set{#1}{#2}}
\newcommand{\sups}[2]{\sup\set{#1}{#2}}
\newcommand{\supS}[2]{\sup\Set{#1}{#2}}

\newcommand*{\0}[1]{^{(#1)}}
\newcommand*{\1}[1]{_{\mathtt{#1}}}
\newcommand*{\2}[1]{_{\text{#1}}}

\newcommand{\figbox}[3][c]{\ovalbox{\begin{minipage}[#1]{#2\textwidth}#3\end{minipage}}}
\newcommand{\pbox}[3][c]{\begin{minipage}[#1]{#2\textwidth}#3\end{minipage}}

\newcommand*{\proba}{\prob}
\newcommand*{\Proba}{\Prob}
\newcommand*{\expt}{\E}						
\newcommand*{\expect}{\E}
\newcommand*{\esp}{\E}
\newcommand*{\Expt}{\EE}
\newcommand*{\Expect}{\EE}
\newcommand*{\Esp}{\EE}

\newcommand*{\ind}{\0}
\newcommand*{\sub}{\1}

\newcommand*{\Par}{\m(}	    
\newcommand*{\Sqpar}{\m[}	
\newcommand*{\Brace}{\m.}	
\newcommand*{\Braket}{\m<}	
\newcommand{\probm}[2][P]{\prob[#1]\m({#2}}
\newcommand{\Probm}[2][P]{\Prob[#1]\m({#2}}
\newcommand{\exptm}[2][E]{\prob[#1]\m[{#2}}
\newcommand{\Exptm}[2][E]{\Prob[#1]\m[{#2}}

\newcommand{\Pcond}{\mm(}	
\newcommand{\Econd}{\mm[}	

\newcommand*{\Map}{\mathfrak} 

\newcommand*{\gmap}{\Map{g}} 
\newcommand*{\tmap}{\Map{t}} 
\newcommand*{\fmap}{\Map{F}} 
\newcommand*{\qmap}{\Map{q}} 

\newcommand*{\ml}[1][]{(\map#1,\loopconf#1)}
\newcommand*{\ql}[1][]{(\qmap#1, \boldsymbol{\ell}#1)}
\newcommand*{\mg}[1][]{(\map#1,\gmap#1)}
\newcommand*{\bt}[1][]{(\tmap#1,\sigma#1)}

\newcommand*{\Op}[1][p]{\mathcal{LQ}_{#1}}
\newcommand*{\tmaps}{\mathcal{T}}
\newcommand*{\bts}{\mathcal{B}\tmaps}

\newcommand*{\nb}[1]{\texttt\#\,\text{#1}}

\newcommand*{\qgn}[1][\qseq]{#1,g,n}
\newcommand*{\Rq}{R_{\qseq}}
\newcommand*{\uq}{u_{\qseq}}

\newcommand*{\dego}{\deg_\mathtt{out}}
\newcommand{\mujs}{\mu_{\mathrm{JS}}}

\newcommand*{\W}{\mathcal{W}}
\newcommand*{\wt}{\widetilde}


\newcommand*{\w}{w_{\hqseq}}
\newcommand*{\Wp}[1][p]{F_{#1}(\hqseq)}
\newcommand*{\Wpp}[1][p]{F^\bullet_{#1}(\hqseq)}
\newcommand*{\wbdg}{w_{\hqseq}^{\text{BDG}}}
\newcommand*{\wjs}{w_{\hqseq}^{\text{JS}}}

\newcommand*{\TJS}{\mathfrak{T}^{\mathrm{JS}}}

\newcommand*{\Sj}{\overline{S}_\gamma}
\newcommand*{\gen}{\mathfrak{f}}
\newcommand*{\nongen}{\mathfrak{g}}
\newcommand*{\Wi}{\mathfrak{h}}

\newcommand*{\oseq}{\mathbf0}
\newcommand*{\dseq}{\boldsymbol\delta}
\newcommand*{\ddseq}{\dseq_1-\dseq_k}

\newcommand*{\Wqgn}[1][\qseq]{\W_{\qgn[#1]}}
\newcommand*{\Rqgn}[1][\qseq]{\rho_{\qgn[#1]}}
\newcommand*{\Wj}[1][\qseq]{\W\0\gamma_{\qgn[#1]}}
\newcommand*{\Rj}[1][\qseq]{\rho\0\gamma_{\qgn[#1]}}

\newcommand*{\jmax}{g^{-1/2}}
\newcommand*{\jint}{(0,\jmax}

\newcommand*{\intI}{\int_{-1}^1}
\newcommand*{\intj}{\int_{-\gamma}^\gamma}

\newcommand*{\dom}{\mathbb{D}}
\newcommand*{\Dall}{\mathcal{D}}

\newtheorem{citetheorem}{Theorem}
\newtheorem{citeproposition}[citetheorem]{Proposition}
\newtheorem{citecorollary}[citetheorem]{Corollary}
\newtheorem{citelemma}[citetheorem]{Lemma}
\newtheorem*{remark*}{Remark}

\newcommand*{\res}{\mathrm{Res}}

\newcommand*{\intC}{\oint_{\mathcal{C}_0}}
\renewcommand*{\rt}[1][x]{\sqrt{\gamma^2-#1^2}}
\newcommand*{\rto}[1][x]{\sqrt{1-\frac{\gamma^2}{#1^2}}}
\newcommand*{\im}[1][x]{\sgn(\mathrm{Im}\,#1)}


Theorem~\ref{thm:phasediagram} establishes the equations and inequalities that determine the phase diagram of the Boltzmann loop-decorated maps. These equations and inequalities are derived from the so-called \emph{resolvent function}
\begin{equation*}
	\Wqgn(x) := 1+\sum_{p\ge 1} x^{-2p} F\0p(\qgn)\,.
\end{equation*}
The proof of Theorem~\ref{thm:phasediagram} is based on a close analysis of the set of solutions to the equation satisfied by the resolvent, which we give in the following proposition.

For $\gamma> 0$, let $S_\gamma=\ocomplex\setminus[-\gamma,\gamma]$, where $\ocomplex=\complex \cup \{\infty\}$ is the Riemann sphere. The set $S_\gamma$ is topologically an open disk. We denote by $\Sj$ the closed topological disk obtained by gluing a circle to the boundary of $S_\gamma$. It can be parametrized as $\Sj= (\ocomplex\setminus(-\gamma,\gamma)) \cup (\{\pm1\} \times (-\gamma,\gamma))$, where an element of $\{\pm1\} \times (-\gamma,\gamma)$ is denoted by $x\pm i0$ with $x\in(-\gamma,\gamma)$.
A function $\W:\Sj \to \complex$ is said to be \emph{holomorphic on $\Sj$} if it is holomorphic on $S_\gamma$ and continuous on $\Sj$. A such $\W$ must be bounded because $\Sj$ is compact.
We define the \emph{spectral density} associated to $\W$ by $\rho(x) = \frac{\W(\gamma x-i0)-\W(\gamma x+i0)}{2\pi i x}$, where $x\in[-1,1]$.\footnote{This definition differs slightly from the one used in \cite{borot_recursive_2012} and is related to it by $x\W\1{BBG}(x) = \W(x)$ and $\rho\1{BBG}(x) = \gamma\cdot\rho(\gamma^{-1} x)$.} Since $\W$ is continuous at $\pm\gamma$, the spectral density $\rho$ must vanish at $x=\pm1$. 
Finally, let
\begin{equation*}
\Dall = \set{(\qgn,\gamma) \in \dom \times (0,\infty)}{\gamma \le \jmax}
\end{equation*}

\begin{proposition}[Equation of the resolvent \cite{borot_recursive_2012}] \label{prop:free equation}~
\begin{enumerate}
\item
If $(\qgn)\in \Ddomain$ is admissible, then there exists $\gamma\in \jint]$ such that $\W \equiv \Wqgn$ satisfies
\begin{equation}
\begin{cases}
\begin{aligned}
\W\text{ is an even holomorphic function on }\Sj&\text{ such that for all }x\in(-\gamma,\gamma),
\\	\W(x-i0) + \W(x+i0) + n \W\mB({ \frac1{hx} } &= n+x^2 - \sum_{k=1}^\infty q_k x^{2k}\,.
\end{aligned}
\end{cases}
\label{eq:free equation}
\end{equation}
Moreover, $\Rqgn$ is a non-negative even function on $[-1,1]$.
\item
For all $(\qgn,\gamma)\in \Dall$, the system \eqref{eq:free equation} has a unique solution, which we denote by $\Wj$.
The associated spectral density $\Rj(x)$ is jointly continuous with respect to $(x;\qgn,\gamma)\in [-1,1]\times \Dall$.
\end{enumerate}
\end{proposition}

Remark that the function $\W$ is determined by its spectral density via $[x^{-2p}]\W(x) = \gamma^{2p} \intI x^{2p} \rho(x)\dd x$. This also implies that if $\rho$ is non-negative on $[-1,1]$, then $[x^{-2p}]\W(x)\ge 0$ for all $p$.
In \cite{borot_recursive_2012} the equation \eqref{eq:free equation} was solved explicitly using a change of variable involving elliptic functions. The continuity in
Proposition~\ref{prop:free equation} (ii) follows from the explicit expression of the function $(x;\qgn,\gamma)\mapsto \Rj(x)$.

To find $\Wqgn$ for a particular triple $(\qgn)\in \Ddomain$, one still needs a way to determine whether $(\qgn)$ is admissible, and if so, to select the value of $\gamma$ such that $\Wj=\Wqgn$. Observe that $\Wqgn(\infty)=1$ by definition. It turns out that this condition and the positivity of $\Rqgn$ suffice to determine $\Wqgn$:

\begin{proposition}[Admissibility criterion]\label{prop:admiss cond}
$(\qgn)\in \dom$ is admissible if and only if for some $\gamma \in \jint]$,
\begin{align}\label{eq:tau}
\Wj(\infty)=1 \qtq{and} \Rj(x)\ge 0 \text{ for all }x\in[-1,1] \,.
\end{align}
In this case, the above condition uniquely determines $\gamma$, and we have $\Wj=\Wqgn$.
\end{proposition}

In practice, it is difficult to check directly whether $\Rj\ge 0$ on the whole interval $[-1,1]$. On the other hand, it is relatively easy to compute the first terms in the asymptotic expansion of $\Rj$ at $\pm1$. It turns out that the sign of the leading term determines completely whether $\Rj\ge 0$ on all $[-1,1]$.

\begin{proposition}[Positivity criterion]\label{prop:pos bootstrap}
Let $(\qgn,\gamma)\in \Dall$. Recall that $b=\frac1\pi \arccos(\frac n2) \in (0,\frac12)$.
\begin{enumerate}
\item
When $\gamma<\jmax$, the asymptotic expansion of $\Rj(x)$ at $x=\pm1$ is of the form
\begin{equation}\label{eq:rho asymp gen}
\Rj(x) = \gen \sqrt{1-x^2} + \gen' (1-x^2)^{3/2} +o((1-x^2)^{3/2}),
\end{equation}
where $\gen, \gen'$ are functions of $(\qgn,\gamma)$ such that $\Rj(x)\ge 0$ for all $x\in [-1,1]$ \Iff\ $\gen \ge 0$. Moreover, we have $\gen'>0$ when $\gen=0$.
\item
When $\gamma=\jmax$, the asymptotic expansion of $\Rj(x)$ at $x=\pm1$ is of the form
\begin{equation}\label{eq:rho asymp nongen}
\Rj(x) = \nongen (1-x^2)^{1-b} + \nongen' (1-x^2)^{1+b} +o((1-x^2)^{1+b}),
\end{equation}
where $\nongen, \nongen'$ are functions of $(\qgn)$ such that $\Rj(x)\ge 0$ for all $x\in [-1,1]$ \Iff\ $\nongen \ge 0$. Moreover,  we have $\nongen'>0$ when $\nongen=0$ unless $q_k=\delta_{k,1}$ for all $k\ge 1$.
\end{enumerate}
\end{proposition}

\begin{proof}[Proof of Theorem~\ref{thm:phasediagram}]
Define $\Wi(\qgn,\gamma) \equiv \Wj(\infty)$. Then Proposition~\ref{prop:admiss cond} and \ref{prop:pos bootstrap} imply that a triple $(\qgn)\in \Ddomain$ is admissible \Iff\ there exists some $\gamma \in \jint]$ such that $\Wi(\qgn,\gamma)=1$ and $\gen(\qgn,\gamma)\ge 0$ (when $\gamma<\jmax$) or $\nongen(\qgn)\ge 0$ (when $\gamma=\jmax$). 

According to Proposition~\ref{prop:pos bootstrap}, in the four cases $\gen>0$, $\gen=0$, $\nongen>0$ and $\nongen=0$, the dominant term in the asymptotic expansion of $\Rj(x)$ at $x=\pm 1$ are respectively of order $(1-x^2)^{1/2}$, $(1-x^2)^{3/2}$, $(1-x^2)^{1-b}$ and $(1-x^2)^{1+b}$. Applying Laplace's method to the relation $F\0p(\qgn) = \gamma^{2p} \intI x^{2p} \Rqgn(x)\dd x$, it follows that the exponents $\alpha$ in the asymptotics \eqref{eq:Fasymp} of $F\0p(\qgn)$ are respectively $1$, $2$, $3/2-b$ and $3/2+b$ in the four cases. The case where $q_k=\delta_{k,1}$ in Proposition~\ref{prop:pos bootstrap} (ii) could \emph{a priori} give rise to a different exponent. However one can check that it does not correspond to an admissible triple $(\qgn)$ for any $g$ and $n$. This proves Theorem~\ref{thm:phasediagram} provided that Proposition~\ref{prop:admiss cond} and \ref{prop:pos bootstrap} are true.
\end{proof}

The asymptotic expansions of $\Rj$ in Proposition~\ref{prop:pos bootstrap} were derived in \cite{borot_recursive_2012} using the explicit formula of $\Rj$. Both Proposition~\ref{prop:admiss cond} and \ref{prop:pos bootstrap} were conjectured in \cite{borot_recursive_2012} with the support of numerical evidence. They were then used in conjunction with the explicit expression of $\Rj$ to determine the phase diagram of the Boltzmann loop-decorated quadrangulation, which we include here as Figure~\ref{fig:phase diagram}.
In the rest of this appendix we will first prove Proposition~\ref{prop:admiss cond} and Proposition~\ref{prop:pos bootstrap}.

\begin{figure}[t]
\centering
\includegraphics[scale=0.85,page=3]{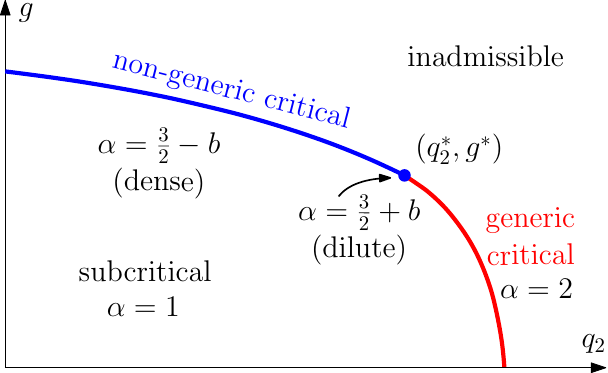}
\caption{Phase diagram of the loop-decorated quadrangulation model \cite{borot_recursive_2012} for a given value of $n \in (0,2)$ (the diagram is qualitatively the same for all such $n$). The critical phases form a line that separates the sub-critical region (below) and the inadmissible region (above).}\label{fig:phase diagram}
\end{figure}

\subsection{Proof of Proposition~\ref{prop:admiss cond}}\label{sec:fct eq derivation}

The proof of Proposition~\ref{prop:admiss cond} is based on the following lemma, which is basically a reformulation of classic results on the enumeration of bipartite maps without loops.

\begin{lemma}\label{lem:loopless admis}
For all $\hqseq \in [0,\infty)^\natural$ and $\gamma>0$, consider the system
\begin{equation}\label{eq:loopless eq}
\left\{
\begin{aligned}
\W\text{ is an even holomorphic function on }\Sj&\text{ such that for all }x\in(-\gamma,\gamma),
\\	\W(x-i0) + \W(x+i0) &= x^2 - \sum_{k=1}^\infty \hat q_k x^{2k}\,.
\end{aligned}
\right.
\end{equation}
\begin{enumerate}
\item
The solution of \eqref{eq:loopless eq}, when it exists, is unique and given by 
\begin{equation}\label{eq:loopless sol}
\W\0\gamma_{\hqseq}(\gamma x) := \frac{x^2}{2\pi} \sqrt{1-\frac1{x^2}} \intI \frac{\sum_{k=1}^\infty (\delta_{k,1} - \hat q_k) \gamma^{2k} y^{2k}}{x^2-y^2} \frac{\dd y}{\sqrt{1-y^2}}
\end{equation}
for all $x\in \complex \setminus [-1,1]$.
\item
The sequence $\hqseq$ is admissible if and only if there exists $\gamma > 0$ such that 
\begin{equation}\label{eq:loopless admis}
\W\0\gamma_{\hqseq}(\infty) = 1    \qtq{and}
\lim \limits_{x\to 1^-} { \rho\0\gamma_{\hqseq}(x) }/{\sqrt{1-x^2}} \ge 0 \,.
\end{equation}
In this case, $\gamma$ is determined by these conditions, and we have $\W\0\gamma_{\hqseq}(x) = 1 + \sum_{k=1}^\infty W\0k(\hqseq) x^{-2k}$.
\end{enumerate}
\end{lemma}

\begin{proof}
(i) If $\W$ is a solution of \eqref{eq:loopless eq}, then for all $x\in \complex \setminus [-1,1]$,
\begin{equation*}
\mathcal I(x):= 
\intI \frac{\W(\gamma y-i0) + \W(\gamma y+i0)}{x^2-y^2} \frac{\dd y}{\sqrt{1-y^2}}
\ =\
\intI \frac{\sum_{k=1}^\infty (\delta_{k,1} - \hat q_k) \gamma^{2k} y^{2k}}{x^2-y^2} \frac{\dd y}{\sqrt{1-y^2}} \,.
\end{equation*}
It suffices to show that
\begin{equation}\label{eq:loopless int}
\mathcal I(x) = \frac{2\pi}{ x^2 \sqrt{1-\frac1{x^2}} } \W(\gamma x) \,.
\end{equation}
$\mathcal I(x)$ can be viewed as the integral of $F(y):=\frac{\W(\gamma y)}{x^2-y^2} \frac1{\sqrt{1-y^2}}$ on the union of $[-1,1]-i0$ and $[-1,1]+i0$.
Assume that $\mathrm{Im}(x)>0$, and consider the determination of square root such that $\mathrm{Re}(\sqrt t)>0$ for all $t\in \complex \setminus (-\infty,0]$. Then $F$ is a meromorphic function on $\complex\setminus\real$ with two simple poles at $\pm x$. By deforming the contours of integration from $[-1,1]\pm i0$ to $\mathcal{C}^\pm$ as in Figure~\ref{fig:contour-deformation}, we get
\begin{align*}
\mathcal I(x)\ =\ \m({\int_{\mathcal C_+}+\int_{\mathcal C_-}} F(y)\,\dd y + 2\pi i\, \m({ \res_{y\to x}-\res_{y\to-x} } F(y)\,.
\end{align*}
The factor $\sqrt{1-y^2}$ changes sign when $y$ crosses $\real \setminus [-1,1]$. It follows that $F(y+i0)+F(y-i0)=0$ for all $y \in \real \setminus [-1,1]$. Hence the integrals on the straight line parts of $\mathcal{C}_+$ and $\mathcal{C}_-$ cancel each other. On the other hand, $F(y) = O(\abs{y}^{-3})$ when $\abs{y}\to\infty$, hence $\lim_{R\to\infty} \m({\int_{\mathcal{C}_+} + \int_{\mathcal{C}_-}}F(y)\dd y=0$.
As $\W$ is even, we have $\res_{y\to-x}F(y) = - \res_{y\to x}F(y) = \frac1{2x}\frac{\W(\gamma x)}{\sqrt{1-x^2}}$. Therefore $\mathcal I(x) = - \frac{2\pi i}x \frac{\W(\gamma x)}{\sqrt{1-x^2}}$ for all $x\in \complex$ such that $\mathrm{Im}(x)>0$. It is not hard to see that \eqref{eq:loopless int} gives the correct analytic extension of $\mathcal I(x)$ on $\complex \setminus [-1,1]$.

\begin{figure}[t!]
\centering
\includegraphics[scale=1.2,page=2]{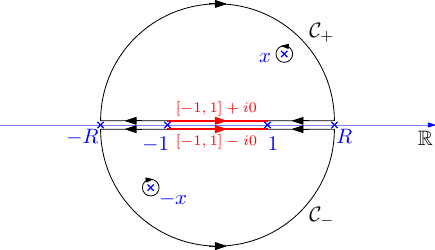}
\caption{Deformation of the contour of integration of the function $F(y) = \frac1{y^2-x^2} \frac{\W(y)}{\rt[y]}$.}
\label{fig:contour-deformation}
\end{figure}

(ii) 
It has been shown in \cite[Section 6]{borot_more_2012} that if $\hqseq$ is admissible, then $\W_\hqseq(x):= 1 + \sum_{k=1}^\infty W\0k(\hqseq) x^{-2k}$ is the solution of \eqref{eq:loopless eq} for some $\gamma>0$, and the associated spectral density is non-negative over $[-1,1]$. In particular, this $\gamma$ satisfies \eqref{eq:loopless admis} and we have $\W\0\gamma_{\hqseq} = \W_{\hqseq}$.

Inversely, assume that there exists $\gamma>0$ which satisfies \eqref{eq:loopless admis}.  First let us compute the spectral density associated to \eqref{eq:loopless sol}. Thanks to the identity $\frac{x^2}\pi \sqrt{1-\frac1{x^2}} \intI \frac1{x^2-y^2} \frac{\dd y}{\sqrt{1-y^2}} = 1$ for all $x\in \complex\setminus [-1,1]$, we have
\begin{equation*}
\W\0\gamma_{\hqseq} (\gamma x) - \sum_{k=1}^\infty (\delta_{k,1}-\hat q_k) \gamma^{2k} \frac{x^{2k}}2 \ =\ 
- \frac{x^2}{2\pi} \sqrt{1-\frac1{x^2}} \intI \sum_{k=1}^\infty (\delta_{k,1} - \hat q_k) \gamma^{2k} \frac{x^{2k}-y^{2k}}{x^2-y^2} \frac{\dd y}{\sqrt{1-y^2}} \,.
\end{equation*}
The fact that \eqref{eq:loopless eq} has a solution implies $\sum_{k=1}^\infty \hat q_k \gamma^{2k}<\infty$. Hence both the second term on the left hand side and the integral on the right hand side are analytic functions in the unit disk. The determination of the square root prescribes that $x \sqrt{1-(x\pm i0)^{-2}} = \pm i \sqrt{1-x^2}$ for all $x\in (-1,1)$. It follows that 
\begin{equation}\label{eq:loopless rho}
\rho\0\gamma_{\hqseq} (x) = 
\frac{\sqrt{1-x^2}}{2\pi^2} \intI \sum_{k=1}^\infty (\delta_{k,1} - \hat q_k) \gamma^{2k} \frac{x^{2k}-y^{2k}}{x^2-y^2} \frac{\dd y}{\sqrt{1-y^2}}
\end{equation}
for all $x\in (-1,1)$.

Thanks to the absolute convergence of $\sum_{k=1}^\infty (\delta_{k,1} - \hat q_k) \gamma^{2k}$, one can take the limits $x\to\infty$ and $x\to 1^-$ respectively in \eqref{eq:loopless sol} and \eqref{eq:loopless rho}, then interchange the summation with the integrals. This gives
\begin{align*}
\W\0\gamma_{\hqseq}(\infty) \ &= \
\frac1{2\pi} \sum_{k=1}^\infty (\delta_{k,1}-\hat q_k) \gamma^{2k} \intI \frac{y^{2k} \dd y}{\sqrt{1-y^2}} \ =\ 
\sum_{k=1}^\infty (\delta_{k,1}-\hat q_k) \frac12 \binom{2k}k \m({ \frac{\gamma^2}4 }^k    \qt{and}
\\
\lim _{x\to 1^-} \frac{ \rho\0\gamma_{\hqseq}(x) }{\sqrt{1-x^2}} \ &=\ 
\frac1{2\pi^2} \sum_{k=1}^\infty (\delta_{k,1} - \hat q_k) \gamma^{2k} \intI  \frac{1-y^{2k}}{1-y^2} \frac{\dd y}{\sqrt{1-y^2}} \ =\ 
\frac1\pi \sum_{k=1}^\infty (\delta_{k,1} - \hat q_k) k \binom{2k}k \m({\frac{\gamma^2}4}^k .
\end{align*}
Let $u_{\hqseq}(r) = r-\sum_{k=1}^\infty \frac12 \binom{2k}k \hat q_k r^k$.
The two right hand sides above are respectively $u_{\hqseq}(\gamma^2/4)$ and $\frac{\gamma^2}{2\pi} u_{\hqseq}' (\gamma^2/4)$. Therefore $\gamma$ satisfies \eqref{eq:loopless admis} \Iff\ $r=\gamma^2/4$ satisfies $u_{\hqseq}(r)=1$ and $u_{\hqseq}'(r)\ge 0$. Since $u_{\hqseq}$ is a concave function, these conditions uniquely determines $\gamma^2/4$ as the smallest positive solution of the equation $u_{\hqseq}(r)=1$. 
It is shown in \cite[Proposition 1]{marckert_bipartite_2007} that the weight sequence $\hqseq$ is admissible \Iff\ such a solution exists.
This completes the proof of the part (ii) of the lemma.
\end{proof}

\begin{proof}[Proof of Proposition~\ref{prop:admiss cond}]
If $(\qgn) \in \Ddomain$ is admissible, Proposition~\ref{prop:free equation} implies that there exists $\gamma \in \jint]$ such that $\Wj=\Wqgn$, which satisfies the condition \eqref{eq:tau}. This $\gamma$ is also unique because $(-\gamma,\gamma)$ is the locus of discontinuity of $\Wqgn$.

Inversely, assume there exists $\gamma\in \jint]$ that satisfies \eqref{eq:tau}. Let $F\0k_\gamma = [x^{-2k}]\Wj(x)$ be the coefficients of the Taylor expansion of $\Wj$ at $x=\infty$. According to the remark after Proposition~\ref{prop:free equation}, the assumption $\Rj\ge 0$ implies that $F\0k_\gamma \ge 0$ for all $k\ge 1$. In particular, $\hat{q}_k := q_k + n \,g^{2k} F\0k_\gamma$ defines a non-negative sequence. Since $\Wj(\infty)=1$, we have
\begin{equation}\label{eq:effq gfun}
\sum_{k=1}^\infty \hat{q}_k x^{2k} =  \sum_{k=1}^d q_k x^{2k} + n\,\Wj\mB({ \frac1{gx} } - n \,.
\end{equation}
Plug this into \eqref{eq:free equation}, we see that $\Wj = \W\0\gamma_{\hqseq}$. By Lemma~\ref{lem:loopless admis}(ii), the assumption \eqref{eq:tau} implies that the sequence $\hqseq$ is admissible and $F\0k_\gamma = W\0k(\hqseq)$ for all $k\ge 1$.
It follows that $q_k=\hat q_k - n g^{2k} W\0k(\hqseq)$. According to Theorem~\ref{thm:admissibility}, this implies that the triple $(\qgn)$ is admissible.
\end{proof}

\subsection{Proof of Proposition~\ref{prop:pos bootstrap} (i) --- the subcritical and generic critical cases}

\newcommand*{\rj}{\tilde\rho\0\gamma_{\qgn}}

Our approach is based on the following integral equation on the spectral density. It is derived from the functional equation \eqref{eq:free equation} using elementary complex analysis.

\begin{lemma}[Integral equation for $\Rj$]\label{lem:rho bootstrap}
For $(\qgn,\gamma)\in \Dall$, let $\tau=\gamma^2 g \in (0,1]$. Then we have
\begin{equation}\label{eq:rho integral}
\frac{2\pi\,\Rj(x)}{\sqrt{1-x^2}}\ =\ (1-q_1)\gamma^2 - \sum_{k=2}^d q_k \gamma^{2k} I_k(x) - n \intI \frac{\tau^2y^2}{1-\tau^2x^2y^2} \frac{\Rj(y)\,\dd y}{\sqrt{1-\tau^2 y^2}}
\end{equation}
for all $x\in[-1,1]$, where $I_k(x) = \frac1\pi\!\intI\! \frac{x^{2k}-y^{2k}}{x^2-y^2}\!\frac{\dd y}{\!\sqrt{1-y^2}}$ is an even polynomial with positive coefficients.
\end{lemma}

\begin{proof}
Like in the proof of Proposition~\ref{prop:admiss cond}, we have
$\Wj = \W\0\gamma_{\hqseq}$ for $\hqseq$ related to $(\qgn)$ by \eqref{eq:effq gfun}.
Plug this into the integral expression \eqref{eq:loopless sol} of $\W\0\gamma_{\hqseq}$, we get that for all $x\in \complex \setminus [-1,1]$,
\begin{equation*}
\Wj(\gamma x) = \frac{x^2}{2\pi} \sqrt{1-\frac1{x^2}} \m({ \intI \frac{\sum_{k=1}^d (\delta_{k,1}-q_k) \gamma^{2k} y^{2k}}{x^2-y^2} \frac{\dd y}{\sqrt{1-y^2}}
- n \intI \frac{\Wj \m({\frac1{g\gamma y}} -1}{x^2-y^2} \frac{\dd y}{\sqrt{1-y^2}} } \,.
\end{equation*}
The first integral has already appeared in the proof of Lemma~\ref{lem:loopless admis} and give rise to the spectral density
\begin{equation}\label{eq:rho n=0}
\rho\0\gamma_{\qgn=0} (x) = \frac{\sqrt{1-x^2}}{2\pi} \mh({ (1-q_1)\gamma^2 - \sum_{k=2}^d q_k \gamma^{2k} I_k(x) } \,.
\end{equation}
Using again that $y \sqrt{1-(y\pm i0)^{-2}} = \pm i\sqrt{1-y^2}$ for all $y\in (-1,1)$, the second integral 
can be viewed as the integral of $G(y) := \frac{\W(\gamma/(\tau y)) -1}{x^2-y^2} \frac i{2y\sqrt{1-1/y^2}}$ on the counter-clockwise contour around the cut $[-1,1]$. The function $G$ is meromorphic on $\complex \setminus \{y\in \real:|y|\le 1 \text{ or }|y|\ge \frac1\tau\}$ and has two simple poles with opposite residues at $\pm x$. Therefore we can deform the contour of integration to the unit circle $\mathcal C$ (assuming that $x$ is not on $\mathcal C$). After the change of variable $z=1/(\tau y)$, we obtain
\begin{equation*}
\intI \frac{\Wj \m({\frac{\gamma}{\tau y}} -1}{x^2-y^2} \frac{\dd y}{\sqrt{1-y^2}} =
\oint_{\frac1\tau \mathcal C} \frac{\Wj(\gamma z)-1}{x^2-1/(\tau z)^2} \, \frac{i\, \tau z}{2\sqrt{1-\tau^2 z^2}} \, \frac{\dd z}{\tau z^2} 
\ .
\end{equation*}
The integral on the right hand side can be evaluated by deforming the contour $\frac1\tau \mathcal C$ back to the counter-clockwise contour around $[-1,1]$, and observe that $\Wj$ is the only term that is discontinuous along the cut $[-1,1]$. Formally this amounts to replacing $\oint_{\frac1\tau \mathcal C}$ by $\intI$ and replacing $\Wj(\gamma z)-1$ by $2\pi i\, z\, \Rj(z)$. After rearranging terms, we get
\begin{equation*}
\intI \frac{\Wj \m({\frac\gamma{\tau y}} -1}{x^2-y^2} \frac{\dd y}{\sqrt{1-y^2}}
= \pi \intI \frac{\tau^2 y^2}{1-\tau^2 x^2 y^2} \, \frac{\Rj(y) \, \dd y}{\sqrt{1-\tau^2 y^2}} \ .
\end{equation*}
Notice that the right hand side is a holomorphic function of $x$ on $[-1,1]$. Plug it into the expression of $\Wj(\gamma x)$ at the beginning of the proof and take the spectral density on both sides, then we obtain the integral equation \eqref{eq:rho integral} instead of \eqref{eq:rho n=0}.
\end{proof}

By expanding the integral on the right hand side of \eqref{eq:rho integral} as a power series in $x$, we see that \eqref{eq:rho integral} defines an analytic continuation of $x\mapsto \Rj(x)/\sqrt{1-x^2}$ on the disk $\{x\in \complex: |x|< \tau^{-1}\}$. We denote by $\rj$ this analytic continuation.
When $\gamma<\jmax$, that is $\tau<1$, the function $\rj$ is analytic at $x=1$. This implies the expansion \eqref{eq:rho asymp gen} of $\Rj$ in Proposition~\ref{prop:pos bootstrap} (i). In particular, $\gen(\qgn,\gamma)=\rj(\pm 1)$.
\mbox{In the sequel} we will use the following short hand notations
\begin{equation*}
\Dall^< = \m.{(\qgn,\gamma)\in \Dall : \gamma<\jmax}     \qtq{and}
\{\gen > 0\} = \m.{(\qgn,\gamma)\in \Dall^< : \gen(\qgn,\gamma)>0} .
\end{equation*}
Thanks to the continuity of the function $(x;\qgn,\gamma)\mapsto \Rj(x)$ in Proposition~\ref{prop:free equation} (ii), Equation \eqref{eq:rho integral} implies that $(x;\qgn,\gamma)\mapsto \rj(x)$ is also continuous for all $|x|<\tau^{-1}$ and $(\qgn,\gamma)\in \Dall$, in particular for all $(x;\qgn,\gamma)\in [-1,1]\times \Dall^<$. It follows that $\gen$ is continuous on $\Dall^<$.

Our proof of Proposition~\ref{prop:pos bootstrap} relies on a continuity argument that bootstraps the positivity of $\Rj$ near $x=1$ to its positivity on $[-1,1]$. It is based on the following simple but important consequence of the integral equation \eqref{eq:rho integral}.

\begin{lemma}
\label{lem:no int vanish}
For all $(\qgn,\gamma)\in \{\gen>0\}$, if $\Rj \ge 0$ on $[-1,1]$, then $\rj>0$ on $[-1,1]$.
\end{lemma}

\begin{proof}
When $\Rj\ge 0$ on $[-1,1]$, the right hand side of \eqref{eq:rho integral} shows that $\rj$ is increasing on $[-1,0]$ and decreasing on $[0,1]$. Hence $\rj(\pm 1) = \gen(\qgn,\gamma)>0$ implies that $\rj>0$ on $[-1,1]$. 
\end{proof}

To apply our continuity argument, we also need to know that:

\newcommand*{\A}{\mathcal{A}}

\begin{lemma}
The set $\{\gen>0\}$ is connected.
\end{lemma}

\begin{proof}
For any fixed $(g,n,\gamma)$, the spectral density $\Rj$ depends linearly on the polynomial $P(x)=n+\sum_k (\delta_{k,1}-q_k) x^{2k}$ on the right hand side of \eqref{eq:free equation}.
When $P(x)=n$ (which corresponds to the non-admissible weights $q_k=\delta_{k,1}$), the solution of \eqref{eq:free equation} would be $\W(x)= \frac{n}{n+2}$, which has zero spectral density. Thus $\Rj$ depends linearly on the coefficients $(1-q_1,q_2,\dots,q_d)$. It follows that \mbox{for all $(\qgn,\gamma)\in \Dall^<$,}
\begin{equation}\label{eq:f lin}
\gen(\qgn,\gamma) = (1-q_1)\, \gen(\qgn[\oseq],\gamma) - \sum_{k=2}^d q_k \, \gen(\qgn[\ddseq],\gamma) \ .
\end{equation}
where $\oseq$ is the zero sequence, and $\dseq_k$ is the sequence whose $k$-th term is 1 and all other terms are zero.

Following the methods of \cite{borot_recursive_2012}, one can obtain a simple enough formula for $\gen(\qgn[\oseq],\gamma)$ and show that it is always positive. But going into the details will lead us too far, so we will simply admit the following:

\begin{lemma}\label{lem:f pos}
For all $g>0$, $n\in(0,2)$ and $\gamma \in \jint)$, we have $\gen(\qgn[\oseq],\gamma)>0$.
\end{lemma}

\noindent
Then the decomposition \eqref{eq:f lin} implies that every point $(\qgn,\gamma)$ in $\{\gen>0\}$ is connected to $(\qgn[\oseq],\gamma)$ by a line segment in $\{\gen>0\}$. The points of the form $(\qgn[\oseq],\gamma)$ are obviously connected to each other. Therefore $\{\gen>0\}$ is connected.
\end{proof}

\begin{remark*}
Physically, Lemma~\ref{lem:f pos} means that a fully-packed model (this means $\qseq=\oseq$, i.e.\ every internal face is visited by a loop) is never generic critical (i.e.\ $\gen=0$). However the lemma is a bit stronger than the previous statement because it also includes non-physical values of the parameter $\gamma$.
\end{remark*}

\begin{lemma}[Positivity bootstrap: $\gamma\!<\!\jmax$]\label{lem:pos boot gen}
$\Rj\!\ge\! 0$ on $[-1,1]$ for all $(\qgn,\gamma)$ in the closure of $\{\gen>0\}$.
\end{lemma}

\begin{proof}
Consider $\A = \Setb{(\qgn,\gamma)\in \{\gen>0\}}{\Rj(x)\ge 0\text{ for all }x\in[-1,1]}$. By comparing $F_k(\qgn)$ to the partition function of non-decorated maps, we see that
$(\qgn)$ is admissible and subcritical when $n<1$ and $\max(h,q_1,\ldots,q_d)$ is small enough. Then Proposition \ref{prop:free equation} provides a $\gamma$ such that $\Rj\ge 0$ and $\gen(\qgn,\gamma)>0$. Therefore $\A\neq \varnothing$.

Now fix $(\qgn,\gamma)\in\A$. By Lemma~\ref{lem:no int vanish}, we have $\rj>0$ on $[-1,1]$. Since $[-1,1]$ is compact and $(x;\qgn,\gamma) \mapsto \rj(x)$ is continuous, this implies that $\tilde \rho\0{\gamma'}_{\qseq',g',n'}>0$ on $[-1,1]$ for all $(\qseq',g',n',\gamma')$ close enough to $(\qgn,\gamma)$. Therefore $\A$ is open in $\{\gen>0\}$.
But since $(\qgn,\gamma)\mapsto \Rj(x)$ is continuous for each $x\in[-1,1]$, $\A$ is also closed in $\{\gen>0\}$. Since $\{\gen>0\}$ is connected, we conclude that $\A=\{\gen>0\}$, that is, $\Rj\ge 0$ for all $(\qgn,\gamma)\in \{\gen>0\}$. By the continuity of $(x;\qgn,\gamma) \mapsto \Rj(x)$, the same holds for all $(\qgn,\gamma)$ in the closure of $\{\gen>0\}$.
\end{proof}

\begin{proof}[Proof of Proposition~\ref{prop:pos bootstrap} (i)]
The asymptotic expansion of $\Rj(x)$ when $\gamma<\jmax$ has been derived in the discussion above Lemma~\ref{lem:no int vanish}. 

Since $\gen$ is continuous on $\Dall^<$, the linear decomposition \eqref{eq:f lin} implies that the set $\{ (\qgn,\gamma)\in \Dall^< : \gen(\qgn,\gamma)\ge 0\}$ is contained in the closure of $\{\gen>0\}$. Then it follows from Lemma~\ref{lem:pos boot gen} that for all $(\qgn,\gamma)\in \Dall^<$, $\gen(\qgn,\gamma)\ge 0$ implies that $\Rj\ge 0$ on $[-1,1]$. The converse is obviously true.

Finally, when $\gamma<\jmax$ and $\gen(\qgn,\gamma)=0$, the previous paragraph shows that $\Rj\ge 0$ on $[-1,1]$. Then it is not hard to see that the right hand side of the integral equation \eqref{eq:rho integral} has a strictly negative derivative at $x=1$. This is equivalent to $\gen'>0$ in the asymptotic expansion of $\Rj$.
\end{proof}

\subsection{Proof of Proposition~\ref{prop:pos bootstrap} (ii) --- the non-generic critical cases}
\label{sec:pos boot nongen}

\newcommand*{\Pqgn}[1][\qseq]{P_{\qgn[#1]}}
\newcommand*{\B}{\mathcal{B}}

When $\gamma=\jmax$, the integral equation \eqref{eq:rho integral} remains valid, but is less relevant because $\Rj(x)/\sqrt{1-x^2}$ no longer has an analytic continuation at $x=1$. Instead of \eqref{eq:rho integral}, it is much easier to derive results directly from the explicit expression of $\Rj(x)$, as it is relatively simple in this case:

\begin{lemma}
\label{lem:non-gen resolvent}
For $(\qgn)\in \dom$ and $\gamma=\jmax$, we have
\begin{equation}\label{eq:Rj}
\Rj(x) = \frac1{2 \pi \sqrt{4-n^2}} \frac1x \m({ L_{\qgn}(x) \mB({\frac{1-x}{1+x}}^b - L_{\qgn}(-x) \mB({\frac{1+x}{1-x}}^b }
\end{equation}
where $L_{\qgn}(x) = \Pqgn(x)-\Pqgn(x^{-1})$ and $\Pqgn$ is the polynomial with no constant term determined by
\begin{equation}\label{eq:Pqgn def}
\Pqgn\mB({\frac1x} = \sum_{k=1}^d (q_k-\delta_{k,1})\gamma^{2k} \frac1{x^{2k}} \mB({ \frac{1+x}{1-x} }^b + O(1) \qt{as }x\to 0\,.
\end{equation}
In other words, $\Pqgn(\frac1x)$ is the singular part in the Laurent series expansion of the above function at $x=0$.
\end{lemma}

We leave the derivation of the above formulae to the reader as it follows step by step the derivation in the special case $\qseq = g\dseq_2$ found in \cite{borot_recursive_2012}. 
The above expression of $\Rj(x)$ can be easily expanded at $x=1$, which gives the expansion \eqref{eq:rho asymp nongen} with
\begin{equation}\label{eq:nongen explicit}
	\nongen (\qgn) = - \frac{2^{2b} \Pqgn'(-1)}{2\pi \sqrt{4-n^2}} 
\qtq{and}
    \nongen'(\qgn) = - \frac{2^{-2b} \Pqgn'(1)}{2\pi \sqrt{4-n^2}} 
\end{equation}

Now we complete the characterization of the phase diagram following the same scheme as in the case $\gamma<\jmax$. 
Consider the set
\begin{equation*}
\{\nongen>0\} = \setb{(\qgn) \in \Ddomain}{ \nongen(\qgn)>0 }
\end{equation*}

\begin{lemma}
The set $\{\nongen>0\}$ is connected.
\end{lemma}
\begin{proof}
Similarly to $\gen(\qgn,\gamma)$, the polynomial $\Pqgn'$ has the following linear decomposition
\begin{equation}\label{eq:Pqgn decomp}
\Pqgn'(x) = (1-q_1)\gamma^2\, P'_1(x) - \sum_{k=2}^d q_k \gamma^{2k}\, P'_k(x) \,,
\end{equation}
where $P_k$ is the polynomial without constant term determined by
\begin{equation}\label{eq:P_k def}
	P_k\mB({ \frac1x } = \frac1{x^{2k}} \mB({ \frac{1+x}{1-x} }^b + O(1)
\end{equation}
By Newton's binomial formula, the coefficients of $P_k$ are polynomials of $b$. Thus $\nongen$ is continuous on $\Ddomain$. 

In particular, $P_1(x)=x^2+2bx$ and $-P'_1(-1) = 2(1-b)>0$. So \eqref{eq:nongen explicit} shows that $\nongen(\qgn[\oseq])>0$ for all $g>0$, $n\in(0,2)$. (Similarly to Lemma~\ref{lem:f pos}, physically this means that a fully-packed model cannot be non-generic critical and dilute.) Using this fact and the decomposition \eqref{eq:Pqgn decomp}, we see that $\{\nongen>0\}$ is connected in the same way as for $\{\gen>0\}$.
\end{proof}

\begin{proof}[Proof of Proposition~\ref{prop:pos bootstrap} (ii)]
We have seen the asymptotic expansion of $\Rj(x)$ when $\gamma=\jmax$ with explicit formulae \eqref{eq:nongen explicit} for $\nongen$ and $\nongen'$.

Next we show that $\Rj \ge 0$ on $[-1,1]$ whenever $\nongen(\qgn)\ge 0$ and $\gamma=\jmax$. We know that both $\nongen$ and $(x,\qgn,\gamma=\jmax) \mapsto \Rj(x)$ are continuous, and $\{\nongen>0\}$ is connected.
Following the proof of Lemma~\ref{lem:pos boot gen}, it suffices to show that the set $\B = \setb{(\qgn) \in \{\nongen>0\} }{ \Rj(x) \ge 0 \text{ for all }x \in [-1,1] }$ is non-empty and both closed and open in $\{\nongen>0\}$.
Thanks to Lemma~\ref{lem:f pos}, the quadruple $(\qgn[\oseq],\jmax)$ is in the closure of $\{\gen>0\}$. Thus Lemma~\ref{lem:pos boot gen} implies that $\Rqgn[\oseq]\ge 0$ on $[-1,1]$ for any $g>0$ and $n \in(0,2)$ and $\gamma=\jmax$. This shows that $\B$ is non-empty. It is also closed in $\{\nongen>0\}$ because $(x;\qgn=\jmax)\mapsto \Rj(x)$ is continuous for each $x\in[-1,1]$.
To see that $\B$ is open, consider the function $\check \rho_{\qgn}(x) := \Rj(x)/(1-x^2)^{1-b}$ where $\gamma=\jmax$. It is jointly continuous on $(x;\qgn) \in (-1,1) \times \Ddomain$ because the spectral density $\Rj(x)$ is. Using the explicit formula \eqref{eq:Rj} one can check $\check \rho_{\qgn}(x)$ extends continuously to $(x;\qgn) \in [-1,1] \times \Ddomain$. Also, Lemma~\ref{lem:rho bootstrap} remains true when $\rj$ is replaced by $\check \rho_{\qgn}$.
Then the same argument as in Lemma~\ref{lem:pos boot gen} shows that $\B$ is open in $\{\nongen>0\}$.

When $\nongen(\qgn) = 0$, since we have $\Rj\ge 0$ on $[-1,1]$, the second coefficient $\nongen'$ must be non-negative. As $\nongen$ and $\nongen'$ are proportional to $\Pqgn'(-1)$ and $\Pqgn'(1)$ respectively, it suffices to prove the following lemma to complete the proof of Proposition~\ref{prop:pos bootstrap}.
\end{proof}

\begin{lemma}[There are only two non-generic critical phases]\label{lem:only 2 non-gen}
For all $(\qgn)\in \dom$, the values $\Pqgn'(1)$ and $\Pqgn'(-1)$ both vanish only when $\qseq = \dseq_1$.
\end{lemma}

\begin{proof}

We have seen that $P_1(x)=x^2+2bx$. Thus $P'_1(-1)= -2(1-b)$ and $P'_1(1)=2(1+b)$ are both non-zero. Then the decomposition \eqref{eq:Pqgn decomp} of $\Pqgn$ shows that $\Pqgn'(1) = \Pqgn'(-1) = 0$ \Iff
\begin{equation*}
(1-q_1)\gamma^2	  =\, \sum_{k=2}^d q_k \gamma^{2k}\, \frac{P'_k( 1)}{P'_1( 1)}
				\,=\, \sum_{k=2}^d q_k \gamma^{2k}\, \frac{P'_k(-1)}{P'_1(-1)} \,,
\end{equation*}
which implies
\begin{equation}\label{eq:pos coeff linear comb}
	 \sum_{k=2}^d \mh({ \frac{P'_k( 1)}{P'_1( 1)} - \frac{P'_k(-1)}{P'_1(-1)} }
	 q_k \gamma^{2k}  \, =\, 0 \,.
\end{equation}
Since $q_k\ge 0$, to prove Lemma~\ref{lem:only 2 non-gen} it suffices to show that the coefficients in front of the $q_k \gamma^{2k}$ are all positive.

Let $H_k$ be the polynomial such that $P_k'(1)=H_k(b)$. From the definition \eqref{eq:P_k def} of $P_k$ it is easy to see that $P_k'(-1)= -H_k(-b)$. Therefore the coefficients in \eqref{eq:pos coeff linear comb} can be written as $\frac{H_k(b)}{2(1+b)} - \frac{H_k(-b)}{2(1-b)}$. It will be positive if $b\mapsto \frac{H_k(b)}{1+b}$ is strictly increasing. Below we show that $H_k(b)$ is actually an odd polynomial with positive coefficients with respect to the variable $1+b$, which will conclude the proof of the lemma.

By the definition \eqref{eq:P_k def} of $P_k$,
\begin{align*}
P_k\mB({\frac1x}
&= \frac1{x^{2k}} \frac{1-x}{1+x} \exp\mh({ (1+b) \log\mB({\frac{1+x}{1-x}} } +O(1)	\\
&= \frac1{x^{2k}} \frac{1-x}{1+x} \sum_{m=0}^{2k-1} \frac{(1+b)^m}{m!}  \mB({ \log \mB({\frac{1+x}{1-x}} }^m  +O(1)		\ ,	\\
\tq{so} - \frac1{x^{2k}} P'_k\mB({\frac1x}
&= \sum_{m=0}^{2k-1} \frac{B^m}{m!}  \od{}x\m[{ \frac1{x^{2k}} \frac{1-x}{1+x} \mB({ \log \mB({\frac{1+x}{1-x}} }^m  +O(1) }		\ .
\end{align*}
Write $H_k(b) = \sum_m \frac{h_{m,k}}{m!} (1+b)^m$ and $\m({ \log \m({\frac{1+x}{1-x}} }^m \!\!= \sum_n a_{n,m} x^n$, then the last line implies
\begin{equation*}
h_{m,k} = \sum_{l+n<2k} (2k-n-l)\, c_l\, a_{n,m}
\end{equation*}
where $(c_l)_{l\ge 0}$ is defined by $\frac{1-x}{1+x} = \sum_l c_l x^l$. Or explicitly, $c_0=1$ and $c_l = 2(-1)^l$ for all $l\ge 1$. Using the fact that
$\sum_{l=0}^K (K-l)c_l = 0$ if $K$ is even and $1$ if $K$ is odd, we get
\begin{equation*}
h_{m,k} = \sum_{n=0}^{2k-1} \mh({ \sum_{l=0}^{2k-n} (2k-n-l)c_l } a_{n,m}
= \sum_{n=1}^k a_{2n-1,m}
\end{equation*}
Since $\log \m({\frac{1+x}{1-x}} = \sum_n \frac{x^{2n+1}}{2n+1}$, the coefficient $a_{2n-1,m}$ is always non-negative, and vanishes if and only if $m$ is even. So $h_{m,k}$ is also non-negative, and vanishes if and only if $m$ is even.
\end{proof}

\bibliographystyle{hsiam} 
\bibliography{notes,aux}

\end{document}